\documentclass[11pt,reqno, oneside]{amsart}
\usepackage{etex}

\usepackage[left=1in, right=1in,top=0.9in,bottom=0.9in, marginparwidth=2.2cm]{geometry}
\usepackage{stackengine}
\usepackage{stmaryrd}
\usepackage{xcolor}
\usepackage[latin1]{inputenc}
\usepackage[T1]{fontenc}
\usepackage[normalem]{ulem}
\usepackage[english]{babel}
\usepackage{verbatim}
\usepackage{graphicx}
\usepackage{enumerate}
\usepackage{amsmath,amssymb,amsfonts,amsthm,amscd,mathrsfs}
\usepackage{mathtools}
\usepackage{array}
\usepackage{bbm}
\usepackage[mathscr]{euscript}
\usepackage{microtype}
\usepackage{dsfont}
\usepackage[T1]{fontenc}
\usepackage{babel}
\usepackage{url}
\usepackage{caption}
\usepackage{subcaption}
\usepackage[bookmarksopen, bookmarksnumbered]{hyperref}
\usepackage{tikz}
\usetikzlibrary{intersections}
\usetikzlibrary{calc}
\usetikzlibrary{decorations.markings}
\usetikzlibrary{decorations.pathreplacing}
\usetikzlibrary{arrows}
\usetikzlibrary{snakes}
\usetikzlibrary{overlay-beamer-styles}
\usetikzlibrary{external}
\usetikzlibrary{matrix}
\usetikzlibrary{shapes.misc,calc, arrows.meta, fit, hobby}
\usetikzlibrary{decorations.pathmorphing}

\usetikzlibrary{decorations.markings,fpu}
\makeatletter
\tikzset{use fpu reciprocal/.code={%
\def\pgfmathreciprocal@##1{%
    \begingroup
    \pgfkeys{/pgf/fpu=true,/pgf/fpu/output format=fixed}%
    \pgfmathparse{1/##1}%
    \pgfmath@smuggleone\pgfmathresult
    \endgroup
}}}%
\makeatother

\usepackage{cleveref}

\newtheorem{theorem}{Theorem}[section]

\newtheorem{lemma}[theorem]{Lemma}

\newtheorem{proposition}[theorem]{Proposition}
\newtheorem{corollary}[theorem]{Corollary}

\theoremstyle{definition}

\newtheorem{rem}[theorem]{Remark}
\newtheorem{remark}[theorem]{Remark}

\input xy
\xyoption{all}

\allowdisplaybreaks

\DeclareMathOperator*{\argmax}{argmax}

\DeclareMathOperator*{\emaxop}{max}
\NewDocumentCommand{\emax}{oe{_}}{%
  \IfNoValueTF{#2}
    {\emaxop\nolimits\IfValueT{#1}{^{#1}}}
    {\IfNoValueTF{#1}{\emaxop_{#2}}{\emaxcomplex{#1}{#2}}}%
}
\makeatletter
\newcommand{\emaxcomplex}[2]{\mathop{\mathpalette\emax@{{#1}{#2}}}}
\newcommand{\emax@}[2]{\emax@@{#1}#2}
\newcommand{\emax@@}[3]{
  \ifx#1\displaystyle\emax@@@{#2}{#3}\else\emaxop_{#3}^{#2}\fi
}
\newcommand{\emax@@@}[2]{%
  \begingroup\m@th
  \sbox0{$\displaystyle\emaxop$}%
  \sbox2{$\displaystyle\emaxop_{#2}$}%
  \dimen@=\dimexpr(\wd2-\wd0)/2\relax
  \sbox4{$^{#1}$}%
  \ifdim\wd4>\dimen@ \dimen@=\dimexpr\wd4-\dimen@ \else \dimen@=0pt\fi
  \operatorname*{max^{#1\kern-\wd4}}_{#2}\kern\dimen@
  \endgroup
}
\makeatother

\newcommand{\maxbeta}{\emax[(\beta)]}
\newcommand{\mrm}{\mathrm}
\newcommand{\mc}{\mathcal}
\newcommand{\mf}{\mathfrak}

\newcommand{\EE}{\mathbb E}
\newcommand{\E}{\mathbb E}
\newcommand{\PP}{\mathbb P}
\renewcommand{\P}{\mathbb P}
\newcommand{\ZZ}{\mathbb Z}
\newcommand{\Z}{\mathbb Z}
\newcommand{\RR}{\mathbb R}
\newcommand{\R}{\mathbb R}

\newcommand{\NN}{\mathbb N}
\newcommand{\N}{\mathbb N}

\newcommand{\rDe}{\mathring{\Lambda}}

\newcommand{\don}{\mathds{1}}

\newcommand{\cK}{\mathcal K}

\newcommand{\cF}{\mathcal F}
\newcommand{\cE}{\mathcal E}

\newcommand{\cL}{\mathcal L}

\newcommand{\cM}{\mathcal M}

\newcommand{\cS}{\mathcal S}

\newcommand{\cZ}{\mathcal Z}

\newcommand{\F}{\mathcal F}
\newcommand{\Fext}{\mathcal F_\mathrm{ext}}
\newcommand{\PF}{\mathbb P_{\mathcal F}}
\newcommand{\EF}{\mathbb E_{\mathcal F}}
\newcommand{\one}{\mathbbm{1}}
\newcommand{\B}{\mathcal{B}}

\newcommand{\bx}{\mathbf x}
\newcommand{\by}{\mathbf y}

\newcommand{\fh}{\mathfrak h}
\newcommand{\hfh}{\hat{\mathfrak h}}

\newcommand{\bz}{\mathbf z}

\newcommand{\midd}{\ \Big|\ }

\newcommand{\diff}{\,\mathrm{d}}
\newcommand{\h}{\mathfrak h^\beta}
\newcommand{\polyP}{\boldsymbol{\mathscr{P}}}

\newcommand{\qq}[1]{[\![{#1}]\!]}



\numberwithin{equation}{section}

\title{
\vspace*{-0.75cm}Brownian bridge limit of path measures \\ in the upper tail of KPZ models}
\author{Shirshendu Ganguly}%
\thanks{Department of Statistics, UC Berkeley, Berkeley, CA, USA. e-mail: sganguly@berkeley.edu}
\author{Milind Hegde}%
\thanks{Division of Mathematical Sciences, School of Physical and Mathematical Sciences, Nanyang Technological Uni-
versity, Singapore. e-mail: milind.hegde@ntu.edu.sg}
\author{Lingfu Zhang}%
\thanks{The Division of Physics, Mathematics and Astronomy, California Institute of Technology, Pasadena, CA, USA. e-mail: lingfuz@caltech.edu}
\date{}

\begin{document}

\begin{abstract}
For models in the KPZ universality class, such as the zero temperature model of planar last passage-percolation (LPP) and the positive temperature model of directed polymers, the upper tail behavior has been a topic of recent interest, with particular focus on the associated path measures (i.e., geodesics or polymers).
For Exponential LPP, diffusive fluctuation had been established in \cite{basu2023connecting}. In the directed landscape, the continuum limit of LPP, the limiting Gaussianity at one point, as well as of related finite-dimensional distributions of the KPZ fixed point, were established in \cite{liu2022geodesic,liu2022conditional} using exact formulas. 
It was further conjectured in these works that the limit of the corresponding geodesic should be a Brownian bridge. We prove this in both zero as well as positive temperatures; for the latter, neither the one-point limit nor the scale of fluctuations was previously known.
Instead of relying on formulas (which are still missing in the positive temperature literature), our arguments are geometric and probabilistic, using results on the shape of the weight and free energy profiles under the upper tail from \cite{GH22,GHZ25} as a starting point. Another key ingredient involves novel coalescence estimates, developed using the recently discovered shift-invariance \cite{borodin2022shift} in these models. 
Finally, our proof also yields insight into the structure of the polymer measure under the upper tail conditioning, establishing a quenched localization exponent around a random backbone.
\end{abstract}

\maketitle

\setcounter{tocdepth}{1}

\tableofcontents{}

\section{Introduction} 

Planar last passage percolation (LPP) models are paradigm examples of models of random planar geometry. In such models, one studies
the weight and geometry of the maximum weight directed path (termed as the geodesic) between two far away points in a 2D i.i.d.~random field, such as $\Z^2$ with i.i.d.~random variables at each vertex, or a homogeneous Poisson point process in $\R^2$. The geodesic weight is often also termed as the last passage time between the corresponding endpoints. A handful of such models, when the underlying noise is given by i.i.d.~Exponentials, a Poisson point process (connected to the problem of the longest increasing subsequence in a random permutation) or Brownian motions, admit exact-solvability.
Pioneered by the seminal work of Baik, Deift and Johansson \cite{BDK}, for these models exact formulae stemming from algebraic combinatorics and representation theory have been employed to rigorously deduce the predicted KPZ behavior: the weight of a geodesic between points with Euclidean separation $n$ fluctuates by $n^{1/3}$ while the geodesic itself deviates from the straight line joining the endpoints by $n^{2/3}$ (leading to the well-known $1:2:3$ scaling of KPZ). Moreover, the geodesic weight scaled by $n^{1/3}$ is known to converge to the GUE Tracy-Widom distribution. 
More recently, in a breakthrough work \cite{DOV}, a continuum object known as the directed landscape was constructed as the scaling limit of Brownian LPP, and further in \cite{DV} it was shown that all the known exactly solvable LPP models indeed converge to it under the KPZ scaling.

A particularly important topic of study has been tail estimates of the last passage time. Beyond the above mentioned Tracy-Widom limit, a line of work has emerged relying on moderate deviations estimates \cite{ledoux2010small,lowe2001moderate,lowe2002moderate}  that have proved to be key in the study of several central problems. This includes, for instance, the solution of the slow bond problem of TASEP \cite{BSS,SSZ} and the study of the correlation structure in KPZ \cite{basu2021time,basu2021temporal,corwin2021kpz}. 
An important tool featuring in many of these works is the geometry of geodesics. Some recent work focusing on the study of geodesics include \cite{basu2022nonexistence,balazs2020non} which rules out bi-infinite geodesics in LPP as well as the study of local statistics around geodesics \cite{MSZ}. An approach using geometric ideas to obtain tail estimates was also developed in \cite{ganguly2023optimal}.

Going beyond moderate deviation estimates, various attempts have been directed towards the understanding of the large deviation behavior of passage times. The study in the setting of exactly solvable LPPs appears across the papers \cite{logan1977variational,seppalainen1998coupling,seppalainen1998large,deuschel1999increasing,johansson2000shape}, using a range of methods.
More recently there has been significant progress in studying the large deviations behavior of the entire passage time profile, extending the one point results \cite{olla2019exceedingly,quastel2021hydrodynamic,GH22}. 

While the above results primarily considered the last passage time values, in another line of work the object of focus has been the distribution of the geodesic under the large deviation conditioning. The behavior of the geodesic differs starkly across the upper and lower tail regimes. This is already manifested in the disparity between the upper and lower tail behaviors of the Tracy-Widom distribution (which has upper and lower tail exponents of $\frac{3}{2}$ and $3$ respectively). This stems from the basic reasoning that the upper tail event simply requires an atypically large weighted path while the lower tail event stipulates that \emph{all} paths must have low weights, leading to a more global event and consequently a much lower probability. In the large deviation scale, for two points with Euclidean distance $n$,  the logarithm of the probability for the geodesic weight to be $\delta n$ larger (resp. smaller) than its expectation is of order $n$ (resp. $n^2$). 

Such a discrepancy between the upper and lower tails has been noted in the rigorous literature even beyond exactly-solvable models going back to Kesten \cite{kesten1986ecole} who studied the problem in the context of first passage percolation (FPP) (see also \cite{basu2021upper} for a more recent result about the existence of a rate function in this context). 

Significantly more has been established for exactly solvable models. Under upper tail large deviation (i.e., the rate $n$ tail), in \cite{deuschel1999increasing} it was shown that the transversal fluctuation for the geodesic still stays $o(n)$ for Poissonian  LPP. This was improved significantly in \cite{basu2023connecting} where the geodesic fluctuation exponent was shown to decrease from $\frac{2}{3}$ to $\frac{1}{2}$.
Under the lower tail large deviation of LPP (i.e., the rate $n^2$ tail) for LPP models for a large class of weight distributions, including and going beyond exactly-solvable models, it was shown in 
\cite{basu2019delocalization} that 
the transversal fluctuation of the geodesic becomes order $n$, i.e. the transversal fluctuation exponent equals $1$ and hence the geodesic becomes fully delocalized.

The exponents $\frac{1}{2}$ and $1$ could be attributed, at least heuristically, to the fact that the upper tail event entails local changes effected by picking a uniformly random directed path connecting the endpoints and making its weight large, whereas the lower tail event forces every path to have a small weight; this makes even paths with large transversal fluctuations competitive.

This reasoning suggests in particular that LPP geodesics in the upper tail large deviation regime should scale to a Brownian bridge. Evidence in this direction was provided in \cite{liu2022conditional}, where it was shown that in the directed landscape, if the weight between two points is conditioned to be $>L$, for a process that one expects to be close to the geodesic, as $L\to\infty,$ the multi-point distribution converge to a joint Gaussian with correlation structure matching that of a Brownian bridge; we also mention \cite{liu2022geodesic}, which proves the one-point convergence for the geodesic itself. These results led Liu and Wang to conjecture the Brownian bridge limit for the geodesic. Such a statement was also previously posed as a question in the prior work \cite{basu2023connecting} in the context of Exponential LPP.

\begin{center}{
\textit{Proving this conjecture is the main focus of this article.}
}\end{center}

We consider this problem in both the zero and the positive temperature settings.
In the former case, the pre-limiting geodesics in exactly-solvable LPP were shown in \cite{DOV} and \cite{DV} to converge to their scaling limits, which may be termed as geodesics in the directed landscape. This is the first object that we will work with.
For these processes various properties have been established, such as being $\smash{\frac{2}{3}^-}$-H\"older regular and possessing a finite $\smash{\frac{3}{2}}$-variation in \cite{DSV}. Associated local time processes were constructed and shown to be $\smash{\frac{1}{3}^-}$-H\"older regular in \cite{GZ22}. Finally, an explicit, albeit somewhat complicated, expression for the one point distribution of the geodesic was established in \cite{liu2022one} relying on exact formulae.

The positive temperature case describes the more general directed polymer models, where one still has a 2D i.i.d.~random field, but where one now considers a Gibbs measure parametrized by temperature on the space of all directed paths between two points. 
More precisely, the probability density of a path $\gamma$ is proportional to $\exp(\beta H(\gamma))$, where $\beta>0$ is the inverse temperature and $H(\gamma)$ is the weight of the path in the random field.
We henceforth refer to the random path under this measure as the polymer (between the two points).
At least formally, when $\beta\to\infty$, the polymer degenerates into the corresponding geodesic, which can therefore be viewed as the zero-temperature polymer.

Certain integrable features persist even for positive temperature models, such as the log-gamma polymer and the O'Connell-Yor polymer {\cite{MR1865759,MR2917766}}. Exploiting such special properties, KPZ behavior has also been established, at least to some degree, for these examples as well.
In particular, for the KPZ equation from the original paper of Kardar, Parisi, and Zhang \cite{KPZ}, the Cole-Hopf solution turns out to be the free energy of the continuum directed random polymer (CDRP) model \cite{alberts2014continuum}, and the KPZ scaling convergence to the directed landscape has been established in a series of recent works \cite{QS,dimitrov2021characterization,Wu21,wu2023kpz}. The polymer measure in the CDRP is the second object we work with.

The tail probabilities of these positive temperature models have also been studied. In particular, there have been extensive works on tails of the KPZ equation: its one-point upper tail probability is similar to that of the Tracy-Widom distribution, as established in \cite{khoshnevisan2017intermittency, corwin2020kpz, das2021fractional, GH22};
the lower tail probability is much more involved with a cross-over behavior, as has been shown in \cite{CG20, corwin2020kpz, cafasso2022riemann, tsai2022exact}. The behavior of the profile under large deviations has also been studied \cite{lin2023spacetime,lin2022lower,GH22}.
Besides the KPZ equation, estimates on both tails have also been obtained for the O'Connell-Yor polymer in \cite{landon2022tail}, adapting zero temperature techniques from \cite{ganguly2023optimal} to the polymer setting, as well as the log gamma model \cite{barraquand2021fluctuations,rassoul2023coalescence}.

The decision behind working with the directed landscape and the CDRP is guided by the various symmetries and scaling invariance properties that are absent in the pre-limit, which help make our arguments more transparent. 
However, we emphasize that we rely on probabilistic and geometric arguments rather than exact formulae, in contrast to \cite{liu2022conditional} and \cite{liu2022geodesic}. 
In particular, our proofs primarily rely on the line ensemble representation through the RSK and geometric RSK correspondences, and hence should be adaptable to any pre-limiting exact-solvable LPP or polymer models, 
where similar line ensembles given by different Brownian motions or random walks are available \cite{DNV}.

We now move on to the main results of this paper which prove the full Brownian bridge conjecture for both the directed landscape (zero temperature) and the continuum directed random polymer (positive temperature). Further, in the latter case, our results also establish a quenched localization phenomenon, where the polymer localizes around a random backbone, the law of the latter being a Brownian bridge.

\subsection{Main results}

We start by defining some of the key objects to help set up the framework to state our main results.
The directed landscape $\cL$, constructed in \cite{DOV}, is a continuous random function from the parameter space 
$$
\RR^4_\uparrow = \left\{(x, s; y, t) \in \RR^4 : s < t\right\}
$$
to $\RR$; $(x,s)$ and $(y,t)$ should be thought of as specifying a pair of space-time coordinates. 
It satisfies the composition law
\begin{equation}\label{e.compopsition law zero temp}
\cL(x,r;z,t)=\max_y \cL(x,r;y,s) + \cL(y,s;z,t)
\end{equation}
for any $r<s<t$, yielding the `reverse triangle inequality' $\cL(x,r;z,t) \geq \cL(x,r;y,s) + \cL(y,s;z,t)$ and thus making it a `directed metric'.

 We next describe the `directed geometry' induced by  $\cL$ and record some facts about it, following \cite[Section 12]{DOV}.
For any $s<t$, and $x, y$, a path from $(x, s)$ to $(y, t)$ is a continuous function $\pi:[s, t] \to \R$ with $\pi(s) = x$ and $\pi(t) = y$; and its length is given by
\[
\|\pi\|_\cL=\inf_{k\in \N}\inf_{s=t_0<t_1<\ldots<t_k=t}\sum_{i=1}^k\cL(\pi(t_{i-1}),t_{i-1};\pi(t_i),t_i).
\]
A path $\pi$ is a geodesic if $\|\pi\|_\cL$ is maximal among all paths with the same start and endpoints.
Equivalently, a geodesic between $(x,s)$ and $(y,t)$ is any path $\pi$ with $\|\pi\|_\cL = \cL(x,s;y,t)$. Almost surely, geodesics exist between every pair $(x, s), (y, t)$ with $s < t$. Moreover, there is almost surely a unique geodesic between any fixed pair $(x, s), (y, t)$, and we use $\pi_{(x,s;y,t)}$ to denote any such geodesic. Let $\pi_0:[0,1]\to \RR$ denote the geodesic from $(0,0)$ to $(0,1)$. 

We now arrive at our first main result concerning its limit, under the upper tail event.
\begin{theorem}  \label{thm:main-dl}
{As $L\to\infty$, $2L^{1/4}\pi_0$ conditioned on $\cL(0,0;0,1)>L$ converges to a standard Brownian bridge, weakly in the topology of uniform convergence.}
\end{theorem}

While the source of the exponent $\frac{1}{4}$ will be reviewed shortly in the upcoming idea of proofs section, momentarily we switch to our positive temperature ($\beta=1$) model of the continuum directed random polymer (CDRP).

For $(x,s;y,t) \in \R^4_{\uparrow}$, let $(x,s; y,t) \mapsto \cZ(x,s;y,t)$ be the (mild) solution to the multiplicative stochastic heat equation (SHE) defined by requiring, for all $x,s\in\R$,
\begin{equation}\label{e.SHE definition}
\begin{cases}
\partial_t \cZ(x,s;y,t) = \frac{1}{4}\partial_{yy} \cZ(x,s;y,t) + \cZ(x,s;y,t)\xi(x,s;y,t) & s<t \\
\cZ(x,s; \cdot, s) = \delta_x,
\end{cases}
\end{equation}
where $\xi$ is a space-time white noise (which is the same for all choices of initial coordinates $(x,s)$) and $\delta_x$ is the delta mass at $x$. The initial condition is understood in the weak sense, i.e., with probability one $\lim_{t\to s}\int f(y)\cZ(x,s; y, t)\, \diff y = f(x)$ for all smooth functions $f$ of compact support.

This random field was constructed in \cite[Theorem 3.1]{alberts2014continuum} and is a continuous process. (More precisely, the field constructed in \cite{alberts2014continuum} satisfies \eqref{e.SHE definition} with coefficient $\frac{1}{2}$ in place of $\frac{1}{4}$, but this is related to our solution by a simple scaling by constants; we will discuss this a bit more in Section~\ref{sec:pre}. We adopt this variant of the SHE in order to match more closely with the directed landscape.) 

The KPZ equation is the stochastic PDE
\begin{align*}
\partial_t \mc H(x,t) = \tfrac{1}{4}(\partial_{x} \mc H(x,t))^2 + \tfrac{1}{4}\partial_{xx}\mc H(x,t) + \xi(x,t)
\end{align*}
(again we have adopted coefficients of $\frac{1}{4}$ in place of the more standard $\frac{1}{2}$). It is related to the SHE by the \emph{Cole-Hopf} transform: the solution $\mc H$ to the KPZ equation is defined to be $\log \cZ$; this solution, which corresponds to delta initial conditions for the SHE, is called the \emph{narrow-wedge} solution.

A key property of $\cZ$, allowing one to define the CDRP, is the following semi-group property, which can be understood as a positive temperature analogue of the composition law \eqref{e.compopsition law zero temp}: almost surely, for any $x, y \in \RR$ and $s<r<t$,
\begin{equation}  \label{eq:cZ-comp}
   \cZ(x,s;y,t) = \int \cZ(x,s;z,r)\cZ(z,r;y,t)\diff z
\end{equation}
(see e.g.~\cite[Theorem 2.6(v)]{AJRS}).
As in \cite{alberts2014continuum}, conditional on $\cZ$, we can define the random polymer from $(x,s)$ to $(y,t)$ (for any $s<t$), as the continuous random function $\Gamma:[s,t]\to \RR$, such that for any $k\in\N$ and $s=s_0<s_1<\cdots<s_k<s_{k+1}=t$, and $x=x_0, x_1, \ldots, x_k, x_{k+1}=y$, the probability density for $\Gamma(s_1)=x_1, \ldots, \Gamma(s_k)=x_k$ is proportional to
\begin{equation}  \label{eq:gafi}
\prod_{i=0}^k \cZ(x_i,s_i; x_{i+1}, s_{i+1}).
\end{equation}
We will use $\polyP$ to denote the (quenched) measure of the polymer. Thus $\mc Z$ is the partition function of this polymer model and $\mc H$ is the free energy.

Letting $\Gamma_0$ denote the (annealed) random polymer from $(0,0)$ to $(0,1)$, the counterpart of \Cref{thm:main-dl} in the setting of the CDRP is the following.

\begin{theorem}  \label{thm:main-dp}
As $L\to\infty$, $2L^{1/4}\Gamma_0$ conditioned on $\log\cZ(0,0;0,1)>L$ converges to a standard Brownian bridge, weakly in the topology of uniform convergence.
\end{theorem}

Further, our proof will show that in this case a further exponent arises in the form of a quenched structural result. More precisely, under the upper tail event and conditional on the environment, the polymer fluctuates on a scale $L^{-1/2}$ (up to logarithmic factors), which is much smaller than $L^{-1/4}$, around a random path which we call the random backbone. The latter when scaled by $L^{-1/4},$ converges to a Brownian bridge. 
We postpone a further discussion of this to later in the article, once the relevant arguments have already been presented (see \Cref{rem:backbone}).

Our arguments involve several ingredients which we now provide an overview of. Owing to the length of the paper and the technical nature of some of the arguments, to aid the reader, we have attempted to include a fair amount of detail in the ensuing discussion.

\subsection{Idea of proofs: zero temperature} \label{iop}
In this section and the next we outline the proofs for \Cref{thm:main-dl} and \Cref{thm:main-dp}, focusing for the most part on the zero temperature setting in this section.
Further complications arise in the positive temperature setting due to the extra layer of randomness in the path measure, and we will expand on these in Section~\ref{s.intro.outline.pos temp}.

Given that our main result is a weak convergence statement, unsurprisingly, the proof broadly consists of two components: (i) tightness of the path measures and (ii) the convergence of the finite dimensional distributions to that of the Brownian bridge. We highlight two main difficulties: first, the existing upper tail estimates are not sharp enough (in positive temperature) for several of the purposes we require; second, while one-point distributions of a geodesic or the polymer measure can be described relatively straightforwardly using structures known as line ensembles, there is no clear description of multi-point joint distributions. Thus new ideas are required to address these challenges. For the first one, we instead prove and apply a new upper tail comparison estimate. For the second, we make the observation that under the upper tail, the geodesic and polymer measure at different points are approximately independent in a certain sense, due to the phenomenon of \emph{coalescence} which we elaborate on shortly.

Turning to details, we start off by introducing the starting point of many of the arguments.

\subsubsection{The tent picture}\label{s.iop.tent picture} The basic perspective of our proof relies on a result recently proved by the first two authors in \cite{GH22} on the shape of the profile $\cL(0,0; \cdot, 1)$ conditional on $\cL(0,0;0,1) > L$. First we recall that $x\mapsto\cL(0,0;x, 1) + x^2$ is stationary, so when there is no conditioning $\cL(0,0;x,1)$ fluctuates around the parabola $-x^2$. (This is a consequence of the mentioned shear invariance of $\cL$, and a fuller form of this invariance will play a crucial role in our arguments; we expand on it in the next subsection.)

In short, the result in \cite{GH22} says that the conditioned profile adopts a \emph{tent}-like shape (see Figure~\ref{fig:tent}) on the interval $[-L^{1/2}, L^{1/2}]$ (which is determined by locating the points of tangency of the tangent lines to $-x^2$ which pass through $(0,L)$). Further, the fluctuations of the profile around the tent are essentially Brownian, in particular with a fluctuation scale of $L^{1/4}$.

This description then allows to produce sharp upper tail asymptotics; while the one-point estimate in the zero temperature case that we will need was already known, \cite{GH22} also provides these estimates for the KPZ equation, using a correlation inequality from \cite{GHZ25} as a key input.

\begin{figure}[!hbt]
    \centering
\begin{tikzpicture}[scale=0.95, use fpu reciprocal,x=1cm,y=1.5cm]
\clip(-5,-1.8) rectangle (5,.8);

\pgfmathsetseed{23657}
\draw[green!70!black, thick]  plot[smooth, domain=-2.5:2.5] (\x, -0.3*\x * \x - 0.1);
\draw[blue!70!black, very thick, decorate, decoration={random steps,segment length=1.8pt,amplitude=0.8pt}]  plot[smooth, domain=-2.7:-1.5] (\x, -0.3*\x * \x);

\begin{scope}
\clip (1.5, -0.3*1.5*1.5) rectangle (2.7, -2.2);
\pgfmathsetseed{23658}
\draw[blue!70!black, very thick, decorate, decoration={random steps,segment length=1.8pt,amplitude=0.8pt}]  plot[smooth, domain=1.4:2.7] (\x, -0.3*\x * \x);
\end{scope}

\draw[blue!70!black, very thick, decorate, decoration={random steps,segment length=1.8pt,amplitude=0.8pt}]  plot[smooth, domain=-1.5:0] (\x, 2*0.3*1.5*\x + 0.3*1.5*1.5);
\draw[blue!70!black, very thick, decorate, decoration={random steps,segment length=1.8pt,amplitude=0.8pt}]  plot[smooth, domain=-1.5:0] (-\x, 2*0.3*1.5*\x + 0.3*1.5*1.5);

\newcommand{\ltan}{1.5}

\draw[blue, thick, dashed] (-1.5, -0.3*1.5*1.5) -- ++(-3.05, -2*0.3*\ltan*3.05);
\draw[blue, thick, dashed] (1.5, -0.3*1.5*1.5) -- ++(3.05, -2*0.3*\ltan*3.05);

\node[circle, fill, blue, inner sep = 1pt] at (0, -0.3*1.5*1.5 + 2*0.3*1.5*1.5) {};

\node[circle, fill, blue, inner sep = 1pt] at (1.5, -0.3*1.5*1.5) {};

\node[circle, fill, blue, inner sep = 1pt] at (-1.5, -0.3*1.5*1.5) {};

\begin{scriptsize}
\draw (0, -0.3*1.5*1.5 + 2*0.3*1.5*1.5) node[anchor=east]{$(0,L)$};
\draw (-1.5, -0.3*1.5*1.5) node[anchor=south east]{$(-L^{1/2},-L)$};
\draw (1.5, -0.3*1.5*1.5) node[anchor=south west]{$(L^{1/2},-L)$};
\end{scriptsize}

\end{tikzpicture}
    \caption{An illustration of the profile $\cL(0,0;\cdot, 1)$ conditional on it equaling $L$ at $0$, and the parabola $-x^2$ that it fluctuates around when there is no conditioning.
    }
    \label{fig:tent}
\end{figure}
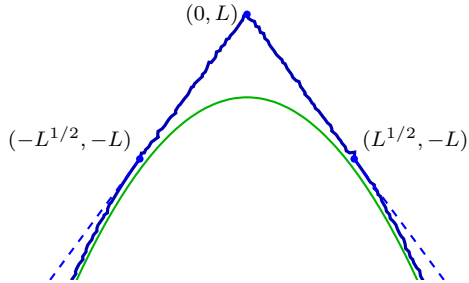

A guiding principle in many of our arguments will be that the conditioning $\cL(0,0;0,1) > L$ corresponds to the existence (at some location at height $s$ for any fixed $s$) of such a peak of height approximately $sL$.

\subsubsection{Tightness}

We prove tightness of $\pi_0$, conditioned on the upper tail at depth $L$ of the path weight/free energy and as a family of continuous random functions on $[0,1]$ indexed by $L$, using the Kolmogorov-Chentsov criterion for tightness (e.g.~\cite[Theorem 23.7]{Kallenberg}). This reduces the task to estimating the transversal fluctuation (on the $\smash{L^{-1/4}}$ scale) at two points.
More precisely, we will prove an upper bound on the probability that $|\pi_0(s)-\pi_0(t)|\smash{(t-s)^{-1/2}L^{1/4}} > M$, conditional on $\cL(0,0;0,1)>L$, that is uniform in $L$ and $0<s<t<1$. 

The strategy uses a fuller form of the earlier mentioned shear invariance property of the $\cL$ field in a crucial way. The event under consideration implies that there exist some $x$, $y$ such that 
$$|x-y|  > M(t-s)^{1/2}L^{-1/4} \quad\text{and}\quad \cL(0,0; x,s) + \cL(x,s; y,t) + \cL(y,t; 0,1)>L.$$
By doing a trivial bound on the conditional probability, we must show that the unconditional probability of this event is much smaller than the probability of $\cL(0,0;0,1)>L$.
Shear invariance of $\cL$ says that, for any fixed $z\in\R$ 
$$\cL(x,s; y,t) \stackrel{d}{=} \cL(x, s; y+z,t) + \frac{(y+z-x)^2}{t-s} - \frac{(y-x)^2}{t-s}$$
as processes in $(x,s; y,t)$ (see \Cref{ssec:dlg} below). Using this, the probability of the displayed event can be reduced to that of $\cL(0,0;0,1) > L + \frac{(M(t-s)^{1/2}L^{-1/4})^2}{t-s} = L + M^2 L^{-1/2}$. So, essentially, shear invariance yields that the geodesic moving out by an amount $\varepsilon$ between times $s$ and $t$ corresponds to a loss in weight of $\varepsilon^2/(t-s)$ for any $\varepsilon>0$ and $s<t$, which must be made up in order to achieve the upper tail conditioning.

To implement this strategy, we must obtain a relatively sharp comparison of the probabilities of $\cL(0,0;0,1) > L+\delta$ and that of $\cL(0,0;0,1) > L$ for $\delta>0$. The distribution of $\cL(0,0;0,1)$ is the same as that of the GUE Tracy-Widom distribution, for which extremely precise tail asymptotics are available using the exact formulas it is described by: indeed, for instance, it is known that $\P(\cL(0,0;0,1) > L) = (16\pi)^{-1}L^{-3/2}\exp(-\frac{4}{3}L^{3/2} + O(L^{-3/2}))$ (\cite[eq. (25)]{baik2008asymptotics}) which immediately implies that $\P(\cL(0,0;0,1) > L+\delta) = \exp(-2\delta \smash{L^{1/2}} + O(L^{-1/2}))\P(\cL(0,0;0,1) > L)$.

However, our proof must also work for the positive temperature case, and there such precise asymptotics are not available (the discussion above using shear invariance will not apply verbatim to the KPZ equation as we will explain in Section~\ref{s.intro.outline.pos temp}, but we will still need a tail comparison statement). Indeed, the sharpest upper tail asymptotics currently available for $\mc H(0,1)$ are from \cite{GH22} and state that $\P(\mc H(0,1) > L) = \exp(-\frac{4}{3}L^{3/2} + O(L^{3/4}))$; the error term is too large to see that the ratio is $\exp(-2\delta L^{1/2} + o(L^{1/2}))$. However, the methods from \cite{GH22} can be used to couple together the tent pictures coming from the events $\cL(0, 0; 0,1) > L+\delta$ and the same with $\delta=0$ to directly obtain a bound on the ratio which also holds for the positive temperature case:

\begin{theorem}\label{t.intro.comparison}
Let $L\geq 2$ and $0 < \delta < L^{1/4}$. Then
\begin{align*}
\frac{\P\left(\cL(0,0;0,1) > L + \delta\right)}{\P\left(\cL(0,0;0,1) > L\right)} = \exp\left(-2\delta L^{1/2} + O(\delta L^{-1/4}\log L)\right),
\end{align*}
and the same bound also holds in the positive temperature case of the KPZ equation.
\end{theorem}

This theorem is restated and proven in a more precise and quantified form as Theorem~\ref{t.comparison} in Section~\ref{ssec:tailcom}.

Applying this ratio estimate with $\delta = M^2L^{-1/2}$ thus yields the desired transversal fluctuation bound with a tail of $\exp(-2M^2)$ (which we note matches the upper tail asymptotic of $\frac{1}{2}(s(1-s))^{-1/2}B(s)$), which more than suffices to invoke the Kolmogorov-Chentsov criterion.

Thus in our argument, the fluctuation scale of $L^{-1/4}$ of the geodesic is obtained as a consequence of the shear invariance of $\cL$ and the precise one-point upper tail asymptotics. However, the implicit source of the exponent $\smash{\frac{1}{4}}$ is the underlying Brownian behavior of the passage time profiles. Ignoring $s$-dependent constants, from the tent picture, $\cL(0,0; \cdot,s) + \cL(\cdot, s; 0,1)$ is essentially a sum of independent rate two Brownian bridges on $[-L^{1/2}, L^{1/2}]$ which reach order $L$ above their starting points, and we are interested in the maximizer of the sum. The maximizer density at $x$ is essentially the density that $L$ will be reached at $x$ (ignoring fluctuations above $L$). Doing a Taylor expansion, we see that at $x$ (where \ the variance is $(L^{1/2}+x)(L^{1/2}-x)/L^{1/2} = L^{1/2}-x^2L^{-1/2}$), the one-point density is proportional to
\begin{align*}
\exp\left(-c\frac{L^2}{L^{1/2}-x^2L^{-1/2}}\right) = \exp\left(-cL^{3/2} -cx^2L^{1/2} + O(L^{-1/2})\right).
\end{align*}
The first term is the same for all $x$ and gets removed in the normalization, and thus we see that the distribution of the maximizer is on scale $L^{-1/4}$.

However, we do not use this intuition explicitly in our arguments. The details of our argument will be given in \Cref{sec:tight}. The argument in the positive temperature case is substantially more complicated, and we will discuss it in Section~\ref{s.intro.outline.pos temp} below.

For the moment we turn to proving the Gaussianity of the finite dimensional distributions. We proceed by setting up the primary tool we rely on in this endeavor.

\begin{figure}[b]
\includegraphics[scale=0.75]{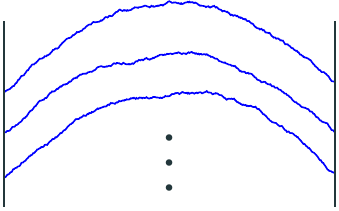}
\caption{A depiction of the parabolic Airy line ensemble.}\label{f.airy line ensemble}
\end{figure}

\subsubsection{Line ensembles} First, consider the one point distribution, say $\pi_0(s)$ for some $s\in(0,1)$.
Observe that it is precisely the argmax of $x\mapsto \cL(0,0;x,s)+\cL(x,s;0,1)$,
whereas the conditioning of $\cL(0,0;0,1)>L$ is precisely $\max_x \cL(0,0;x,s)+\cL(x,s;0,1) > L$.
The two processes $\cL(0,0;\cdot,s)$ and $\cL(\cdot,s;0,1)$
(without conditioning on the upper tail) are independent, and up to a rescaling each has the distribution of the (parabolic) Airy$_2$ process (see \cite{quastel2014airy} for a survey about it). For us an important property of the parabolic Airy$_2$ process is the earlier mentioned fact that it is equal to a stationary process minus a parabola $x^2$.

This process is also the top (or lowest indexed) line  of the Airy line ensemble, a $\N$-indexed family of continuous non-intersecting processes on $\R$ constructed by Corwin and Hammond \cite{CH11} who also showed that it admits a crucial resampling invariance property which they termed as the Brownian Gibbs property.
Essentially, this states that, for any fixed interval $[a,b]$, conditional on the second line of the Airy line ensemble and the values of the Airy$_2$ process at the boundary points of this interval, inside $[a,b]$ the distribution of the process is that of a (rate two) Brownian bridge conditioned to stay above the second line (which can be thought of as the negative parabola $-x^2$ for the purposes of this discussion, ignoring the stationary fluctuations of it around the parabola).

It is using this tool that \cite{GH22} deduces the tent description of the shape of the top curve conditioned on the upper tail mentioned in Section~\ref{s.iop.tent picture}, which we recall briefly.
Conditional on the Airy$_2$ process at $0$ equaling $h$ for some large value $h$, the process on either side of 0 behaves like two independent Brownian bridges, one each on $[-h^{1/2}, 0]$ and $[0, h^{1/2}]$, with values $-h$ (plus a random fluctuation term of lower order)  at $\pm h^{1/2}$ (see \Cref{fig:tent}). In particular, the slope of the lines around which the Brownian bridges fluctuate are respectively $\pm h^{1/2}$ to first order.
By stationarity, there is a similar picture when the Airy$_2$ process is conditioned to be large at any other point.

In the positive temperature case, the narrow wedge solution to the KPZ equation (the logarithm of \eqref{e.SHE definition}) can also be embedded as the lowest indexed curve in an $\N$-indexed family of continuous curves (though they are no longer non-intersecting) known as the KPZ line ensemble. This ensemble too has a resampling property in terms of a rate two Brownian bridge, though instead of being conditioned to not intersect the next curve, it is reweighted by an energetic penalization for intersection with the lower curve (see \eqref{e.rn derivative}).
The arguments in \cite{GH22} plus an input from \cite{GHZ25} (see \Cref{ssec:ee}) also established such a tent picture in the setting of the KPZ line ensemble, which was then  further used as a key input to prove the mentioned sharp upper tail estimates for the KPZ equation.

\subsubsection{One-point Gaussianity}
Discretizing space using a fine mesh, and considering the closest mesh point to the geodesic, leads to the following simplified form of the basic idea which drives this argument.
For any $s>0$, the conditioning that $\cL(0,0; 0,1) > L$ can be realized as a union of the events that $\cL(0,0; x,s) = h_1, \cL(x,s; 0,1) = h_2$ over a fine mesh of points $x$ and a collection of heights $h_1$ and $h_2$ (which we will often refer to as \emph{peak heights} in the discussion) such that their sum is close to $L$ and they are approximately proportionate, i.e., $h_1 \approx sL$ and $h_2 \approx (1-s)L$. Conditioning on one of the latter events where $x$, $h_1$, $h_2$ are fixed then allows us to make use of the just described tent picture. We also note that the tent picture for $\cL(0,0; \cdot, s)$ corresponds to a rescaled parabolic Airy$_2$ process such that the parabola is $-x^2/s$, and thus the tangent lines under $\cL(0,0; \cdot, s) = h_1$ will have slope $\smash{\pm s^{-1/2}h_1^{1/2}(1+o(1))}$; so if $h_1\approx sL$, the slope is to first order $\pm L^{1/2}$ independent of $s$, and similarly for the time interval $[s,1]$.

With this idea in mind, we want to use the tent picture to understand the distribution of $\argmax_x\cL(0,0;x,s)+\cL(x,s;0,1)$, conditional on $\cL(0,0;x_*,s)=h_1$ and $\cL(x_*,s;0,1)=h_2$, for given $h_1$, $h_2$ and $x_*$. Then we would like to average over $h_1$ and $h_2$ such that $h_1 + h_2 > L$ to obtain the distribution of $\pi_0(s)$. 
However, in the above strategy, the tent picture will only be useful if $h_1$ and $h_2$ are such that $\pi_0(s),$ with high probability, lands in the mesh interval adjacent to $x_*$. 
Even then, this as just described is technically very delicate; indeed, apart from obtaining the correct expression in the exponent for the density of $L^{1/4}\pi_0(s)$ as $L\to\infty$, one also needs to obtain the correct constant prefactor of $(2\pi_0 \cdot \frac{1}{4}s(1-s))^{-1/2}$ (as we are trying to prove $L^{1/4}\pi_0(s) \stackrel{\smash{d}}{\to} \frac{1}{2}B(s) = \mathcal{N}(0,\frac{1}{4}s(1-s))$). 

To obtain such sharp estimates, instead of directly obtaining probability bounds, we instead study the ratio of the probability that $L^{1/4}\pi_0(s)$ is close to $x_*$ versus that it is close to $y_*$. In this approach, we simply need to obtain the correct expression in the exponent, since the constant pre-factor then gets determined by normalization considerations. By a decomposition based on the location and height of the tent, we essentially must compare the two probabilities
\begin{align}\label{e.iop.probabilities to compare}
\P\Bigl(\pi_0(s) \approx z, \cL(0,0;0,1) > L, \cL(0,0; z, s) \approx h_1, \cL(z, s; 0,1) \approx h_2\Bigr)
\end{align}
with $z = x_*L^{-1/4}$ or $y_*L^{-1/4}$ and $h_1, h_2$ in a nice set. By shear invariance of $\cL$, it holds that
\begin{align*}
\cL(0,0; x_*L^{-1/4}, s) &\stackrel{d}{=} \cL(0,0; y_*L^{-1/4}, s) + \left(\frac{y_*^2}{s} - \frac{x_*^2}{s}\right)L^{-1/2}\\
\cL(x_*L^{-1/4}, s; 0,1) &\stackrel{d}{=} \cL(y_*L^{-1/4}, s; 0,1) + \left(\frac{y_*^2}{1-s} - \frac{x_*^2}{1-s}\right)L^{-1/2}.
\end{align*}
Thus by Theorem~\ref{t.intro.comparison} and the independence of $\cL$ across disjoint temporal strips, it holds that
\begin{align*}
\MoveEqLeft[6]
\P\left(\cL(0,0; x_*L^{-1/4}, s) \approx h_1, \cL(x_*L^{-1/4}, s; 0,1) \approx h_2\right)\\
&= \exp\left(-2\frac{y_*^2-x_*^2}{s(1-s)} + o(1)\right)\cdot\P\left(\cL(0,0; y_*L^{-1/4}, s) \approx h_1, \cL(y_*L^{-1/4}, s; 0,1)\approx h_2\right).
\end{align*}

So to complete the comparison of the probabilities in \eqref{e.iop.probabilities to compare}, by a Bayes' argument, it remains to show that
\begin{align*}
\P\Bigl(\pi_0(s) \approx z, \cL(0,0;0,1) > L \midd \cL(0,0; z, s) \approx h_1, \cL(z, s; 0,1) \approx h_2\Bigr)
\end{align*}
is the same up to a $1+o(1)$ factor for $z=x_*L^{-1/4}$ and $y_*L^{-1/4}$ for $h_i$ in a nice set. Towards this, the tent description tells us that conditional on $\cL(0,0; z, s) \approx h_1$ and if $h_1 \approx sL$, $\cL(0,0; \cdot, s)$ is essentially a pair of independent Brownian bridges with slope $\pm 2L^{1/2}$ on either side of $0$ (shifting coordinates so that the new origin corresponds to $z$).  The event $\pi_0(s) \approx z, \cL(0,0;0,1) > L$ is that the sum of these Brownian bridges has maximizer near $0$ and maximum at least $L$. Now, a Brownian bridge with slope $-L^{1/2}$ drops by order $L^{-1/2}$ in a neighborhood of order $L^{-1}$, which is the same order as its fluctuations on the interval; this means that such a Brownian bridge typically achieves its maximum in an order $L^{-1}$ neighbourhood of zero, and its maximum is greater than the value at zero by $O(L^{-1/2})$. 

The effect of $z = x_*L^{-1/4}$ vs $y_*L^{-1/4}$ is manifest by an $O(L^{1/4})$ perturbation of the endpoint value of the Brownian bridge at the tangency locations  at location $\pm \Theta(L^{1/2})$ which does not significantly affect the distribution of the bridges on the size $L^{-1}$ size interval. This essentially establishes that the comparison in \eqref{e.iop.probabilities to compare} is exactly the ratio of Gaussian densities.

To make these arguments precise is actually a somewhat delicate task, particularly due to complexities introduced in the positive temperature case (that we discuss shortly). It is done in Sections~\ref{s.fdd convergence} and~\ref{s.joint comparison of probabilities}.

\subsubsection{Multi-point: coalescence of geodesics}
We next discuss the two-point joint distribution of $\pi_0$, namely that of $\pi_0(s)$ and $\pi_0(t)$, which are the argmax of $(x,y)\mapsto \cL(0,0;x,s)+\cL(x,s; y,t)+\cL(y,t; 0,1)$.
Even without conditioning on the upper tail, in contrast to the one-point distribution of $\pi_0$ which can be described as the argmax of the sum of two independent scaled Airy$_2$ processes, the two-point joint distribution does not have such a simple or accessible description.
The reason is that there is no exact formula for the two-variable process $\cL(\cdot, s;\cdot, t)$.
However, the crucial argument we make to address this is that in the upper tail regime, the argmax problem essentially decouples, i.e., we are able to write $\cL(\cdot, s;\cdot, t)$ as a sum of two one-argument processes, using a strong coalescence phenomenon which is brought about by the upper tail conditioning that we explain next.

The key first observation is that, assuming $\cL(0,0;0,1)>L$ for a large $L$, all the geodesics from $(x,s)$ to $(y,t)$ for any $|x|, |y|$ of order smaller than $L^{1/2}$ would tend to coalesce (see the left panel of \Cref{fig:coal}); that the window of coalescence is an interval of size of order $L^{1/2}$ is closely related to the fact that the tent profile under the upper tail holds on an interval with size of the same order, as can be seen in the proof in Section~\ref{sec:coal}.
The coalescence phenomenon can be described via the following quadrangle equality (see also the right panel of Figure~\ref{fig:coal} as well as Lemma~\ref{lem:DL-quad-st} below):
with high probability it holds that
\begin{equation}  \label{eq:quantpol}
\cL(x, s; y, t) = \cL(x', s; y, t) + \cL(x, s; y',t) - \cL(x', s; y', t),
\end{equation}
for all $|x|, |y|$ and $|x'|, |y'|$ that are of order smaller than $L^{1/2}$.
Such a quantitative description of coalescence can also be generalized to the positive temperature setting of CDRP, to be discussed shortly.
Therefore, to have $\pi_0(s)=x_*$ and $\pi_0(t)=y_*$, roughly one just needs to ensure that 
\[
x_*=\argmax_x \cL(0,0; x,s) + \cL(x,s; 0,t),\quad y_*=\argmax_y \cL(0,s; y,t)+\cL(y,t; 0,1).\]
Observe that the terms $\cL(x,s; 0,t)$ and $\cL(0,s; y,t)$ are still functions of the same temporal strip $[s,t]$. As a consequence, these two terms are a priori highly dependent. This leads us to one of the key ingredients of our proof, which is to use coalescence to obtain that these processes are in fact essentially independent. Indeed, coalescence ensures that when $x$ and $y$ are varied, the changes in $\cL(x,s; 0,t)$ and $\cL(0,s; y,t)$ are local; note that the changes themselves determine the argmax. In particular, with high probability, changes in $\cL(x,s; 0,t)$ and $\cL(0,s; y,t)$ are determined by the values of $\cL$ near temporal heights $s$ and $t$, respectively, which are disjoint parts of the environment and therefore approximately independent. 
We can then estimate the ratio of the probabilities of the event $\pi_0(s)=x_*$ and $\pi_0(t)=y_*$ using multi-point versions of the one-point arguments sketched above.

However, obtaining independence by an explicit use of geodesic coalescence seems challenging, in particular in the positive temperature setting. The route we take instead to obtain the above independence crucially makes use of recently developed tools such as the multi-point passage times (which in the setting of the directed landscape is studied in \cite{DZ}) and a shift-invariance symmetry of the directed landscape and CDRP from \cite{borodin2022shift} (see also \cite{Gal,DauS,zhang2023shift}). As far as we are aware, this is the first use of the shift-invariance phenomenon to study polymer behavior.
The details are presented in \Cref{sec:coal}.

An analogous analysis can also be done for the multi-point joint distributions. This accounts for a significant part of the technical effort, particularly in the positive temperature case. The last two sections of this paper are devoted to them.

\begin{figure}[!hbt]
    \centering
\begin{tikzpicture}[scale=1.3, use fpu reciprocal]

\begin{scope}
    \pgfmathsetseed{23351}

    \draw[black, dashed] (-1, 2) -- (1, 2);
    \draw[black, dashed] (-1, 3.5) -- (1, 3.5);


    \draw[brown!75!black, thick, decorate, decoration={random steps,segment length=1pt,amplitude=0.6pt}] (0, 1) -- (0,2.3) -- (0,2.7); 

    \draw[brown!75!black, thick, decorate, decoration={random steps,segment length=1pt,amplitude=0.6pt}] (0,4.5) -- (0,3.2) -- (0,3); 

    \draw[green!70!black, thick, decorate, decoration={random steps,segment length=1pt,amplitude=0.6pt}] (0,3) -- (0,2.7); 


    \draw[blue, thick, decorate, decoration={random steps,segment length=0.8pt,amplitude=0.6pt}]  plot[smooth, domain=2:2.8] (-0.24, 2) -- (0, 2.7);

    \draw[blue, thick, decorate, decoration={random steps,segment length=1pt,amplitude=0.6pt}] (0.15, 2) -- (0, 2.3); 

    \draw[blue, thick, decorate, decoration={random steps,segment length=1pt,amplitude=0.6pt}]  plot[smooth, domain=3:3.5] (0.4, 3.5) -- (0,3);

    \draw[blue, thick, decorate, decoration={random steps,segment length=1pt,amplitude=0.4pt}]  plot[smooth, domain=3.2:3.5] (-0.12, 3.5) -- (0,3.2);

    \node[circle, fill, black, inner sep = 1pt] at (0, 1) {};
    \node[circle, fill, black, inner sep = 1pt] at (0, 4.5) {};

    \begin{scriptsize}
    \draw (0, 1) node[anchor=east]{$(0,0)$};
    \draw (0, 4.5) node[anchor=east]{$(0,1)$};

    \draw (-1, 2) node[anchor=east]{$\frac{1}{3}$};
    \draw (-1, 3.5) node[anchor=east]{$\frac{2}{3}$};
    \end{scriptsize}
\end{scope}

\begin{scope}[shift={(4.5,0)}]

\draw[dashed] (-1,1) -- (1,1);
\draw[dashed] (-1,4.5) -- (1,4.5);

\draw[blue, thick, decorate, decoration={random steps,segment length=1pt,amplitude=0.8pt}]  plot[smooth, domain=2:2.8] (0.3, 1) -- ++(0, 0.3) -- ++(-0.1, 0.3) --++ (-0.05, 0.3) --coordinate[at end](S1) ++(-0.2, 0.5);

\draw[blue, thick, decorate, decoration={random steps,segment length=1pt,amplitude=0.8pt}]  plot[smooth, domain=2:2.8] (-0.3, 1) -- ++(0.1, 0.3) -- ++(0, 0.3) --++ (0, 0.2) -- (S1);

\draw[blue, thick, decorate, decoration={random steps,segment length=1pt,amplitude=0.8pt}]  plot[smooth, domain=2:2.8] (S1) -- ++(-0.1, 0.4) -- ++(0.05, 0.4) --coordinate[at end](S2) ++(0, 0.2);

\draw[blue, thick, decorate, decoration={random steps,segment length=1pt,amplitude=0.8pt}] (S2) -- ++(-0.1, 0.3) -- ++(0,0.3) -- ++(-0.1,0.3) --coordinate[at end] (y) ++(-0.1,0.2);

\draw[blue, thick, decorate, decoration={random steps,segment length=1pt,amplitude=0.8pt}] (S2) -- ++(0.1, 0.3) -- ++(0,0.3) -- ++(0.05,0.3) --coordinate[at end] (y') ++(0.1,0.2);

\node[anchor=north, scale=0.8] at (-0.3, 0.98) {$x$};
\node[anchor=north, scale=0.8] at (0.3, 1.05) {$x'$};

\node[anchor=south, scale=0.8] at (y) {$y$};
\node[anchor=south, scale=0.8] at (y') {$y'$};

\end{scope}

\end{tikzpicture}
    \caption{Left: An illustration of the coalescence phenomenon in time $[\frac{1}{3}, \frac{2}{3}]$, under upper tail. The geodesic from $(0,0)$ to $(0,1)$ is shown in brown except for the portion common to all the paths, which is in dark green. Right: a depiction of how coalescence implies equality in the quadrangle inequality. It is clear that for fixed $x< x'$ and $y<y'$, under coalescence, the union of the geodesic from $x$ to $y$ and the geodesic from $x'$ to $y'$ equals (with multiplicity) the union of the geodesic from $x$ to $y'$ and the geodesic from $x'$ to $y$.}
    \label{fig:coal}
\end{figure}
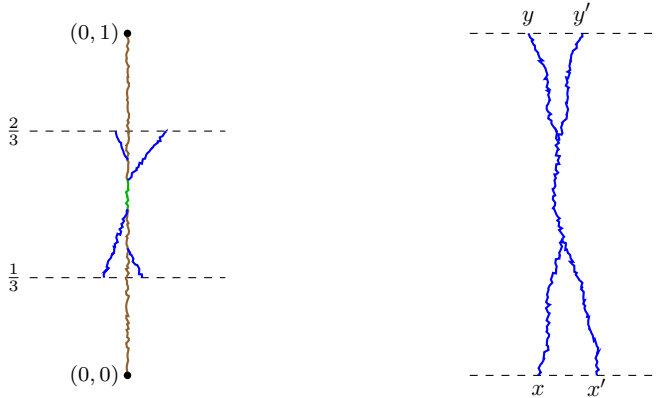

\subsection{Complications in positive temperature}\label{s.intro.outline.pos temp}
As has been mentioned several times, the positive temperature case introduces several significant complications. We explain the key differences in the proof of \Cref{thm:main-dp} (versus \Cref{thm:main-dl}) next.

We will work with the free energy landscape defined by $\cL^\beta(x,s; y,t) = \log \cZ(x,s; y,t) + (t-s)/12$. 

While it should be thought of as largely the same as the directed landscape $\cL$ in terms of properties such as shear invariance and independence in disjoint temporal strips, a significant difficulty already appears in the first step regarding tightness, which we discuss next. 

\subsubsection{Tightness and polymer concentration} The first difficulty can be seen by the fact that the zero temperature path measure, conditional on the environment and at any given height, is a delta mass. In contrast, a priori we have no non-trivial concentration of the quenched polymer measure. For this reason, though the shear invariance argument described above also applies in the positive temperature case, it only yields a bound on the transversal fluctuations of the polymer $\Gamma_0$ of order $1$, not order $L^{-1/4}$ as in the zero temperature case. 

The key new idea to obtain the correct scale $L^{-1/4}$ of fluctuations is to show that the quenched polymer measure in fact concentrates on an interval of size of lower order than $L^{-1/4}$, which is the quenched localization result mentioned after \Cref{thm:main-dp}.

More specifically, the polymer measure concentrates in an order $L^{-1/2}$ (up to a $\log L$ factor) window around a random backbone, which is $\pi(s):=\argmax_x \cL^{\beta}(0,0;x,s) + \cL^{\beta}(x,s;0,1)$ at level $s$:

\begin{proposition}
Fix $s\in(0,1)$. There exist $C$, $M_0>0$ such that, for any $L\ge 2$, $M>M_0$, 
\begin{align*}
\P\left(\polyP\left(|\Gamma_0(s)-\pi(s)| > ML^{-1/2}\log L\right) > L^{-2M} \midd \cL^{\beta}(0,0;0,1) > L\right) < C\exp(-c(\log(L))^2).
\end{align*}
\end{proposition}
The technically precise statement that we will actually prove and use is given as Proposition~\ref{p.closeness of maximizer and polymer}, where we allow $s$ to approach $0$ or $1$ in an $L$-dependent manner.

The exponent of $-\frac{1}{2}$ is due to the fact that, away from $\pi(s)$ and by the tent picture, the free energy profile $\smash{\cL^{\beta}}(0,0;\cdot,s) + \smash{\cL^{\beta}}(\cdot,s;0,1)$ decays  with a slope of order $\smash{L^{1/2}}$. Then the window of order $\smash{L^{-1/2}}$ around $\pi(s)$ is precisely where the free energy profile is within order one of its maximum, and so the polymer measure density (defined via the partition function, i.e., the exponential of the free energy) is uniformly positive there.

The complete proof of this concentration result will be given in \Cref{s.tightness.polymer measure concentration}, with preparations in the two sections before that.
This concentration of the polymer measure will allow us to upgrade the crude bounds coming from shear invariance in  \Cref{sec:tight} to tightness in \Cref{sec:pos_tight}; for technical reasons,  we only bound $|\Gamma_0(s)-\Gamma_0(t)|$ by order $(t-s)^{1/11}L^{-1/4}$ (rather than $(t-s)^{1/2}L^{-1/4}$ in the zero temperature setting).

\subsubsection{Quantitative coalescence} A second difficulty working with polymers is establishing the coalescence phenomenon, which we need to obtain the multi-point distributions. 
Indeed, unlike geodesics with different yet close by end points, which actually coalesce with high probability, the notion of coalescence for the CDRP certainly has to be qualified with an associated coupling of the corresponding Gibbs measures. For instance, independent samples of polymers with different endpoints will in fact  stay disjoint for almost all of their journey.

And indeed, while there have been a number of works studying coalescence in zero temperature (e.g.~\cite{PLPR, Zop,BSScoal,SX,BFu}), there has been only one previous studies on polymer coalescence in the positive temperature setting, namely \cite{rassoul2023coalescence}.

However, ultimately the zero temperature arguments outlined above only use coalescence by way of its effect on the weight profile, i.e., in the sense of \eqref{eq:quantpol} holding with high probability. Fortunately, this relation has the potential to be generalized more directly to positive temperature. In fact, we show that it still holds up to a error term which is exponentially small in $L$, as captured in the following statement.

\begin{proposition}
There exists $C,c$ such that for any $L>0$ and with $\delta=10^{-6}$, conditionally on $\cL^\beta(0,0;0,1) > L$, it holds with probability at least $1-C\exp(-cL^{3/2})$ that
\[
\sup_{|x|, |y| \leq \delta L^{1/2}}\left|\cL^\beta(x,0; y, 1) - \bigl(\cL^\beta(x,0; 0, 1)+\cL^\beta(0,0; y, 1)-\cL^\beta(0,0; 0, 1)\bigr)\right| < C\exp(-cL).
\]

\end{proposition}

This is proven in a more technically flexible form as Proposition~\ref{prop:dp-tent-coal}, and suffices for the remaining arguments go through. (Note however that the conditioning and the coalescence is for free energies between heights $0$ and $1$, while we will need it for intermediate temporal intervals; Proposition~\ref{prop:dp-tent-coal} allows this flexibility and we obtain the conditioning on the intermediate free energies using the decomposition into tents with proportional peak heights mentioned above.)

In proving this coalescence, we make connections to multi-point partition functions of CDRP (analogs of the mentioned above multi-point passage times, studied in e.g.~\cite{o2016multi,nica2021intermediate}), and use geometric arguments of performing surgeries on the polymers.
We note that while such surgery techniques have appeared several times in zero temperature settings (in e.g.~\cite{hammond2016brownian}), their application in the positive temperature setting seems to be novel.

\subsubsection{Multi-point Gaussianity}
The CDRP setting also faces extra complexities in computing the joint multi-point Gaussian limit.
From the concentration result, it suffices to deduce the joint multi-point Gaussian limit of the backbone $\pi$.

As in the zero temperature setting, the upper tail conditioning on $\cL^\beta(0,0;0,1)$ shall be realized as conditioning on tents of certain height and certain locations, and we then show that these peaks are close to the random backbone $\pi$ (and this is implemented in \Cref{s:gasm}).

Both tasks of changing the conditioning and deducing the closeness would require more careful treatment in the CDRP setting.
There are two basic sources of complication. The first is that entropy effects must be taken into account, since in positive temperature we must take integrals over partition function profiles instead of maximums as in the zero temperature setting. The second is that in positive temperature the tent picture obtained by decomposing into peaks at different locations is not fully accurate on certain small scales.

The complication from entropy is simple to see: in zero temperature, the peak heights of the tents were essentially such that their sum was at least $L$, but in positive temperature there is an extra correction term of order $\log L$ which must be added. This is because, by the tent picture, the slopes around the peak locations are $\pm 2L^{1/2}$, which means almost all the contribution to the integral in the composition formula \eqref{eq:cZ-comp} comes from an interval of order $L^{-1/2}$ around the peak. On taking logarithms, the small size of the interval leads to a loss of order $\log L$ which must be made up by the peak heights.

The departure from the tent picture on small scales (in fact, on scale $L^{-1/2}$) is more stark. 
Recall that the integral of the free energy profile $\smash{\cL^{\beta}}(0,0;\cdot,s) + \smash{\cL^{\beta}}(\cdot,s;0,1)$ has dominant contributions on scale $\smash{L^{-1/2}}$ around the peak location (due to the fact that that the slopes in the tent picture are of order $\smash{L^{1/2}}$). Further, as in zero temperature and from Theorem~\ref{t.intro.comparison}, the integral (i.e., the total free energy) has fluctuations of order $\smash{L^{-1/2}}$. However, though the latter fluctuation scale matches with zero temperature, the fact that the integral is on a $O(\smash{L^{-1/2}})$ window and a Brownian computation indicate that a significant contribution to the partition function instead comes from the event of the Brownian bridges atypically rising up by $O(1)$ on a $\smash{L^{-1/2}}$ window, with the peak heights being slightly lower by $O(1)$.
Therefore, conditional on the total free energy upper tail, the Brownian bridges behave extremely atypically: they exhibit behaviors of probability $\smash{\exp(-cL^{1/2})}$. Thus the tents are flat in an order $L^{-1/2}$ interval around the peak location.
This is in contrast to zero temperature where under the analogous conditioning the free energy profiles to the left and right of the peak location are with high probability independent Brownian bridges, and thus are flat only on order $L^{-1}$ intervals (where Brownian fluctuations match with slope loss), which is their typical behavior.

\subsection*{Organization of the remaining text}
In \Cref{sec:pre} we give the formal setup and quote existing tools and estimates that we will use.
\Cref{sec:coal} establishes coalescence under the upper tail, and \Cref{ssec:tailcom} proves the estimate on the ratio of upper tail probabilities.
The next four sections prove the tightness; in particular, the zero temperature tightness is done in \Cref{sec:tight}, while the positive temperature tightness is much more involved and is proved in \Cref{sec:pos_tight}, along with the localization around a random backbone.
The last four sections are devoted to proving the finite-point Gaussian limit.

\subsection*{Acknowledgement}
SG was partially supported by NSF grant DMS-1855688, NSF CAREER grant DMS-1945172, and a Sloan Research Fellowship. MH is partially supported by NSF grant DMS-1937254.
LZ is supported by the Miller Institute for Basic Research in Science, at the University of California, Berkeley, and NSF grant DMS-2246664.
The authors are grateful to Zhipeng Liu for many useful discussions and explaining the ideas in \cite{liu2022geodesic,liu2022conditional}.

\section{Preliminaries}\label{sec:pre}

\noindent\textbf{Notation.}
We start by setting up some necessary notation. To make our proofs as short as possible, we will aim to use notation which allows us to unify the positive temperature and zero temperature arguments as much as possible.

Throughout this paper, we will use $C, c>0$ to denote large and small constants, whose values may, and often will, change from line to line.
We will also use the Bachmann-Landau notations: for $A>0$, $O(A)$ denotes a number $B$ such that $|B|<CA$, and $\Omega(A)$ denotes a number $B$ such that $B>cA$.
For any  $x,y\in\RR\cup\{-\infty, \infty\}$, $x\le y$, we denote $\qq{x, y} = [x,y]\cap \ZZ$.

For any $m\in\NN$, we introduce the following simplexes,
\[
\Lambda_m=\{(x_1, \cdots, x_m): x_1\le \cdots \le x_m\},
\]
and for any $a<b$, 
\[
\Lambda_m([a,b])=\{(x_1, \cdots, x_m): a\le x_1\le \cdots \le x_m\le b\}.
\]
We use $\rDe_m$ and $\rDe_m([a,b])$ to denote the interiors of $\Lambda_m$ and $\Lambda_m([a,b])$, respectively.
For any $\bx=(x_1, \cdots, x_m) \in \Lambda_m$, let $\Delta(\bx)=\prod_{i<j}(x_i-x_j)$.

For any $\mu$ and $\sigma$, we let $\mathcal{N}(\mu, \sigma^2)$ denote the normal distribution with mean $\mu$ and variance $\sigma^2$.

For any topological space $X$, we use $\mc C(X, \R)$ to denote the space of real continuous functions on $X$ with the uniform topology.

We will occasionally use notation such as $\P(\cdot \mid X \in (K, K+ \mrm dK))$ for a random variable $X$ and real number $K$ to denote the conditional probability distribution given that $X=K$. The precise meaning of this is to consider the regular conditional distribution $\P(\cdot \mid X)$ (which exists when conditioning on real random variables by well-known abstract results such as \cite[Theorem 8.5]{Kallenberg}) and evaluate the associated probability kernel at $K$.

Similarly the notation $\P(X\in (K,K+\mrm dK))/\mrm dK$ is simply the density of the random variable $X$ with respect to Lebesgue measure, evaluated at $K$. It is well-known that the one-point distributions of objects such as $\cL$ and $\cL^\beta$ have densities, e.g. from the Brownian Gibbs properties \cite{CH11,CH14}.

We next move on to some properties of the directed landscapes and the associated geodesics therein which will appear in our arguments repeatedly. 
\subsection{The directed landscape and geodesics}  \label{ssec:dlg}
The directed landscape $\cL$, which is a random continuous function on $\RR^4_\uparrow$, is shift, shear, reflection, and $1:2:3$ scaling invariant.
More precisely, we have the following:

\begin{lemma}[Lemma 10.2, \cite{DOV}]\label{l.dl symmetries}
$\cL$ has the same distribution as
\begin{itemize}
    \item (Shift and shear) $(x,s;y,t)\mapsto \cL(x+\nu s+\alpha, s+\eta; y+\nu t + \alpha, t + \eta)+ 2\nu(y-x) + \nu^2(t-s)$, for any $\nu, \alpha, \eta\in \RR$;
    \item (Reflection) $(x,s;y,t)\mapsto \cL(-x,s;-y,t)$, and $(x,s;y,t)\mapsto \cL(y,-t;x,-s)$;
    \item (Scaling) $(x,s;y,t) \mapsto w\cL(w^{-2}x,w^{-3}s;w^{-2}y,w^{-3}t)$, for any $w>0$.
\end{itemize}
Further, for any disjoint time intervals $\{(s_i,t_i)\}_{i=1}^k$, the functions $\cL(\cdot, s_i; \cdot, t_i)$ are independent.
\end{lemma}

We next describe the multi-point passage times, studied in \cite{DZ}.
For any $k\in\NN$, $\mathbf{x}=(x_1,\cdots, x_k)\in \Lambda_k$, $\mathbf{y}=(y_1,\cdots, y_k)\in \Lambda_k$, and $s<t$, we define 
\[
\cL(\bx, s; \by, t) = \sup_{\pi_1, \dots, \pi_k} \sum_{i=1}^k \|\pi_i\|_\cL,
\]
where the supremum is over all $k$-tuples of paths $\pi = (\pi_1, \dots, \pi_k)$ where each $\pi_i$ is a path from $(x_i, s)$ to $(y_i, t)$, satisfying the disjointness condition $\pi_i(r) < \pi_j(r)$ for all $i < j$ and $r \in (s, t)$.
It is shown (in \cite[Theorem 1.7]{DZ}) that, almost surely, for every set of endpoints the supremum is achieved by some paths satisfying the disjointness condition.
Therefore the following statement holds.
\begin{lemma}[\protect{\cite[Corollary 1.11]{DZ}}]  \label{lem:quad-equa}
Almost surely the following holds. For any  $k\in\NN$, $\mathbf{x}=(x_1,\cdots, x_k), \mathbf{y}=(y_1,\cdots, y_k)\in \Lambda_k$, and $s<t$, we have
\[
\cL(\bx, s; \by, t) =\sum_{i=1}^k \cL(x_i, s; y_i, t),
\]
if and only if there exist geodesics $\pi_1, \dots, \pi_k$, where $\pi_i$ is from $(x_i, s)$ to $(y_i, t)$, satisfying $\pi_i(r) < \pi_{i+1}(r)$ for each $1\le i<k$ and $r \in (s, t)$.
\end{lemma}

For any $s<t$ and $x,y$, we also denote
\[
\cL_k(x,s;y,t)=\cL(x\mathbf{1}_k, s; y\mathbf{1}_k, t),
\]
where $\mathbf{1}_k\in \RR^k$ is the vector with each coordinate equal $1$.

A key property of the directed landscape is the following inequality due to planarity.
\begin{lemma}[\protect{\cite[Lemma 5.7]{DZ}}]  \label{lem:quad-gen}
The following holds almost surely.
    Take any $k\in \NN$ and $\bx, \by, \bx', \by'\in \Lambda^k$. Define $\bx^{\ell}, \by^{\ell}, \bx^r, \by^r\in \Lambda^k$ by setting $x^{\ell}_i=x_i\wedge x_i'$, $y^{\ell}_i=y_i\wedge y_i'$, and $x^r_i=x_i\vee x_i'$, $y^r_i=y_i\vee y_i'$, for each $1\le i \le k$.
    Take any $s<t$. Then we have
\[
\cL(\bx^{\ell}, s; \by^{\ell}, t) + \cL(\bx^r, s;\by^r, t) \ge \cL(\bx, s;\by, t) + \cL(\bx', s;\by', t).
\]
\end{lemma}
In the case of $k=1$, this is the quadrangle inequality: for any $s<t$, $x_1<x_2$, $y_1<y_2$, we have
\begin{equation}\label{eq:DL-quad}
 \cL(x_1,s;y_1,t) + \cL(x_2,s;y_2,t) \ge \cL(x_1,s;y_2,t) + \cL(x_2,s;y_1,t).   
\end{equation}
See e.g.~\cite[Lemma 9.1]{DOV}.
Besides, the strict inequality is known to be equivalent to the disjointness of geodesics.
\begin{lemma}[\protect{\cite[Lemma 3.15]{GZ22}}]  \label{lem:DL-quad-st}
For any fixed $s<t$, $x_1<x_2$, $y_1<y_2$, almost surely the inequality 
\[
 \cL(x_1,s;y_1,t) + \cL(x_2,s;y_2,t) > \cL(x_1,s;y_2,t) + \cL(x_2,s;y_1,t)
\]
is equivalent to that the geodesics $\pi_{(x_1,s;y_1,t)}$ and $\pi_{(x_2,s;y_2,t)}$ are disjoint; i.e., $\pi_{(x_1,s;y_1,t)}(r)<\pi_{(x_2,s;y_2,t)}(r)$, $\forall r\in [s,t]$.
\end{lemma}

Another degeneration of \Cref{lem:quad-gen} is the following inequality which will be used later.
For any $x$ and $y_1\le y_2 \le y_3$, and $s<t$, we have almost surely,
\begin{equation} \label{eq:mult-quad}
\cL((x,x),s;(y_1,y_3),t) + \cL(x,s;y_2,t) \ge \cL((x,x),s;(y_1,y_2),t) + \cL(x,s;y_3,t).
\end{equation}
We note that by the reflection symmetry of $\cL$, this also holds for $y_1\ge y_2\ge y_3$.
\begin{proof}[Proof of \eqref{eq:mult-quad}]
We take $z<x\wedge y_1$ and apply \Cref{lem:quad-gen} with $k=2$, $\bx=(x,x)$, $\bx'=(z,x)$, $\by=(y_1,y_2)$, $\by'=(z,y_3)$, to get
\begin{equation} \label{eq:mult-quad-p}
\cL((x,x),s;(y_1,y_3),t) + \cL((z,x),s;(z,y_2),t) \ge \cL((x,x),s;(y_1,y_2),t) + \cL((z,x),s;(z,y_3),t).
\end{equation}
We note that by \Cref{lem:quad-equa}, and the fact that any geodesic is almost surely continuous (therefore bounded), as well as from the stationarity of $\cL$ (so that the geodesic from $(z,s)$ to $(z,t)$ has the same distribution as that from $(0,s)$ to $(0,t)$, shifted by $z$), we have
\[
\lim_{z\to-\infty}\PP\bigl(\cL(x,s;y_2,t)=\cL((z,x),s;(z,y_2),t)-\cL(z,s;z,t)\bigr) = 1,
\]
\[
\lim_{z\to-\infty}\PP\bigl(\cL(x,s;y_3,t)=\cL((z,x),s;(z,y_3),t)-\cL(z,s;z,t)\bigr) = 1.
\]
Plugging these into \eqref{eq:mult-quad-p} we get \eqref{eq:mult-quad}.    
\end{proof}

\subsection{Continuum directed random polymer}  \label{ssec:CDRP}

We define the multi-line continuum partition functions through the chaos expansion, as done in \cite{o2016multi}. 
For CDRP, we take the inverse temperature $\beta=1$ throughout this paper for simplicity of notations, while our arguments go through verbatim for any fixed $\beta$.

\subsubsection{Partition function}  \label{sss:pf}
Let $W$ be a cylindrical Brownian motion on $L^2(\RR)$, and $\dot{W}$ be the space-time white noise associated with $W$.
Denote 
\[
p_t(x)=\frac{1}{\sqrt{2\pi t}}\exp(-x^2/2t).
\]
For any $s<t$, $x, y \in \RR$, and $k\in\ZZ_+$, we let $\tilde\cZ_k(x,s;y,t)$ be defined by
\[
p_{t-s}(y-x)^k
\left(
1+\sum_{m=1}^\infty \int_{\Lambda_m([s,t])}\int_{\RR^m}R((x_1,t_1),\ldots, (x_m,t_m)) W(dt_1, dx_1)\cdots W(dt_m, dx_m)
\right),
\]
where $R$ denotes the $m$ point correlation function for a collection of $k$ non-intersecting Brownian bridges which all start at
 $x$ at time $s$ and end at $y$ at time $t$.
We write $\tilde\cZ=\tilde\cZ_1$, which, historically, was introduced as the solution to the multiplicative stochastic heat equation (SHE) with Dirac delta
initial condition, via the Feynman-Kac representation.
Specifically, $\tilde u(x,t)=\tilde\cZ(0,0;x,t)$ satisfies
\[
\partial_t \tilde u = \frac{1}{2}\partial_x^2 \tilde u + \tilde u\dot{W},
\]
with $\tilde u(\cdot, 0)$ being the delta mass at $0$.

In what sense is $\tilde\cZ_k(x,s;y,t)$ defined?
In \cite{o2016multi} this is defined for any fixed $k$ and $x,s,y,t$, by proving the $L^2(W)$ convergence of the chaos expansion. In \cite{nica2021intermediate}, it is shown that \[
(y,k)\mapsto \log(\tilde\cZ_k(0,0;y,t)/\tilde\cZ_{k-1}(0,0;y,t))\] is a (scaled) KPZ$_t$ line ensemble, as defined in \cite[Theorem 2.15]{CH14}; therefore $\tilde\cZ_k(x,s;y,t)$ can be thought of as a continuous function of $y$, for any fixed $k$ and $x, s, t$ (see \cite[Corollary 1.9, 1.11]{nica2021intermediate}).
In \cite{LW}, it is further shown that $(y,t)\mapsto \tilde\cZ_k(x,s;y,t)$ can be defined as a continuous function, for any fixed $k, x, s$.

Moreover, $\tilde\cZ$ can be defined as a four-parameter random continuous function.
It is also shift, shear, and reflection invariant (in distribution), as recorded ahead in Lemma~\ref{l.Z symmetries}.

\medskip

\noindent\textbf{Scaling.} Under certain limiting transitions (either $t\to \infty$ or $\beta\to\infty$) and appropriate scaling, the logarithm of $\tilde\cZ$, which can be understood as a solution to the KPZ equation, converges to the directed landscape \cite{QS,wu2023kpz}.
While we do not actually use this convergence in this paper, in light of the scaling involved in this limit transition, and for the purposes of being consistent with the directed landscape setting and reducing notations (which will be clear shortly), we denote $\cZ_k(x,s;y,t)=2^k\tilde\cZ_k(2x,2s;2y,2t)$
and $\cZ=\cZ_1$.\footnote{The reason for the outside factor of 2 in $\cZ(x,s;y,t) = 2\tilde\cZ(2x,2s;2y,2t)$ is to ensure that indeed $\cZ(x,s; \cdot, t)\to \delta_x$ as $t\to s$ in the weak sense.}
Now $\cZ$ satisfies the SHE of \eqref{e.SHE definition}.
For the shift, shear, and reflection invariances of $\cZ$, we have the following.

\begin{lemma}[Theorem~3.1, \cite{alberts2014continuum} or Proposition 2.3, \cite{AJRS}]\label{l.Z symmetries}
$\cZ$ has the same distribution as

\begin{itemize}
    \item (Shift and shear) $(x,s;y,t)\mapsto \cZ(x+\nu s+\alpha, s+\eta; y+\nu t + \alpha, t + \eta) \exp( \nu^2(t-s) + 2\nu(y-x) )$, for any $\nu, \alpha, \eta\in \RR$;
    \item (Reflection) $(x,s;y,t)\mapsto \cZ(-x,s;-y,t)$, and $(x,s;y,t)\mapsto \cZ(y,-t;x,-s)$.
\end{itemize}
Further, for any disjoint time intervals $\{(s_i,t_i)\}_{i=1}^k$, the functions $\cZ(\cdot, s_i; \cdot, t_i)$ are independent. Also, with probability one $\cZ(x, s; y, t) > 0$ for all $x,y\in\R$ and $0<s<t$
\end{lemma}

As already mentioned in the introduction, for any $s<t$ and $x, y$, in \cite{alberts2014continuum} a measure (denoted by $\polyP$) on $\mc C([s,t],\R)$ is defined and gives the random polymer from $(x,s)$ to $(y,t)$, with finite-dimensional distribution given by \eqref{eq:gafi}.

\subsubsection{Multi-point partition function with distinct endpoints}   \label{sss:mpp}

For any $\bx=(x_1,\cdots, x_k), \by=(y_1,\cdots, y_k)\in \rDe_k$, and $s<t$, we define
\[
\cM(\bx, s;\by, t)=\det[\cZ(x_i,s;y_j,t)]_{i,j=1}^k\Delta(\bx)^{-1}\Delta(\by)^{-1}
\]
where recall $\Delta(\bx) = \prod_{i<j}(x_i-x_j)$.
Then from the continuity of $\cZ=\cZ_1$, we have that $\cM(\bx, s;\by, t)$ is almost surely continuous in all the variables.

\bigskip

\noindent\textbf{Positivity and implications.} Using the Karlin-McGregor theorem, it is straightforward to deduce that $\cM$ is non-negative, as shown in \cite[Proposition 5.5]{o2016multi}.
The simultaneous strict inequality is proved in \cite[Theorem 1.4]{LW}, and also in \cite[Theorem 2.17]{AJRS} with a different method.
\begin{lemma} \label{lem:pos}
Almost surely, for any $s<t$, $k\in\NN$, and $\bx, \by\in \rDe_k$, there is $\cM(\bx,s;\by,t)>0$.
\end{lemma}
The case of \Cref{lem:pos} where $k=2$ can be viewed as an analog to \Cref{eq:DL-quad}; namely, for any $s<t$, $x_1<x_2$, $y_1<y_2$, we have
\begin{equation}  \label{eq:FR-quad}
\cZ(x_1,s;y_1,t)\cZ(x_2,s;y_2,t) >\cZ(x_1,s;y_2,t)\cZ(x_2,s;y_1,t).    
\end{equation}
Another useful statement that can be deduced from \Cref{lem:pos} is the following monotonicity.

\begin{lemma}  \label{lem:cM-t-mon}
Almost surely the following is true. For any $s<t$, $x_1<x_2<x_3$ and $y_1<y_2<y_3$, we have
\begin{align*}
\MoveEqLeft[6]
(y_3-y_1)\cM((x_1,x_3),s;(y_1,y_3),t) \cZ(x_2,s;y_2,t)
\\
&> (y_2-y_1)\cM((x_1,x_3),s;(y_1,y_2),t)\cZ(x_2,s;y_3,t)\\
&\qquad +(y_3-y_2)\cM((x_1,x_3),s;(y_2,y_3),t)\cZ(x_2,s;y_1,t).
\end{align*}
\end{lemma}

\begin{proof}
This is equivalent to $\cM((x_1,x_2,x_3), s; (y_1,y_2,y_3), t)>0$, which holds by \Cref{lem:pos}.
\end{proof}

\noindent\textbf{Composition.}
There is also a composition law of $\cM$, which can be obtained from \eqref{eq:cZ-comp} and the Cauchy-Binet formula: almost surely, for any $k\in\NN$, $\bx, \by \in \rDe_k$ and $s<r<t$, we have
\begin{equation}  \label{eq:cM-comp}
       \cM(\bx,s;\by,t) = \int_{\rDe_k} \cM(\bx,s;\bz,r)\cM(\bz,r;\by,t) \Delta(\bz)^2 d\bz .
\end{equation}
\noindent\textbf{Continuous extension.}
The function $\cM$ is connected to the multi-layer partition function, through the following extension of $\cM$ to the boundary of $\Lambda_k\times \Lambda_k$.
\begin{lemma}[\protect{\cite[Lemma 6.1]{o2016multi}}]  \label{lem:ext}
For any $s<t$ and $k\in\NN$, the function $\bx, \by\mapsto \cM(\bx,s;\by,t)$ extends continuously in $L^2(W)$ to $\Lambda_k\times \Lambda_k$, and the extension satisfies
\[
2^{-k(k-1)/2}(t-s)^{k(k-1)/2}\prod_{i=1}^{k-1}i!\cM(x\mathbf{1},s;y\mathbf{1},t)=\cZ_k(x,s;y,t),
\]
where $\mathbf{1}$ is the vector in $\R^k$ where each entry equals $1$. 
\end{lemma}

\subsection{Line ensembles and Gibbs properties}\label{ss:legp}

As already indicated, a tool widely used in the study of the directed landscape is the Airy line ensemble constructed in \cite{CH11}.
Analogously, 
a tool widely used in the study of the KPZ equation and the free energy of the CDRP is the related KPZ$_t$ line ensemble from \cite{CH14}, and its Gibbs property.
To be concise, instead of giving the complete line ensemble setups, we quote some useful results from these connections.

For any $t>0$, $x\in\RR$, we denote 
\[
\fh^{\beta=1}_{t,1}(x)=\log\cZ(0,0;x,t)+t/12,\quad \fh^{\beta=1}_{t,2}(x)=\log(\cZ_2(0,0;x,t)/\cZ(0,0;x,t))+t/12,
\]
and
\[
\fh^{\beta=\infty}_{t,1}(x) = \cL(0,0;x,t), \quad \fh^{\beta=\infty}_{t,2}(x) = \cL_2(0,0;x,t) - \cL(0,0;x,t).
\]
We note that $\mc H(x,t)= \fh^{\beta=1}_{t,1}(x)$ solves the KPZ equation
\begin{align*}
\partial_t \mc H = \tfrac{1}{4}\partial_x^2 \mc H + \tfrac{1}{4}(\partial_x \mc H)^2 + \dot{W}.
\end{align*}
(Recall that $\dot{W}$ is the space-time white noise.)

Thanks to the scaling in defining $\cZ$, $\smash{\fh^{\beta=1}_{t,1}}$ has the same parabolic decay of $-x^2/t$ as $\smash{\fh^{\beta=\infty}_{t,1}}$. More precisely, $\smash{\fh^{\beta}_{t,1}(x)}+x^2/t$ for both $\beta=1$ and $\infty$ are stationary (which can be deduced from the shear invariance of $\cZ$ and $\cL$).
We also denote
\[
 \hat{\fh}^\beta_{t,1}(x):=t^{-1/3}\fh^\beta_{t,1}(t^{2/3}x), \quad \hat{\fh}^\beta_{t,2}(x):=t^{-1/3}\fh^\beta_{t,2}(t^{2/3}x).
\]
As will be clear later on, using $\smash{\hat{\fh}^\beta_{t,1}}$ and $\smash{\hat{\fh}^\beta_{t,2}}$ instead of $\smash{{\fh}^\beta_{t,1}}$ and $\smash{{\fh}^\beta_{t,2}}$ will reduce some notation, since $\smash{\hat{\fh}^\beta_{t,1}}$ and $\smash{\hat{\fh}^\beta_{t,2}}$ have the parabolic decay of $-x^2$ independent of $t$. Further, in the $\beta=\infty$ case, $\hfh^\beta_t$ has distribution independent of $t$, which is a consequence of scaling properties of the directed landscape.

Another consequence of the scaling in defining $\cZ$ is that  $\smash{\fh^{\beta=1}_{t,1}}$ is locally absolutely continuous with respect to a rate $2$ Brownian bridge, which is also true for $\smash{\fh^{\beta=\infty}_{t,1}}$.
To be more precise, we quote the following Gibbs properties of $\fh^\beta_{t,1}$ given $\fh^\beta_{t,2}$. 
For any $a<b$, denote by $\Fext([a, b])$ the $\sigma$-algebra generated by $\smash{\fh^\beta_{t,1}}$ in $\R\setminus (a, b)$, and $\smash{\fh^\beta_{t,2}}$.
\begin{lemma} \label{lem:Gibbs}
Take any $a<b$ and $t>0$.
Conditional on $\Fext([a, b])$,
for (1) law of $\fh^\beta_{t,1}$ in $[a,b]$, (2) the rate $2$ Brownian bridge connecting  $\smash{\fh^\beta_{t,1}(a)}$ and $\smash{\fh^\beta_{t,1}(b)}$, the former is absolutely continuous with respect to the latter, with Radon-Nikodym derivative (for a path $B$) proportional to $W(B, \fh^\beta_{t,2})$, where
\begin{equation}\label{e.rn derivative}
W(f, g)=
\begin{cases}
\exp\Big( - 2\int_a^b \exp(f(x)-g(x))  dx \Big) & \text{ when } \beta=1, \\ 
\don[g(x)\le f(x),\forall x\in [a,b]]  & \text{ when } \beta=\infty.
\end{cases}
\end{equation}
\end{lemma}
For the case where $\beta=\infty$ (the Airy line ensemble setting), this was established in \cite{CH11}.
For $\beta=1$ (the KPZ$_t$ line ensemble setting), such a Gibbs property was first introduced in \cite{CH14}; and the connection between the KPZ$_t$ line ensemble and CDRP was formally established in \cite{nica2021intermediate}.
The form of the Gibbs property presented here is from \cite[Proposition 2.6, Theorem 2.7]{GH22}.

A useful consequence of the Gibbs property is the monotonicity recorded below.
\begin{lemma}[Monotonicity in boundary data]\label{l.monotonicity}
Fix $a<b$, real numbers $\smash{w^{*}, z^{*}}\in \R$ and measurable functions $\smash{g^{*}}:[a,b]\to\R \cup \{-\infty\}$ for $*\in \{\uparrow, \downarrow\}$ such that $\smash{w^{\downarrow}\leq w^{\uparrow}}$, $\smash{z^{\downarrow}\leq z^{\uparrow}}$, and, for all $s\in (a,b)$, $\smash{g^{\downarrow}(s)\leq g^{\uparrow}(s)}$. 

For $*\in \{\uparrow, \downarrow\}$, let $\smash{\mathcal{Q}^{*}}$ be a process on $[a,b]$ such that $\smash{\mathcal{Q}^{*}}(a)=\smash{w^{*}}$ and $\smash{\mathcal{Q}^{*}}(b)=\smash{z^{*}}$, with Radon-Nikodym derivative with respect to Brownian bridge given by $W(\mathcal{Q}^{*}, g)$ for $\beta=1$.
Then there exists a coupling of the laws of $\smash{\mathcal{Q}^{\uparrow}}$ and $\smash{\mathcal{Q}^{\downarrow}}$ such that almost surely $\smash{\mathcal{Q}^{\downarrow}}(s)\leq \smash{\mathcal{Q}^{\uparrow}_j(s)}$ for all $s\in (a,b)$.

The same is true in the $\beta=\infty$ (zero-temperature) case if additionally $w^* > \smash{g^{*}(a)}$ and $z^* > \smash{g^{*}(b)}$, for $*\in\{\uparrow\,\downarrow\}$.
\end{lemma}

The positive temperature ($\beta=1$) statements are Lemmas~2.6 and 2.7 of \cite{CH14}. The zero temperature ($\beta=\infty$) statements are Lemmas~2.6 and 2.7 of \cite{CH11}. See also \cite{dimitrov2021characterization} and \cite{dimitrov2022characterization} for more detailed proofs of the respective cases.

The following inequality can be deduced using the monotonicity property.
It is contained in \cite[Theorem 2.7]{GH22}, and its proof can be found in \cite[Appendix A]{GH22}.

\begin{lemma}\label{l.fkg}
For any $t>0$, $a<b$, $\beta=1$ or $\infty$, and any pair of increasing events $A$ and $B$ in the space of all real continuous functions on $[a, b]$, 
$$\P(\fh^\beta_{t,1}|_{[a,b]} \in A,\, \fh^\beta_{t,1}|_{[a,b]} \in B) \geq \P(\fh^\beta_{t,1}|_{[a,b]} \in A)\cdot \P(\fh^\beta_{t,1}|_{[a,b]} \in B).$$
\end{lemma}
This is the FKG inequality for $\fh^\beta_{t,1}|_{[a,b]}$; and for discrete models such as the exponential LPP, such a result follows from the classical FKG inequality.

\subsection{Line ensemble estimates}  \label{ssec:ee}
In this subsection, we list some estimates of the line ensembles, which hold in both zero and positive temperature settings. 
For the convenience of notations, we state them in terms of the scaled version $\hfh^\beta_{t,1}$ and $\hfh^\beta_{t,2}$.
We fix arbitrary $\varepsilon>0$ in this subsection, and therefore all constants may depend on $\varepsilon$.

We also note that most of the estimates include a condition on $L$ in terms of $t$. This is really necessary only for the $\beta=1$ case since, as mentioned earlier, $\hfh^{\beta=\infty}_t$ has distribution independent of $t$. But, again to make things slightly more streamlined, we impose the same conditions in both cases.

We start with the one-point upper-tail of $\hfh^\beta_{t,1}$.

\begin{theorem} \label{lem:fh-ut}
There exist $C>0$ and $L_0>0$ such that, for any $t>0$ and $L>(t^{-1/3-\varepsilon}\vee 1)L_0$,
\[
\exp\Big(-\tfrac{4}{3}L^{3/2}-CL^{3/4}\Big)<\PP\left(\hfh^\beta_{t,1}(0)\in (L, L+\diff L)\right)/\diff L<\exp\Big(-\tfrac{4}{3}L^{3/2}+CL^{3/4}\Big).
\]
\end{theorem}

We can also upper bound $\hfh^\beta_{t,2}$ conditional on $\hfh^\beta_{t,1}$ at one-point.

\begin{theorem}\label{l.bk}
There exist $C>0$, $c>0$, and $L_0>0$ such that, for any $t>0$ and $L>M>(t^{-1/3-\varepsilon}\vee 1)L_0$, and any $a\in \R$,
\begin{align*}
\MoveEqLeft[20]
\P\bigg( \sup_{x\in [a,a+1]}\hfh^\beta_{t,2}(x)+x^2> M+Ct^{-1/3}\log(L)\midd \hfh^\beta_{t,1}(0)\in (L, L +\diff L) \bigg)\\%
&\leq \exp\Big(-\tfrac{4}{3}M^{3/2} + CM^{3/4}\Big) + |a|t^{2/3}\exp(-cL^2).
\end{align*}
And the same is true under the conditioning $\hfh^\beta_{t,1}(0)>L$.
\end{theorem}
The above two results are proved in or quickly deduced from \cite{GH22,GHZ25}. First, \Cref{lem:fh-ut} is proved as \cite[Theorem 1]{GH22} using an input from our paper \cite{GHZ25} to verify an assumption, though, in the case of $\beta=1$, \cite[Theorem 1]{GH22} is stated for $t>t_0$ for some $t_0>0$. 
The reason for the $t>t_0$ condition is because the arguments in \cite{GH22} also take as input a priori tail estimates from \cite{corwin2020kpz}, which hold for $t>t_0$ and $L>L_0$ with $L_0$ possibly depending on $t_0$.
However, analogous estimates are also available for arbitrary $t>0$ from \cite{das2023law}, namely Theorems~1.4 (upper bound on upper tail) and 1.7 (upper bound on lower tail) there.
Using these inputs the tail estimate in \cite{GH22} for $\beta=1$ can be upgraded to cover small $t>0$.
More details will be given in \Cref{sec:appc}. Theorem~\ref{l.bk} is also proved straightforwardly using results of \cite{GH22,GHZ25}, and details are also given in Appendix~\ref{sec:appc}.

We can further deduce the following upper-tail bound for one-point distribution of $\hfh^\beta_{t,2}$, which will be useful later.

\begin{lemma}  \label{lem:fh-2ut}
There exist $C, L_0,$ and $\varepsilon'>0$ such that, for any $t>0$ and $L>(t^{-1/3-\varepsilon}\vee 1)L_0$,
\[
\PP\left(\hfh^\beta_{t,1}(0)>L, \hfh^\beta_{t,1}(0)+\hfh^\beta_{t,2}(0)>2L\right)<\exp\Big(-\tfrac{8}{3}L^{3/2}+CL^{1-\varepsilon'}\Big).
\]
\end{lemma}

\begin{proof}
We can write the left-hand side as
\[
\int_{L}^\infty \PP\left(\hfh^\beta_{t,2}(0)>2L -\vartheta \mid \hfh^\beta_{t,1}(0)\in (\vartheta, \vartheta +\diff\vartheta)\right) \PP\left(\hfh^\beta_{t,1}(0)\in (\vartheta, \vartheta +\diff\vartheta)\right) \, \diff\vartheta.
\]
We upper bound the first factor in the integrand using \Cref{l.bk} (if $2L-\vartheta>Ct^{-1/3}\log(L)$) or by $1$ (otherwise), and the second factor in the integrand using \Cref{lem:fh-ut}. It is easy to see that the dominant contribution to the integral whose integrand is the product of these two factors is when $\vartheta$ is approximately $L-Ct^{-1/3}\log L$. Now, note that since $L > t^{-1/3-\varepsilon}$, $t^{-1/3} \log L \leq L^{1-\varepsilon'}$ for some $\varepsilon'>0$, and $(L-L^{1-\varepsilon'})^{3/2} \geq L^{3/2} - CL^{1-\varepsilon'}$. Then the conclusion follows.
\end{proof}

We have the following estimate on the one-point lower tail of $\hfh^\beta_{t,1}$.
\begin{theorem}  \label{lem:fh-lt}\label{l.lower tail}
There exists $L_0>0$ such that, for any $0<t\le 1$ and $L>(t^{-1/6}\vee 1)L_0$,
\[
\P\left(\hfh^\beta_{t,1}(0) < -L\right) \leq \exp(-cL^2t^{1/6}).
\]
If we instead assume $t>t_0$ for some $t_0>0$, and $L>L_0$, then
\[
\PP\left(\hfh^\beta_{t,1}(0)<-L\right)<\exp(-cL^{5/2}),
\]
with the constant $c$ depending on $t_0$.
\end{theorem}
For $\beta=1$ these two estimates can be deduced from  \cite[Theorem 1.7]{das2023law} and \cite[Theorem 1]{CG20} respectively. 
For $\beta=\infty$, $\hfh^{\beta=\infty}_{t,1}(0)$ has GUE Tracy-Widom distribution, whose lower tail is more classical (see e.g., \cite{TW94} or \cite[Proposition 5.1]{CG20}).

The process $\hfh^\beta_{t,1}$ is also $1/2$-H\"older.
\begin{proposition}  \label{lem:fh-cont}
There exists $L_0>0$ such that, for any $t>0$ and $M^2>(t^{-1/6}\vee 1)L_0$,
and $0<d\le 1$, we have
\[
\PP\Bigl(\sup_{x\in[0,d]} |\hfh^\beta_{t,1}(x)-\hfh^\beta_{t,1}(0)|>Md^{1/2} \Bigr) < C\exp(-cM^{2}).
\]
\end{proposition}

For $\beta=\infty$ this is proved in \cite[Lemma 6.1]{DV21} and \cite[Lemma 3.4]{DauWie}, and their arguments carry over to the $\beta=1$ setting. We omit the details here.

We next quote the following tent behavior of $\hfh^\beta_{t,1}$, under the one-point upper-tail event.
\begin{theorem}  \label{lem:fh-tent}
There exists $L_0>0$, such that for any $t>0$ and $L> (t^{-1/3-\varepsilon}\vee 1)L_0$, we have
\[
\PP\left(\sup_{x\in[-L^{1/2}, L^{1/2}]}\left|\hfh^\beta_{t,1}(x)-L+2L^{1/2}|x|\right| > ML^{1/4} \midd \hfh^\beta_{t,1}(0)\in (L, L+\diff L) \right)<\exp(-cM^2),
\]
for any $0<M<cL^{3/4}$. 
\end{theorem}
This follows from \cite{GH22,GHZ25} in a way similar to \Cref{lem:fh-ut}: under the same set of assumptions on line ensembles this is proved in \cite[Theorem 6]{GH22};
then by \cite[Theorem 2.7 and Proposition 1.3]{GH22} (which uses \cite{GHZ25} as an input), the assumptions are verified in both zero and positive temperature settings. This is for $t>t_0$ for any $t_0>0$ fixed, and the extension to $t>0$ under the conditions mentioned in the statement above is done in Appendix~\ref{sec:appc} along with the same for Theorem~\ref{lem:fh-ut}.

A version of this tent behavior with the conditioning replaced by $\hfh^\beta_{t,1}(0) > L$ also holds, and we will only need that for $t$ bounded away from zero. 
\begin{lemma} \label{lem:fh-tent-up}
For any $t_0>0$, there exists $L_0>0$, such that for any $t>t_0$ and $L> L_0$, we have
\[
\PP\left(\sup_{x\in[-L^{1/2}, L^{1/2}]}\left|\hfh^\beta_{t,1}(x)-L+2L^{1/2}|x|\right| > ML^{1/4} \midd \hfh^\beta_{t,1}(0)>L \right)<\exp(-cM^2),
\]
for any $0<M<cL^{3/4}$. Here the constant $c$ can depend on $t_0$.
\end{lemma}
This lemma will help with the presentation of the proof of the upcoming crucial \Cref{t.comparison}.
However, while proving this in \Cref{sec:tnp}, we will show that in fact one does not need to rely on it for the proof of \Cref{t.comparison}.

We next give a more precise comparison of the tent with Brownian bridges.
\begin{lemma}  \label{l.dp-comp-s}
For $t$ and $L$ as in \Cref{lem:fh-tent}, the following is true.
Consider the following processes, each defined on $[0,L^{1/2}/2]$,
\[
x\mapsto \hfh^\beta_{t,1}(x)-L, \quad x\mapsto \hfh^\beta_{t,1}(-x)-L,\]
conditional on $\hfh^\beta_{t,1}(0)\in (L,L+\diff L)$. Also, consider the processes (each defined on $[0,L^{1/2}/2]$ as well)
\[x\mapsto B_1(x)+2xL^{-1/2}(\hfh^\beta_{t,1}(L^{1/2}/2)-L),\quad x\mapsto B_2(x)+2xL^{-1/2}(\hfh^\beta_{t,1}(-L^{1/2}/2)-L,\]
also conditional on $\hfh^\beta_{t,1}(0)\in (L,L+\diff L)$, where $B_1, B_2$ are rate $2$ Brownian bridges in $[0, L^{1/2}/2]$, independent of each other and independent of $\hfh^\beta_{t,1}$.
They can be coupled so that under an event with probability $>1-C\exp(-cL^{3/2})$ for both, their Radon-Nikodym derivative is $1+O(\exp(-ct^{1/3}L))$.
\end{lemma}
The proof of this as well as the following lemma will be given in \Cref{sec:tnp}.
\begin{lemma}\label{l.control near tent center}
For $t$ and $L$ as in \Cref{lem:fh-tent}, any $I\subseteq [-\frac{1}{2}L^{1/2}, \frac{1}{2}L^{1/2}]$ and $\sigma_I = \sup_{x\in I} |x|^{1/2}$, and $0 < M< L^{3/4}$,
$$\P\left(\sup_{x\in I}|\hfh^\beta_{t,1}(x) - (L - 2L^{1/2}|x|)| > M\sigma_I\midd \hfh^\beta_{t,1}(0)\in (L,L+\diff L)\right) < C\exp(-cM^2).$$
If we in addition assume $t>t_0$ for some $t_0>0$, then we have
$$\P\left(\sup_{x\in I}|\hfh^\beta_{t,1}(x) - (L - 2L^{1/2}|x|)| > M\sigma_I\midd \hfh^\beta_{t,1}(0)> L\right) < C\exp\left(-c(M^2 \wedge M\sigma_IL^{1/2})\right),$$
with the constant $c$ depending on $t_0$.
\end{lemma}
The estimates above quickly lead to the following.
\begin{corollary}  \label{cor.tent-pr}
For $t$ and $L$ as in \Cref{lem:fh-tent}, for any $0<M<L^{3/4}$, and any $a>0$, we have
\begin{multline*}
 \P\left(\sup_{|x|\le L^{1/2}}\left|\hfh^\beta_{t,1}(x) - (L - 2L^{1/2}|x|)\right| (|\log(|x|/a)|+1)^{-1} |x|^{-1/2} > M\midd \hfh^\beta_{t,1}(0)\in (L,L+\diff L)\right) \\ < C\exp(-cM^2).
\end{multline*}
\end{corollary}
This estimate is obtained via a union bound over dyadic scales, using \Cref{lem:fh-tent} or \Cref{l.control near tent center} at each scale.
The parameter $a$ represents the scale of $x$ at which the logarithm becomes of constant order, which will provide some convenient flexibility in applications.

\subsection{Gaussian estimate}\label{s.tools.brownian}
Here we recall a standard bound on the tail of centered Gaussian random variables. 

\begin{lemma}\label{l.normal bounds}
For $\sigma>0$ and $x> 0$,
$$\frac{1}{\sqrt{2\pi}}\cdot\frac{\sigma}{x}\left(1-\frac{\sigma^2}{x^2}\right)\exp\left(-\frac{x^2}{2\sigma^2}\right) \leq \P\left(\mathcal{N}(0,\sigma^2) \geq x\right) \leq \frac{1}{\sqrt{2\pi}}\cdot\frac{\sigma}{x}\exp\left(-\frac{x^2}{2\sigma^2}\right).$$
\end{lemma}

\begin{proof}
We set $\sigma=1$ without loss of generality. Now we write $\P\left(\mathcal{N}(0,\sigma^2) \geq x\right)$ as an integral of the normal density and obtain the claimed bounds by doing integration by parts one for the upper bound and again for the lower bound.
\end{proof}

\section{Coalescence and Brownian bridge comparison under upper tail}  \label{sec:coal}
This section develops the key coalescence estimates that our analysis relies on.
It will be convenient to denote (within this section)
\[
\cS^{\beta=1}_t(x,y)=t^{-1/3}(\log\cZ(t^{2/3}x,0;t^{2/3}y,t)+t/12),
\]
for any $t>0$.
One may interpret this as a `rescaled KPZ sheet'. The zero temperature analog, that is the Airy sheet, will be denoted by
\[
\cS^{\beta=\infty}_t(x,y)=t^{-1/3}\cL(t^{2/3}x,0;t^{2/3}y,1).
\] Note that actually the law of $\cS^{\beta=\infty}_t$ is the same for any $t>0$.

From these definitions, we see that for any $t>0$ and $\beta=1, \infty$, we have $\cS^\beta_t(0,x)=\hat{\fh}^\beta_{t,1}(x)=t^{-1/3}\fh^\beta_{t,1}(t^{2/3}x)$. 
By the shear, shift, and reflection invariance properties of $\cL$ and $\cZ$ (introduced in \Cref{ssec:dlg} and \Cref{sss:pf}), we have that $\cS^\beta_t$ has the same law as
\[(x,y)\mapsto\cS^\beta_t(x+a,y+b)+(x+a-y-b)^2-(x-y)^2, \quad (x,y)\mapsto \cS^\beta_t(-x,-y), \quad (x,y)\mapsto \cS^\beta_t(y,x).\]

We next deduce two uniform bounds of $\cS^\beta_t$, which are for both $\beta=1,\infty$ and any $t>0$.
The first is an H\"{o}lder estimate. 
\begin{lemma}  \label{lem:fh-int-ub}
There exists $L_0>0$ such that, for any $t>0$ and $M^2>(t^{-1/6}\vee 1)L_0$,
and $0<d\le 1$, we have
\[
\PP\left(\sup_{x,y\in[0,d]} |\cS^\beta_t(x,y)-\cS^\beta_t(0,0)|>Md^{1/2}\right) < C\exp(-cM^2).
\]
\end{lemma}
\begin{proof}
Using \eqref{eq:DL-quad} or \eqref{eq:FR-quad}, for any $x, y\in [0,d]$, we have
\[
-\cS^\beta_t(0,0) + \cS^\beta_t(0,y) + \cS^\beta_t(x,0) \le \cS^\beta_t(x,y) \le -\cS^\beta_t(d,0) + \cS^\beta_t(d,y) + \cS^\beta_t(x,0).
\]
By \Cref{lem:fh-cont}, and symmetries of $\cS^\beta_t$, we have 
\[
\PP\left(\sup_{x\in [0,d]} |\cS^\beta_t(0,x)-\cS^\beta_t(0,0)| > Md^{1/2}\right), \PP\left(\sup_{x\in [0,d]} |\cS^\beta_t(x,0)-\cS^\beta_t(0,0)| > Md^{1/2}\right) < C\exp(-cM^2), 
\]
and
\[
\PP\left(\sup_{x\in [0,d]} |\cS^\beta_t(d,x)-\cS^\beta_t(d,0)| > Md^{1/2}\right) < C\exp(-cM^2);
\]
therefore the conclusion holds.
\end{proof}
\begin{lemma}  \label{lem:fh-ub}
Fix $\varepsilon>0$. There exist $M_0>0$ and a random variable $H>0$, such that $\PP(H>M)<C\exp(-cM^{3/2})$ for any $M>(t^{-1/3-\varepsilon}\vee 1)M_0$, and
\[
|\cS^\beta_t(x,y)+(x-y)^2| < H + \log(|x|+|y|+2), \quad \forall x, y\in \RR.
\]
\end{lemma}
\begin{proof}
It suffices to show that, for any $M>0$,
\begin{equation} \label{eq:fh-iub}
\PP\left(\sup_{x,y\in [0,1]} |\cS^\beta_t(x,y)+(x-y)^2|>M\right) < C\exp(-cM^{3/2}).
\end{equation}
Then the conclusion follows by splitting $\RR$ into intervals of length $1$, using the shear and shift invariance properties of $\cS^\beta_t$, and taking a union bound.

As for \eqref{eq:fh-iub}, we just apply \Cref{lem:fh-int-ub} for $d=1$, and use that $\PP(|\cS^\beta_t(0,0)|>M)<C\exp(-cM^{3/2})$, which can be obtained from \Cref{lem:fh-ut,lem:fh-lt}.
\end{proof}

\begin{figure}[!hbt]
    \centering
\includegraphics{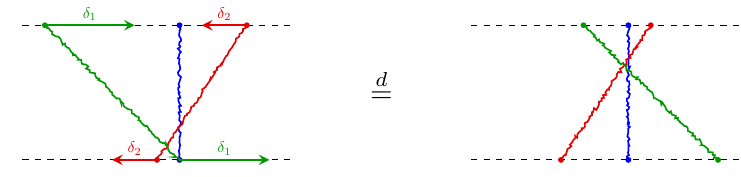}
    \caption{An illustration of the shift-invariance: the joint distributions of the three passage times/partition functions in the left and right panels are the same, under the condition that both endpoints of each path are shifted by the same amount, and the endpoints and shifts are such that the paths are all forced  by planarity to intersect both before and after the shift.}
    \label{fig:sinv}
\end{figure}

The following remarkable shift-invariance property (illustrated in \Cref{fig:sinv}) proved in \cite{borodin2022shift} will also be a key input. It follows from \cite[Theorems 7.8, 7.10]{borodin2022shift} immediately. 
\begin{lemma}  \label{lem:fr-shift}
Take any $m\in\NN$, and $x_1\le \cdots \le x_m$, $y_1\ge \cdots \ge y_m$, and $x_1'\le \cdots \le x_m'$, $y_1'\ge \cdots \ge y_m'$, such that $x_i-y_i=x_i'-y_i'$ for any $i\in\qq{1, m}$. Then $\{\cS^\beta_t(x_i,y_i)\}_{i=1}^m$ and $\{\cS^\beta_t(x_i',y_i')\}_{i=1}^m$ are equal in distribution.
\end{lemma}
We next give the behavior of $\cS^\beta_t$ under upper tail events. To be concise, for the rest of this section we fix $t>0$. All the constants (including all $C,c>0$) are allowed to depend on $t$.
And $\beta$ is taken to equal either $1$ or $\infty$.

\subsection{Coalescence and independent tents with Brownian bridges}

We now give our main coalescence estimate, in the form of stating that the quadrangle inequalities from \eqref{eq:DL-quad} and \eqref{eq:FR-quad} are sharp under the upper tail. Its proof will be given in \Cref{ss:coap}.
\begin{proposition}  \label{prop:dp-tent-coal}
Take any $L>0$ and $L^+>L+\exp(-0.001L^{3/2})$ (including $L^+=\infty$), and denote $H=10^{-6}L^{1/2}$.
When $\beta=\infty$,
\[
\PP\left(\cS^\beta_t(x,y) = \cS^\beta_t(x,0)+\cS^\beta_t(0,y)-\cS^\beta_t(0,0),\; \forall |x|, |y|\le H\midd L<\cS^\beta_t(0,0)<L^+\right) > 1-C\exp(-cL^{3/2}).
\]
When $\beta=1$ and $L>(t^{-1/3-\varepsilon}\vee 1)L_0$, the same bound holds when the event is replaced by
\[
|\cS^\beta_t(x,y) - (\cS^\beta_t(x,0)+\cS^\beta_t(0,y)-\cS^\beta_t(0,0))| < C\exp(-cL),\quad |x|, |y|\le H.
\]
\end{proposition}

The lower bound on $L^+$  above is to ensure a lower bound on $\PP(L<\cS^\beta_t(0,0)<L^+)$ which will be needed in the proof.

The connection between the above estimate and coalescence is that, at zero temperature ($\beta=\infty$), the equality is equivalent to the coalescence of a family of geodesics, according to \Cref{lem:DL-quad-st}.
At positive temperature, almost surely the quadrangle inequality is strict (see \eqref{eq:FR-quad}), so the equality is replaced by an upper bound of $C\exp(-cL)$, which is roughly the probability for the corresponding polymers to be disjoint given the field $\cZ$.

As we have seen from e.g.~\Cref{lem:fh-tent}, there are tent behaviors under the upper tail event.
The following proposition states that conditional on the upper tail event, the two tents seen from both positive and negative directions are roughly independent, and are close to Brownian bridges. 
It is a two-sided version of \Cref{l.dp-comp-s}.
\begin{proposition}  \label{prop:dp-tent-bcomp}
There exists $L_0>0$ such that the following holds. Take any $t>0$, $L>(t^{-1/3-\varepsilon}\vee 1)L_0$ and $L^+>L+\exp(-0.001L^{3/2})$ (including $L^+=\infty$).  
Let  $H=10^{-6}L^{1/2}$.
Consider the following processes, each defined on $[0,H]$,
\[
x\mapsto \cS^\beta_t(0,x)-\cS^\beta_t(0,0), \quad x\mapsto \cS^\beta_t(0,-x)-\cS^\beta_t(0,0),\]
\[
x\mapsto \cS^\beta_t(x,0)-\cS^\beta_t(0,0), \quad x\mapsto \cS^\beta_t(-x,0)-\cS^\beta_t(0,0),
\]
conditional on $L<\cS^\beta_t(0,0)<L^+$. Further, also consider the processes (each defined on $[0,H]$ as well)
\[x\mapsto B_1(x)+x(\cS^\beta_t(0,H)-\cS^\beta_t(0,0))/H,\quad x\mapsto B_2(x)+x(\cS^\beta_t(0,-H)-\cS^\beta_t(0,0))/H,\]
\[x\mapsto B_3(x)+x(\cS^\beta_t(H,0)-\cS^\beta_t(0,0))/H,\quad x\mapsto B_4(x)+x(\cS^\beta_t(-H,0)-\cS^\beta_t(0,0))/H,\]
also conditional on $L<\cS^\beta_t(0,0)<L^+$, where $B_1, B_2, B_3, B_4$ are four rate $2$ Brownian bridges in $[0, H]$, independent of each other and independent of $\cS^\beta_t$.

There is another measure on $\mc C([0, H],\R)^4$, such that, with probability $>1-C\exp(-cL^{3/2})$, (1) its Radon-Nikodym derivative over the first set of processes is $1+O(\exp(-cL))$ and (2) it can be coupled with the second set of processes such that, the $L^{\infty}$ distance between them is $<C\exp(-cL)$.
\end{proposition}

The proof of \Cref{prop:dp-tent-bcomp} involves reducing it to the following statement, using the shift-invariance of \Cref{lem:fr-shift} and \Cref{prop:dp-tent-coal}.

\begin{lemma}  \label{lem:dp-bcomp-s}
There exists $L_0>0$ such that for any $t>0$ and $L>(t^{-1/3-\varepsilon}\vee 1)L_0$ the following holds. Denote $H=10^{-6}L^{1/2}$ and consider the following processes, each defined on $[0,H]$,
\[
x\mapsto \cS^\beta_t(0,x)-\cS^\beta_t(0,0), \quad x\mapsto \cS^\beta_t(0,-x)-\cS^\beta_t(0,0),\]
\[x\mapsto \cS^\beta_t(0,-x-H)-\cS^\beta_t(0,-H), \quad x\mapsto \cS^\beta_t(0,x+H)-\cS^\beta_t(0,H),
\]
conditional on $\cS^\beta_t(0,0)\in (L,L+\diff L)$. Also, consider the processes (each defined on $[0,H]$ as well)
\[x\mapsto B_1(x)+x(\cS^\beta_t(0,H)-\cS^\beta_t(0,0))/H,\quad x\mapsto B_2(x)+x(\cS^\beta_t(0,-H)-\cS^\beta_t(0,0))/H,\]
\[x\mapsto B_3(x)+x(\cS^\beta_t(0,-2H)-\cS^\beta_t(0,-H))/H,\quad x\mapsto B_4(x)+x(\cS^\beta_t(0,2H)-\cS^\beta_t(0,H))/H,\]
also conditional on $\cS^\beta_t(0,0)\in (L,L+\diff L)$, where $B_1, B_2, B_3, B_4$ are four rate $2$ Brownian bridges in $[0, H]$, independent of each other and independent of $\cS^\beta_t$.
They can be coupled so that under an event with probability $>1-C\exp(-cL^{3/2})$ for both, their Radon-Nikodym derivative is $1+O(\exp(-cL))$.
\end{lemma}
This lemma directly follows from \Cref{l.dp-comp-s}, and we omit the details.
We now prove \Cref{prop:dp-tent-bcomp} assuming \Cref{prop:dp-tent-coal} and \Cref{lem:dp-bcomp-s}.

\begin{figure}[!hbt]
    \centering

\includegraphics[width=\textwidth]{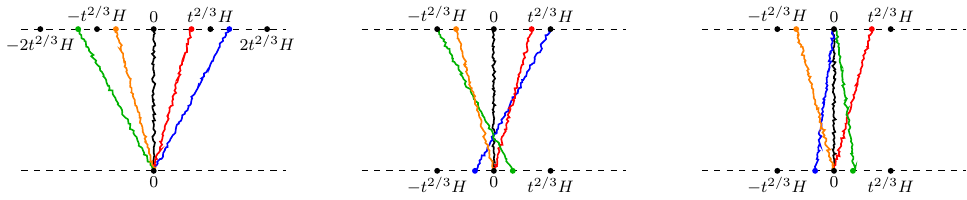}
    \caption{An illustration of transforming the weights of varying one side (left panel) into varying both sides (middle panel) using shift-invariance, then to two tents (right panel) by coalescence.}
    \label{fig:sinvS}
\end{figure}

\begin{proof}[Proof of \Cref{prop:dp-tent-bcomp}]
By shift-invariance (\Cref{lem:fr-shift}), in \Cref{lem:dp-bcomp-s} we can replace $\cS^\beta_t(0, x+H)$ by $\cS^\beta_t(-x,H)$ and $\cS^\beta_t(0,-x-H)$ by $\cS^\beta_t(x,-H)$ (see \Cref{fig:sinvS}).
Note that \Cref{lem:fr-shift} is stated in terms of finitely many points, but we can do the replacement for each $x\in [0,H]$ simultaneously since $\cS^\beta_t$ is continuous.
Therefore we get the following statement.
Consider the processes
\[
x\mapsto \cS^\beta_t(0,x)-\cS^\beta_t(0,0), \quad x\mapsto \cS^\beta_t(0,-x)-\cS^\beta_t(0,0),\]
\begin{equation}  \label{eq:tent-bcomp-pf1}
x\mapsto \cS^\beta_t(x,-H)-\cS^\beta_t(0,-H), \quad x\mapsto \cS^\beta_t(-x,H)-\cS^\beta_t(0,H),    
\end{equation}
on $[0, H]$, conditional on $L<\cS^\beta_t(0,0)<L^+$; and the processes (also on $[0, H]$)
\[x\mapsto B_1(x)+x(\cS^\beta_t(0,H)-\cS^\beta_t(0,0))/H,\quad x\mapsto B_2(x)+x(\cS^\beta_t(0,-H)-\cS^\beta_t(0,0))/H,\]
\begin{equation}  \label{eq:tent-bcomp-pf2}
x\mapsto B_3(x)+x(\cS^\beta_t(H,-H)-\cS^\beta_t(0,-H))/H,\quad x\mapsto B_4(x)+x(\cS^\beta_t(-H,H)-\cS^\beta_t(0,H))/H,
\end{equation}
conditional on $L<\cS^\beta_t(0,0)<L^+$, where $B_1, B_2, B_3, B_4$ are four rate $2$ Brownian bridges in $[0, H]$, independent of each other and independent of $\cS^\beta_t$.
They can be coupled so that under an event with probability $>1-C\exp(-cL^{3/2})$ for both, their Radon-Nikodym derivative is $1+O(\exp(-cL))$.

By \Cref{prop:dp-tent-coal}, with probability $>1-C\exp(-cL^{3/2})$, we have
\[
| (\cS^\beta_t(x,-H)-\cS^\beta_t(0,-H))-(\cS^\beta_t(x,0)-\cS^\beta_t(0,0)) | < C\exp(-cL),
\]
\[
| (\cS^\beta_t(-x,H)-\cS^\beta_t(0,H))-(\cS^\beta_t(-x,0)-\cS^\beta_t(0,0)) | < C\exp(-cL),
\]
for any $x\in [0, H]$; in particular, 
\[
| (\cS^\beta_t(H,-H)-\cS^\beta_t(0,-H))-(\cS^\beta_t(H,0)-\cS^\beta_t(0,0)) | < C\exp(-cL),
\]
\[
| (\cS^\beta_t(-H,H)-\cS^\beta_t(0,H))-(\cS^\beta_t(-H,0)-\cS^\beta_t(0,0)) | < C\exp(-cL).
\]
By plugging these estimates into \eqref{eq:tent-bcomp-pf1} and \eqref{eq:tent-bcomp-pf2} respectively, we get the conclusion.
\end{proof}
It remains to prove \Cref{prop:dp-tent-coal}, which we accomplish in the next subsection.

\subsection{Coalescence of polymers}  \label{ss:coap}

Using the quadrangle inequalities \eqref{eq:DL-quad} and \eqref{eq:FR-quad}, \Cref{prop:dp-tent-coal} can be reduced to the following lemma.

\begin{lemma}  \label{lem:dp-coal}
There exists $L_0$ such that the following holds. Take any $t>0$, $L>(t^{-1/3-\varepsilon}\vee 1)L_0$, and $L^+>L+\exp(-0.001L^{3/2})$, and denote $H=10^{-6}L^{1/2}$.
Then when $\beta=\infty$, we have
\[
\PP\left(\cS^\beta_t(-H,-H)+\cS^\beta_t(H,H)>\cS^\beta_t(-H,H)+\cS^\beta_t(H,-H) \mid L<\cS^\beta_t(0,0)<L^+\right) < C\exp(-cL^{3/2}). 
\]
And when $\beta=1$, the same estimate holds with the event replaced by
\[
\cS^\beta_t(-H,-H)+\cS^\beta_t(H,H)-\cS^\beta_t(-H,H)-\cS^\beta_t(H,-H) > C\exp(-cL).
\]
\end{lemma}
In the $\beta=\infty$ setting, in light of \Cref{lem:DL-quad-st}, the event whose probability we wish to bound is equivalent to that (in the directed landscape) the geodesic from $(-t^{2/3}H,0)$ to $(-t^{2/3}H,t)$ and the geodesic from $(t^{2/3}H,0)$ to $(t^{2/3}H,t)$ are disjoint.
Such disjointness is unlikely to happen under the upper large since then geodesics tend to merge into the geodesic from $(0,0)$ to $(0,t)$, shortly away from the endpoints.
When $\beta=1$, the event is instead interpreted as that the multi-point partition function from the spatial coordinates $-t^{2/3}H$ and $t^{2/3}H$ at time $0$ to $-t^{2/3}H$ and $t^{2/3}H$ at time $t$ is comparable to the product of the two individual ones.
This can be then understood as a positive temperature form of disjointness.

For multi-point passage times and multi-point partition functions of size $2$, we can only estimate them when each side of the endpoints is at a single point, using \Cref{lem:fh-2ut}. 
In the proof of \Cref{lem:dp-coal} we will need to upper bound those from $-t^{2/3}H, t^{2/3}H$ at time $0$ to $-t^{2/3}H, t^{2/3}H$ at time $t$.
Our strategy is to do surgeries around time $0$ and time $1$ (see \Cref{fig:surg}): we instead upper bound those from $0, 0$ at time $-\varepsilon t$ to $0, 0$ at time $(1+\varepsilon )t$. Here $\varepsilon >0$ is a small constant.
Then we use the composition laws and lower bound various passage times and partition functions between time $-\varepsilon t$ and $0$, and time $t$ and $(1+\varepsilon )t$.
We note that such surgery arguments have appeared in the directed landscape and LPP models, see e.g.~\cite{Hexp}.

\begin{figure}[!hbt]
    \centering
\includegraphics{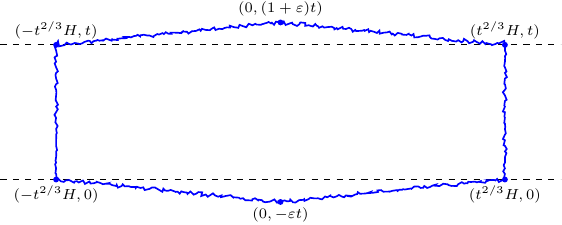}
    \caption{An illustration of the surgery, bringing the endpoints of multi-point passage times and multi-point partition functions together.}
    \label{fig:surg}
\end{figure}

We now give the details. 
For concreteness, we take $\varepsilon =10^{-6}$ in the rest of this section.
First, we note that by taking $C$ large and $c$ small, it suffices to prove \Cref{lem:dp-coal} for large enough $L$.
By \Cref{lem:fr-shift}, $\cS^\beta_t(-H,H)+\cS^\beta_t(H,-H)$ and $\cS^\beta_t(0,0)$ have the same joint distribution as $\cS^\beta_t(0,2H)+\cS^\beta_t(0,-2H)$ and $\cS^\beta_t(0,0)$.
Then by \Cref{lem:fh-tent} we have
\begin{equation}  \label{eq:fhH-c-bd}
\PP\left(\cS^\beta_t(-H,H)+\cS^\beta_t(H,-H)< (2-0.001)L-8L^{1/2}H \mid L<\cS^\beta_t(0,0)<L^+\right) < \exp(-cL^{3/2}).  
\end{equation}
Also by \Cref{lem:fh-ut}, for large $L$ we have 
\begin{equation}  \label{eq:Ltent}
    \PP\left(L<\cS^\beta_t(0,0)<L^+\right)>c\exp(-(4/3+0.005)L^{3/2}).
\end{equation}

\subsubsection{Directed landscape ($\beta=\infty$) setting}
In this case, we can assume $t=1$ since the law of $\cS^{\beta=\infty}_t$ is independent of $t$. Recall that $\cS^{\beta=\infty}_1=\cL(\cdot, 0;\cdot, 1)$.
On the event in \Cref{lem:dp-coal} whose probability we wish to bound,
 \Cref{lem:quad-equa} and \Cref{lem:DL-quad-st}  imply that
\[
 \cL(-H,0;H,1)+\cL(H,0;-H,1) < \cL(-H,0;-H,1)+\cL(H,0;H,1) = \cL((-H,H),0;(-H,H),1).
\]
Therefore, since we have lower bounded the LHS in \eqref{eq:fhH-c-bd}, to prove \Cref{lem:dp-coal} it suffices to prove the following estimate.

\begin{lemma}  \label{lem:L-bd}
For $L>0$ large enough and $H=10^{-6}L^{1/2}$, we have
\[
\PP\big[ \cL((-H,H),0;(-H,H),1)> 1.99L\big]\\
< C\exp\big( -(4/3+0.01)L^{3/2} \big).
\]
\end{lemma}
Using this lemma and \eqref{eq:Ltent}, we can bound the probability of the same event conditional on $L<\cS^\beta_t(0,0)<L^+$, by $C\exp(-0.005L^{3/2})$.
Then by \eqref{eq:fhH-c-bd}, and noting that $(2-0.001)L-8L^{1/2}H>1.99L$, we get \Cref{lem:dp-coal}.

\begin{proof}[Proof of \Cref{lem:L-bd}]
By using \eqref{eq:mult-quad} twice (first with $y_1 = -H$, $y_2 = 0$, $y_3 = H$ and then with $y_1 = y_2 = 0$ and $y_3 = -H$), we have
\[
\cL((0,0),-\varepsilon ;(-H,H),0) \ge \cL((0,0),-\varepsilon ;(0,0),0) + \cL(0,-\varepsilon ;-H,0) + \cL(0,-\varepsilon ;H,0) - 2\cL(0,-\varepsilon ;0,0).
\]
Denote $w:=\PP(\cL((0,0),-\varepsilon ;(0,0),0)>0)$.
There exists $c_*>0$, such that
\[
\PP\Bigl(\cL(0,-\varepsilon ;-H,0) + H^2/\varepsilon  < -c_*\Bigr),\; \PP\Bigl(\cL(0,-\varepsilon ;H,0) + H^2/\varepsilon  < -c_*\Bigr),\; \PP\Bigl(\cL(0,-\varepsilon ;0,0) > c_*\Bigr) < w/4.
\]
We note that $w, c_*$ are independent of $L$.
Therefore
\[
\PP\Bigl(\cL((0,0),-\varepsilon ;(-H,H),0) > -2H^2/\varepsilon  - 4c_*\Bigr) > w/4.
\]
Similarly,
\[
\PP\Bigl(\cL((-H,H),1;(0,0),1+\varepsilon ) > -2H^2/\varepsilon  - 4c_*\Bigr) > w/4.
\]
Note that
\begin{multline*}
\cL_2(0,-\varepsilon ;0,1+\varepsilon ) \\ \ge \cL((0,0),-\varepsilon ;(-H,H),0) + \cL((-H,H),0;(-H,H),1) + \cL((-H,H),1;(0,0),1+\varepsilon ).   
\end{multline*}
So we have
\[
\PP\left(\cL_2(0,-\varepsilon ;0,1+\varepsilon ) > 1.99L-4H^2/\varepsilon -8c_*\right) \ge 
\PP\Bigl( \cL((-H,H),0;(-H,H),1)> 1.99L\Bigr)w^2/16.
\]
Using \Cref{lem:fh-2ut} and the fact that $\cL_2(0,-\varepsilon ;0,1+\varepsilon )\le 2\cL(0,-\varepsilon ;0,1+\varepsilon )$ almost surely, we can bound the left-hand side by $\exp\big( -(4/3+0.01)L^{3/2} \big)$. Thus the conclusion follows.
\end{proof}

\subsubsection{Positive temperature ($\beta=1$) setting}
Recall from \Cref{sss:mpp} the notation $\cM$ for multi-point partition functions. We can then write
\begin{multline}  \label{eq:expSSS}
\exp\big(t^{1/3}(\cS^\beta_t(-H,-H)+\cS^\beta_t(H,H)-\cS^\beta_t(-H,H)-\cS^\beta_t(H,-H))\big)-1\\
= \frac{(2t^{2/3}H)^2\cM((-t^{2/3}H, t^{2/3}H), 0; (-t^{2/3}H, t^{2/3}H), t) }
{\exp\big(t^{1/3}(\cS^\beta_t(-H,H)+\cS^\beta_t(H,-H))-t/6\big)}.
\end{multline}
Therefore, to prove \Cref{lem:dp-coal}, the main task is to prove the following estimate.
\begin{lemma}  \label{lem:cM-t-H-bd}
There exists $L_0$ such that, for any $t>0$ and $L>(t^{-1/3-\varepsilon}\vee 1)L_0$, with $H=10^{-6}L^{1/2}$,
\[
\PP\Bigl( \cM\bigl((-t^{2/3}H, t^{2/3}H), 0; (-t^{2/3}H, t^{2/3}H), t)\bigr) >\exp\bigl(t^{1/3}\cdot 1.99L\bigr)\Bigr)
< C\exp\big( -(4/3+0.01)L^{3/2} \big).
\]
\end{lemma}
Using this lemma and \eqref{eq:Ltent}, we can bound the probability of the same event conditional on $L<\cS^\beta_t(0,0)<L^+$, by $C\exp(-0.005L^{3/2})$.
Then by \eqref{eq:fhH-c-bd}, conditional on $L<\cS^\beta_t(0,0)<L^+$, with probability $>1-C\exp(-cL^{3/2})$ the expression \eqref{eq:expSSS} is upper bounded by
\[
\frac{(2t^{2/3}H)^2 \exp\big(t^{1/3}\cdot 1.99L\big) }{\exp\big(t^{1/3}\cdot (2L-8L^{1/2}H-0.001L) -t/6 \big)} < C\exp(-cL).
\]
Thus we get \Cref{lem:dp-coal}. 

The rest of this section is devoted to proving \Cref{lem:cM-t-H-bd}.
We let $\cE$ denote the event whose probability we wish to bound, in the statement of \Cref{lem:cM-t-H-bd}.
By \Cref{lem:fh-2ut}, we have
\begin{multline*}
  \PP\Big( \cZ(0,-\varepsilon t;0,(1+\varepsilon )t) > \exp\big(t^{1/3}\cdot 0.98L\big),\; \cZ_2(0,-\varepsilon t;0,(1+\varepsilon )t) > \exp\big(t^{1/3}\cdot 1.96L\big) \Big) 
  \\ < C\exp\big( -(8/3-0.1)L^{3/2} \big).
\end{multline*}
By the continuity of $\cM$ (\Cref{lem:ext}), we have almost surely
\[
2^{-1}(1+2\varepsilon )t\cM((-\delta, \delta),-\varepsilon t;(-\delta, \delta),(1+\varepsilon )t) \to \cZ_2(0,-\varepsilon t;0,(1+\varepsilon )t),
\]
as $\delta\to 0$ from the right.
Then there exists small enough $\delta_L>0$, such that $\PP(\cE^+)<C\exp\big( -(8/3-0.1)L^{3/2} \big)$, with $\cE^+$ being the event where
\begin{align*}
    \cZ(0,-\varepsilon t;0,(1+\varepsilon )t) &> \exp\big(t^{1/3}\cdot 0.98L\big),\\
\cM\bigl((-\delta_L, \delta_L),-\varepsilon t;(-\delta_L, \delta_L),(1+\varepsilon )t\bigr) &> \exp\big(t^{1/3}\cdot 1.98L\big).
\end{align*}
We will show that conditional on $\cE$, with positive probability (independent of $L$), $\cE^+$ holds.
For this, we define the following events.

By \Cref{lem:ext}, there exists a small number $w>0$, such that
\[
    \PP\left(\cM((-x,x),-\varepsilon t;(-w, w),0)>w\right) > w
\]
for any $|x|\le w$.
We let $\cE_1$ be the event where
\[
\cM\bigl((-\delta_L, \delta_L),-\varepsilon t; (-w,w), 0\bigr) > w, \quad \cM\bigl( (-w,w), t; (-\delta_L, \delta_L,(1+\varepsilon )t\bigr)>w.
\]
Then $\PP(\cE_1)>w^2$.

We let $\cE_2$ be the event where for any $x, y\in\RR$, 
    \[
    \left| (\varepsilon t)^{-1/3}(\log\cZ((\varepsilon t)^{2/3}x,-\varepsilon t;(\varepsilon t)^{2/3}y,0)+\varepsilon t/12) + (x-y)^2\right| < 10^{-3}L + \log(|x-y|+2),
    \]
    and
        \[
    \left| (\varepsilon t)^{-1/3}(\log\cZ((\varepsilon t)^{2/3}x,t;(\varepsilon t)^{2/3}y,(1+\varepsilon )t)+\varepsilon t/12) + (x-y)^2\right| < 10^{-3}L + \log(|x-y|+2).
    \]
Then by \Cref{lem:fh-ub}, we have $\PP(\cE_2)>1- C\exp(-cL^{3/2})$.

By \Cref{lem:fh-int-ub} and the shift invariance of $\cZ$, we have $\PP(\cE^*)>1-C\exp(-cL^2)$, for $\cE^*$ denoting the event where
\[
\sup_{x,y\in [-H-1, -H]} \bigl|\cS^\beta_t(x,y)-\cS^\beta_t(-H,-H)\bigr|,\; \sup_{x,y\in [H, H+1]} \bigl|\cS^\beta_t(x,y)-\cS^\beta_t(H,H)\bigr| < 10^{-10}L.
\]

\begin{lemma} \label{lem:eveht}
We have $\cE_1\cap\cE_2\cap\cE^*\cap\cE \subset \cE^+$.
\end{lemma}
\begin{proof}
In this proof we always take
\[x, x'\in [-t^{2/3}(H+1), -t^{2/3}H], \quad y, y'\in [t^{2/3}H, t^{2/3}(H+1)].\] 
We will use $\cE_1\cap\cE_2$ to lower bound $\cM((-\delta_L, \delta_L),-\varepsilon t; (x,y), 0)$ and $\cM((x',y'), t ; (-\delta_L, \delta_L),(1+\varepsilon )t)$,
and $\cZ(0,-\varepsilon t; x, 0)$, $\cZ(x',t; 0, (1+\varepsilon )t)$; then we will lower bound $\cM((x,y), 0; (x',y'), t)$ and $\cZ(x,0;x',t)$ for these $x, x', y, y'$, using $\cE\cap \cE^*$.
Then using the composition laws \eqref{eq:cZ-comp} and \eqref{eq:cM-comp}, we obtain $\cE^+$.

\noindent\textbf{Step 1.}
Using \Cref{lem:cM-t-mon}, we have that
\begin{multline*}
(y-x)\cM((-\delta_L, \delta_L),-\varepsilon t; (x,y), 0) > 
2w\cM((-\delta_L, \delta_L),-\varepsilon t; (-w,w), 0)    \\
\times
\frac{\cZ(-\delta_L,-\varepsilon t;x,0)\cZ(\delta_L,-\varepsilon t;y,0)}
{\cZ(-\delta_L,-\varepsilon t;-w,0)\cZ(\delta_L,-\varepsilon t;w,0)}.
\end{multline*}
Under $\cE_1$, the first factor in the right-hand side is $>2w^2$; and under $\cE_2$, the second factor in the right hand side is $>\exp(-t^{1/3}10^{-3}L)$, when $L$ is large enough.
Therefore, we have
\begin{equation}  \label{eq:Mb}
\cM((-\delta_L, \delta_L),-\varepsilon t; (x,y), 0) > c  \exp(-t^{1/3}10^{-3}L).
\end{equation}
By symmetry, the same lower bound holds for $\cM((x',y'), t ; (-\delta_L, \delta_L),(1+\varepsilon )t)$.

The event $\cE_2$ also implies that
\begin{equation}  \label{eq:Zb}
\cZ(0, -\varepsilon t;x, 0), \cZ(x', 2t; 0, (2+\varepsilon )t) > \exp(-t^{1/3}10^{-3}L).
\end{equation}
\noindent\textbf{Step 2.}
Using \Cref{lem:cM-t-mon}, we have that
\begin{multline*}
2\cM((x,y),0; (x',y'), t) > 
\cM((-t^{2/3}H, t^{2/3}H),0; (-t^{2/3}H, t^{2/3}H), t)    \\
\times
\frac{\cZ(x',0;x,t)\cZ(y',-0;y,t)}
{\cZ(-t^{2/3}H,0;-t^{2/3}H,t)\cZ(t^{2/3}H,-0;t^{2/3}H,t)}.
\end{multline*}
By $\cE$, the first factor in the right-hand side is $>\exp(t^{1/3}\cdot 1.99L)$; and by $\cE^*$, the second factor in the right-hand side is 
$>\exp(-t^{1/3}10^{-9}L)$.
Therefore we have
\begin{equation}  \label{eq:Mm}
\cM((x,y),0; (x',y'), t) > c\exp(t^{1/3}\cdot (1.99-10^{-9})L).
\end{equation}

From $\cE$, we also have that
\[
\cZ(-t^{2/3}H,0; -t^{2/3}H, t) \vee \cZ(t^{2/3}H,0; t^{2/3}H, t) > \exp(t^{1/3}\cdot 0.99L).
\]
Without loss of generality, we assume that 
\[
\cZ(-t^{2/3}H,0; -t^{2/3}H, t) > \exp(t^{1/3}\cdot 0.99L).
\]
Then by $\cE^*$, we have
\begin{equation}  \label{eq:Zm}
\cZ(x,0; x', t) > \exp(t^{1/3}\cdot (0.99-10^{-9})L).
\end{equation}
\noindent\textbf{Step 3.}
By the composition law \eqref{eq:cM-comp}, and \eqref{eq:Mb}, \eqref{eq:Mm}, we have
\[
\cM( (-\delta_L, \delta_L),-\varepsilon t; (-\delta_L, \delta_L), (1+\varepsilon )t ) > c\exp(t^{1/3}\cdot (1.99-2\cdot 10^{-3}-10^{-9})L).
\]
By the composition law \eqref{eq:cZ-comp}, and \eqref{eq:Zb}, \eqref{eq:Zm}, we have
\[
\cZ(0,-\varepsilon t;0,(1+\varepsilon )t) > c\exp(t^{1/3}\cdot (0.99-2\cdot 10^{-3}-10^{-9})L).
\]
These imply the event $\cE^+$.
\end{proof}

\begin{proof}[Proof of \Cref{lem:cM-t-H-bd}]
By \Cref{lem:eveht}, and the fact that $\cE_1, \cE_2$ are independent of $\cE^*, \cE$, we have $\PP(\cE_1\cap\cE_2)\PP(\cE^*\cap\cE)\le \PP(\cE^+)$.
Then by the bounds on $\PP(\cE_1)$, $\PP(\cE_2)$, $\PP(\cE^*)$, $\PP(\cE)$, and $\PP(\cE^+)$, we have that
\[
\PP(\cE) \le C\exp\big( -(8/3-0.1)L^{3/2} \big) (w^2-C\exp(-cL^{3/2}))^{-1} - C\exp(-cL^2), 
\]
and this is bounded by $C\exp\big(-(4/3+0.01)L^{3/2}\big)$ for $L$ large enough.
\end{proof}

\section{Tail comparison estimates}\label{ssec:tailcom}

To reduce notations, in the rest of this paper we denote
\[
\cL^{\beta=1}(x,s; y,t) = \log \mc Z(x,s; y,t) + (t-s)/12, \quad \cL^{\beta=\infty}(x,s; y,t) = \cL(x,s; y,t).
\]
We take $\beta=1$ or $\infty$ systematically unless otherwise noted.
From the above definition we have $\fh^\beta_{t,1}(x)=\cL^\beta(0,0;x,t)$.

In this section, $t$ is taken to be any $t>t_0$, where $t_0$ is an arbitrary positive number. All the constants may depend on $t_0$, but are uniform in $t$.

As indicated in \Cref{iop}, the following tail ratio estimate would be central in much of our analysis.
\begin{theorem}\label{t.comparison}
For any $L\ge 2$, and $0< \delta<L^{1/4}$, 
\begin{equation}  \label{eq:tcomparison}
\frac{\P\left(\hfh^\beta_{t,1}(0) > L+\delta\right)}{\P(\hfh^\beta_{t,1}(0)>L)}=\exp(-2\delta L^{1/2} + O(\delta L^{-1/4}\log(L) +L^{-3/2})) .
\end{equation}
For $\delta\ge L^{1/4}$, the same ratio equals $\exp(-\Omega(\delta L^{1/2}))$.
\end{theorem}
The main term $\exp(-2\delta L^{1/2})$ comes from the following: from the tent behavior, 
one considers a Brownian bridge in $[-L^{-1/2}, L^{1/2}]$ that equals $-L$ at the two ends; then the ratio is roughly the probability that it is $>L+\delta$ at $0$ versus the probability that it is $>L$ at $0$. As mentioned in the idea of proofs Section~\ref{iop}, it can also be understood by Taylor expanding $\frac{4}{3}x^{3/2}$ around $x=L$.

The general strategy will be to use the Gibbs property (\Cref{lem:Gibbs}) to resample $\smash{\hfh^\beta_{t,1}}$ in an interval $[-L^{1/2}+M, L^{1/2}-M]$, {with some $M$ large but much smaller than $L^{1/2}$.}
The choice of $M$ is such that, 
conditioned on the upper tail large deviation event, $\smash{\hfh^\beta_{t,1}}$ in the interval $[-L^{1/2}+M, L^{1/2}-M]$ is not much affected by the second line $\smash{\hfh^\beta_{t,2}}$ since, by the tent behavior, it will have obtained some separation from the second line by that time. Thus $\smash{\hfh^\beta_{t,1}}$ can be analyzed as a Brownian bridge on that interval.
For this, we need the following estimate of $\smash{\hfh^\beta_{t,1}}$ at the end point $-L^{1/2}+M$ (and also for $L^{1/2}-M$ by symmetry).

\begin{lemma}\label{l.good separation}
{For any $0<M<L^{1/2}$ and $L\ge 2$},
\begin{align*}
\P\left(\hfh^\beta_{t,1}(-L^{1/2}+M) \leq -(L^{1/2}-M)^2 + \tfrac{1}{2}M^2 \midd \hfh^\beta_{t,1}(0)>L\right) < C\exp(-cM^{5/2}).
\end{align*}
\end{lemma} 

By the tent behavior, we expect the expectation of $\smash{\hfh^\beta_{t,1}(\pm(L^{1/2}-M))}$ to be around $-L+2ML^{1/2}$, with a Gaussian fluctuation of order $M^{1/2}$.
This estimate bounds the probability that it is order $M^2$ smaller than its mean.
The reason the upper bound is $\exp(-cM^{5/2})$ rather than $\exp(-c(M^2)^2/M)=\exp(-cM^3)$ is due to interactions with $\smash{\hfh^\beta_{t,2}}$.

In the following proof we will use the notation $\EF[\cdot]:=\E[\,\cdot\mid \F]$ and $\PF(\cdot) := \P(\cdot\mid \F)$ for a $\sigma$-algebra $\F$. The existence of the regular conditional distribution is justified by the fact that the $\sigma$-algebras we consider will be generated by random variables taking values in Borel spaces and then invoking abstract existence results such as \cite[Theorem~8.5]{Kallenberg}.

\begin{proof}[Proof of \Cref{l.good separation}]
Let $\F = \Fext([-t^{2/3}L^{1/2}, 0])$ (recall the definition of $\Fext$ from \Cref{ss:legp}) and $L_M = (L^{1/2}-M)^2$. Then we can write the probability of the complement of the event in the LHS in the statement of the lemma as
\begin{align}\label{e.conditional probability ratio for near tangent bound}
\frac{\E\left[\PF\left(\hfh^\beta_{t,1}(-L_M^{1/2}) > -L_M + \tfrac{1}{2}M^2\right)\one_{\hfh^\beta_{t,1}(0)>L}\right]}{\P(\hfh^\beta_{t,1}(0)>L)}.
\end{align}
Let us focus on the conditional probability in the numerator. Let $\tilde B$ be a Brownian bridge from $(-L^{1/2}, \hfh^\beta_{t,1}(-L^{1/2}))$ to $(0, \hfh^\beta_{t,1}(0))$, such that $x\mapsto t^{1/3}\tilde{B}(t^{-2/3}x)$ in $[-t^{2/3}L^{1/2}, 0]$ interacts with $\fh^\beta_{t,2}$ by the Radon Nikodym derivative reweighting \eqref{e.rn derivative}. Then the Brownian Gibbs property says that the conditional probability in the previous display equals
\begin{align}\label{e.gap probability lower bound}
\PF\left(\tilde B(-L_M^{1/2}) > -L_M + \tfrac{1}{2}M^2\right).
\end{align}
Next let $B$ be a Brownian bridge from $(-L^{1/2}, -L-M)$ to $(0,L)$ (with no lower boundary conditioning). Then, on the $\F$-measurable event 
\begin{equation}\label{e.A(L,W) definition}
A(L,M) = \left\{\hfh^\beta_{t,1}(-L^{1/2}) > -L-M, \;\hfh^\beta_{t,1}(0)>L\right\},  
\end{equation}
it holds that $\tilde B$ dominates $B$ by monotonicity (Lemma~\ref{l.monotonicity}), so \eqref{e.gap probability lower bound} is lower bounded by
\begin{align}
\PF\left(B(-L_M^{1/2}) > -L_M + \tfrac{1}{2}M^2\right)\one_{A(L, M)}. \label{e.B at L_W bound}
\end{align}
Now $B(-L_M^{1/2})$ is distributed as a normal random variable with mean 
$$-L - M + M\cdot \frac{L-(-L-M)}{L^{1/2}} = -L - M + 2L^{1/2}M + M^2L^{-1/2}$$
and variance $\sigma^2 = \frac{2M (L^{1/2}-M)}{L^{1/2}} \leq 2M$. Since $-L_M = -(L^{1/2}-M)^2 = -L + 2L^{1/2}M - M^2$, we see that \eqref{e.B at L_W bound} equals, on $A(L,M)$,
\begin{align*}
\P\left(\mathcal{N}(0,\sigma^2) > -M^2(\tfrac{1}{2}+L^{-1/2}) + M \right) \geq 1-\exp\left(-cM^3\right)
\end{align*}
using standard tail bounds for the normal distribution (Lemma~\ref{l.normal bounds}).
Putting this back in \eqref{e.conditional probability ratio for near tangent bound} and recalling the definition \eqref{e.A(L,W) definition} of $A(L,M)$ , we see that the LHS in the lemma is upper bounded by
\begin{align*}
\MoveEqLeft
1-(1-\exp(-cM^{3})) \P\left(\hfh^\beta_{t,1}(-L^{1/2}) > -L-M \midd \hfh^\beta_{t,1}(0)>L\right).
\end{align*}
By the FKG inequality (\Cref{l.fkg}) and \Cref{l.lower tail}, we have 
\begin{align*}
\P\left(\hfh^\beta_{t,1}(-L^{1/2}) > -L-M \midd \hfh^\beta_{t,1}(0)>L\right) > 1-\exp(-cM^{5/2})).
\end{align*}
This completes the proof.
\end{proof}

\begin{proof}[Proof of Theorem~\ref{t.comparison}]
We can assume that $L$ is large since otherwise the conclusion follows from Theorem~\ref{lem:fh-ut}.
For some constant $C_0$ sufficiently large, the case of $\delta > C_0L^{1/4}$ also follows from Theorem~\ref{lem:fh-ut}. In the remainder of the proof, we prove \eqref{eq:tcomparison} for $\delta < C_0L^{1/4}$. 

Denote $M=C_1\log(L)$, where $C_1>0$ is a large constant whose value is to be determined.
It suffices to show that
\begin{align*}
 \frac{\P\left(\hfh^\beta_{t,1}(0) > L+\delta\right)}{\P(\hfh^\beta_{t,1}(0)>L)}\cdot\exp(2\delta L^{1/2}) &< (1+CL^{-3/2})\exp( C\delta ML^{-1/4}) \qquad\text{and}\\
 \frac{\P\left(\hfh^\beta_{t,1}(0) > L+\delta\right)}{\P(\hfh^\beta_{t,1}(0)>L)}\cdot\exp(2\delta L^{1/2}) &> (1-CL^{-3/2})\exp( -C\delta ML^{-1/4}).
\end{align*}
We let $L_M = (L^{1/2}-M)^2$, and $\F=\Fext([-t^{2/3}L_M^{1/2}, t^{2/3}L_M^{1/2}])$. We start by considering the ratio of conditional probabilities
\begin{align*}
\frac{\PF\left(\hfh^\beta_{t,1}(0)>L + \delta\right)}{\PF\left(\hfh^\beta_{t,1}(0)>L\right)}.
\end{align*}
We adopt the notation $\B$ for the law of a Brownian bridge $B$ from $(-t^{2/3}L_M^{1/2}, \fh^\beta_{t,1}(-t^{2/3}L_M^{1/2}))$ to $(t^{2/3}L_M^{1/2}, \fh^\beta_{t,1}(t^{2/3}L_M^{1/2}))$, as well as the associated expectation.
With this notation, by the Brownian Gibbs property, the previous display equals
\[
\frac{\B\bigl(\one_{B(0) > t^{1/3}(L+\delta)}W(B,\fh^\beta_{t,2})\bigr)}{\B\bigl(\one_{B(0) > t^{1/3}L}W(B,\fh^\beta_{t,2})\bigr)}
= \frac{\B\bigl(B(0) > t^{1/3}(L+\delta)\bigr)}{\B\bigl(B(0) > t^{1/3}L\bigr)}\cdot\frac{\B\bigl(W(B,\fh^\beta_{t,2})\mid B(0)>t^{1/3}(L+\delta)\bigr)}{\B\bigl(W(B,\fh^\beta_{t,2})\mid B(0)>t^{1/3}L\bigr)},
\]
where $W(B,\fh^\beta_{t,2})$ is the weight factor from \eqref{e.rn derivative}.
Now the second ratio of terms in the previous display is lower bounded by $1$ using stochastic monotonicity properties of Brownian bridges and that $W$ is increasing in $B$. To upper bound the second ratio, we note that, since $W(B,\fh^\beta_{t,2}) \leq 1$,
it suffices to lower bound the denominator $\B\bigl(W(B,\fh^\beta_{t,2})\mid B(0)>t^{1/3}L\bigr)$. To do this we consider the $\F$-measurable event $\mrm{BdyCtrl} = \mrm{BdyCtrl}(L,M)$ defined by
\begin{align*}
\left\{\hfh^\beta_{t,1}(\pm L_M^{1/2}) \ge - L_M + \tfrac{1}{2}M^2\right\}\cap \bigcap_{i=0}^{M^{-1}L_M^{1/2}} \left\{\sup_{|x|\in[L_M^{1/2}-(i+1)M, L_M^{1/2}-iM]}\hfh^\beta_{t,2}(x)+x^2 \leq (i+1)M\right\}.
\end{align*}
By applying \Cref{l.good separation} to the first part, and taking $C_1$ large enough to apply \Cref{l.bk} to the second part (and using that the first term there dominates the term $t^{2/3}\exp(-cL^2)$ in this range of $i$), we get
\begin{multline}\label{e.A complement prob bound}
\P\left(\mrm{BdyCtrl}^c\mid \hfh^\beta_{t,1}(0)>L+\delta\right) < C\exp(-cM^{5/2}) \\ + \sum_{i=0}^{M^{-1}L_M^{1/2}} C\exp\left(-c(i+1)^{3/2}M^{3/2}+C(i+1)^{1/2}M^{1/2}\log(L)\right) < C\exp(-cM^{3/2}).
\end{multline}
Let $\ell^-$ be the line joining $(-L_M^{1/2}, -L_M + \frac{1}{4}M^2)$ and $(0, L-\frac{1}{4}M^2)$, $\ell^+$ be the line joining $(0, L-\frac{1}{4}M^2)$ and $(L_M^{1/2}, -L_M + \frac{1}{4}M^2)$, and $\ell$ be their concatenation. Consider the high corridor event 
$$\mrm{HighCorr} = \left\{t^{-1/3}B(t^{2/3}x)\geq \ell(x) - L_M^{1/2}+|x| \text{ for all } x\in[-L_M^{1/2}, L_M^{1/2}]\right\};$$
it is easy to obtain by standard Brownian bridge estimates that, on $\mrm{BdyCtrl}$, $\B[\mrm{HighCorr}\mid B(0)>t^{1/3}L] > 1-C\exp(-cM^2)$. 

Now on $\mrm{HighCorr}$ and $\mrm{BdyCtrl}$, we have $t^{-1/3}B(t^{2/3}x) \geq \hfh^\beta_{t,2}(x) + \frac{1}{4}M^2 -M$; and it is easy to check with the formula \eqref{e.rn derivative} that
$$W(B,\fh^\beta_{t,2}) > 1-Ct^{2/3}L_M^{1/2}\exp(-ct^{1/3}M^2)>1-C\exp(-cM^2).$$
Thus, these bounds on $W(B,\fh^\beta_{t,2})$ and $\B(\mrm{HighCorr}\mid B(0)>t^{1/3}L)$ yield that, on $\mrm{BdyCtrl}$,
\begin{align*}
\B\bigl(W(B,\fh^\beta_{t,2})\mid B(0)>t^{1/3}L\bigr) > 1-C\exp(-cM^{2}).
\end{align*}
So we see that, on $\mrm{BdyCtrl}(L,M)$,
\begin{align}  \label{eq:pfb}
\frac{\PF\left(\hfh^\beta_{t,1}(0)>L + \delta\right)}{\PF\left(\hfh^\beta_{t,1}(0)>L\right)} = (1+O(e^{-cM^2}))\cdot\frac{\B\bigl(B(0) > t^{1/3}(L+\delta)\bigr)}{\B\bigl(B(0) > t^{1/3}L\bigr)}.
\end{align}
Let $\mu = \frac{1}{2}(\hfh^\beta_{t,1}(-L_M^{1/2})+\hfh^\beta_{t,1}(L_M^{1/2}))$, and observe that $B(0)$ under $\B$ is distributed as a normal random variable with mean $t^{1/3}\mu$ and variance $t^{2/3}L_M^{1/2}$. 
Consider the $\F$-measurable event 
$$\mrm{MeanCtrl}(L) = \Bigl\{\mu \in [-L + 2L^{1/2}M-ML^{1/4}, -L + 2L^{1/2}M+ML^{1/4}]\Bigr\}.$$
We know from \Cref{lem:fh-tent-up} that
\begin{equation}\label{e.E complement prob bound}
\P(\mrm{MeanCtrl}(L)^c\mid \hfh^\beta_{t,1}(0)>L+\delta) < \exp(-cM^2).
\end{equation}
Using standard bounds on the tail of the normal distribution (Lemma~\ref{l.normal bounds}), on $\mrm{MeanCtrl}(L)$ (and on $\mrm{BdyCtrl}(L,M)$ as we are assuming throughout), the RHS of \eqref{eq:pfb} equals
\begin{align*}
(1+O(e^{-cM^2}))(1+O(L^{-3/2}))\cdot \frac{L-\mu}{L+\delta-\mu}\cdot \exp\left(-\frac{1}{2L_M^{1/2}}\Bigl[(L+\delta-\mu)^2 - (L - \mu)^2\Bigr]\right)
\end{align*}
and the last factor can be further written as
\begin{align*}
\exp\left(-\frac{2\delta(L-\mu) + \delta^2}{2L_M^{1/2}}\right)
&=\exp\left(-\frac{2\delta(L+L - 2L^{1/2}M+O(ML^{1/4})) + \delta^2}{2L_M^{1/2}}\right)\\
&=\exp\left(-2\delta L^{1/2} + O(\delta ML^{-1/4})\right).
\end{align*}
Overall, we have at this point established that, on $\mrm{BdyCtrl}(L,M)\cap \mrm{MeanCtrl}(L)$,
\begin{align}\label{e.comparison conditional conclusion}
\frac{\PF\left(\hfh^\beta_{t,1}(0)>L + \delta\right)}{\PF\left(\hfh^\beta_{t,1}(0)>L\right)} = (1+O(e^{-cM^2}))(1+O(L^{-3/2}))\cdot \frac{L-\mu}{L+\delta-\mu}\exp\left(-2\delta L^{1/2} +O(\delta ML^{-1/4})\right).
\end{align}
We next convert this into upper and lower bounds on $\P(\hfh^\beta_{t,1}(0)>L+\delta)$, respectively.

\medskip
\noindent\textbf{Upper bound.}
We see from \eqref{e.comparison conditional conclusion} that
\begin{align*}
\PF\left(\hfh^\beta_{t,1}(0)>L+\delta\right)
&= \PF\left(\hfh^\beta_{t,1}(0)>L+\delta\right)\left(\one_{\mrm{BdyCtrl}(L,M)\cap \mrm{MeanCtrl}(L)} + \one_{(\mrm{BdyCtrl}(L,M)\cap \mrm{MeanCtrl}(L))^c}\right)\\
&\leq (1+Ce^{-cM^2})(1+CL^{-3/2})\exp\left(-2\delta L^{1/2} + C\delta ML^{-1/4}\right)\PF\left(\hfh^\beta_{t,1}(0)>L\right)\\
&\qquad + \PF\left(\hfh^\beta_{t,1}(0)>L+\delta\right)\one_{(\mrm{BdyCtrl}(L,M)\cap \mrm{MeanCtrl}(L))^c},
\end{align*}
so that, by taking expectations,
\begin{align*}
\P\bigl(\hfh^\beta_{t,1}(0)>L+\delta\bigr) 
&\leq (1+Ce^{-cM^2})(1+CL^{-3/2})\exp\left(-2\delta L^{1/2} + C\delta ML^{-1/4}\right)\P\left(\hfh^\beta_{t,1}(0)>L\right)\\
&\qquad+ \P\Bigl(\left\{\hfh^\beta_{t,1}(0)>L+\delta\right\} \cap \bigl(\mrm{BdyCtrl}(L,M)\cap \mrm{MeanCtrl}(L)\bigr)^c\Bigr).
\end{align*}
We focus on the last term. It equals
\begin{align*}
\P\Bigl(\hfh^\beta_{t,1}(0)>L+\delta\Bigr)\cdot \P\Bigl(\bigl(\mrm{BdyCtrl}(L,M)\cap \mrm{MeanCtrl}(L)\bigr)^c\midd \hfh^\beta_{t,1}(0)>L+\delta\Bigr).
\end{align*}
If we can show that the second factor is small, we will be done. By a union bound, it is at most
\begin{align*}
\P(\mrm{BdyCtrl}(L,M)^c\mid \hfh^\beta_{t,1}(0)>L+\delta) + \P(\mrm{MeanCtrl}(L)^c\mid \hfh^\beta_{t,1}(0)>L+\delta),
\end{align*}
which, from \eqref{e.A complement prob bound} and \eqref{e.E complement prob bound}, is upper bounded by $C\exp(-cM^{3/2}) + \exp(-cM^2)$.
This completes the proof of the upper bound on $\P(\hfh^\beta_{t,1}(0)>L+\delta)$.

\medskip

\noindent\textbf{Lower bound.} We observe that, on $\mrm{MeanCtrl}$, $\frac{L-\mu}{L+\delta -\mu} = \frac{2L + O(ML^{1/2})}{2L+O(ML^{1/2}) + \delta} >1-C\delta L^{-1}$. Then we have, using \eqref{e.comparison conditional conclusion},
\begin{align*}
\MoveEqLeft[0]
\PF(\hfh^\beta_{t,1}(0)>L+\delta)\\
&\geq \PF(\hfh^\beta_{t,1}(0)>L+\delta)\cdot \one_{\mrm{BdyCtrl}(L,M)\cap \mrm{MeanCtrl}(L)}\\
&\geq (1-Ce^{-cM^2})(1-CL^{-3/2})\exp\left(-2\delta L^{1/2} -C\delta ML^{-1/4}\right)\PF(\hfh^\beta_{t,1}(0)>L)\cdot\one_{\mrm{BdyCtrl}(L,M)\cap \mrm{MeanCtrl}(L)},
\end{align*}
absorbing the factor of $1-C\delta L^{-1}$ into $\exp(-C\delta ML^{-1/4})$.
Taking expectations yields
\begin{align*}
\P\left(\hfh^\beta_{t,1}(0)>L+\delta\right)
&\geq (1-Ce^{-cM^2})(1-CL^{-3/2})\exp\left(-2\delta L^{1/2}  -C\delta ML^{-1/4}\right)\\
&\qquad\times\P\bigl(\hfh^\beta_{t,1}(0)>L, \mrm{BdyCtrl}(L,M), \mrm{MeanCtrl}(L)\bigr)\\
&\geq (1-Ce^{-cM^2})(1-CL^{-3/2})\exp\left(-2\delta L^{1/2}  -C\delta ML^{-1/4}\right)\\
&\qquad\times\P\bigl(\mrm{BdyCtrl}(L,M), \mrm{MeanCtrl}(L)\mid \hfh^\beta_{t,1}(0)>L\bigr)\cdot \P\left(\hfh^\beta_{t,1}(0)>L\right).
\end{align*}
As we saw above, the latter conditional probability is lower bounded by $1-C\exp(-cM^{3/2})$, thus completing the proof of the lower bound.
\end{proof}

\section{Tightness as continuous functions: geodesics and bounds for polymers}  \label{sec:tight}

We next establish the following tightness of the relevant path measures.

As in \Cref{thm:main-dl,thm:main-dp}, let $\pi_0$ be the geodesic from $(0,0)$ to $(0,1)$, in the directed landscape $\cL^{\beta=\infty}$; and $\Gamma_0$ be sampled from the annealed polymer measure from $(0,0)$ to $(0,1)$, under $\cL^{\beta=1}$.
\begin{proposition}[Tightness]\label{p.tightness}
As random elements in $\mc C([0,1], \R)$, $L^{1/4}\pi_0$ or $L^{1/4}\Gamma_0$ conditional on $\cL^\beta(0,0;0,1) > L$ for all $L\ge 2$ are tight.
\end{proposition}

As mentioned in \Cref{iop}, to prove this requires tail bounds on two-point deviations, which rely on shear invariance of the directed landscape and the CDRP free energy field. The proof in zero temperature is much simpler than in positive temperature as in the former shear invariance alone suffices to give tightness on the $L^{-1/4}$ scale. In positive temperature, the analogous argument only yields tightness on the $O(1)$ scale, and additional arguments are needed to obtain the correct scale. The reason is that in zero temperature, given the environment, the path location at a given height is determined; while in positive temperature there is no a priori concentration of the polymer location.  

In this section, we give the zero temperature proof
and some rough bounds for the transversal fluctuation of polymers.
The positive temperature part of \Cref{p.tightness} will be proved in \Cref{sec:pos_tight}.

\subsection{Geodesic tightness}\label{s.tightness.zero temp}
We start with a one-point estimate and later give the two-point estimate.
\begin{lemma}  \label{l.trans-fluc}
For all $K>0$, $L\ge 2$, and $s\in(0,1)$,
\begin{align*}
\PP\left(|\pi_0(s)| > K(s(1-s))^{1/2}L^{-1/4} \midd \cL(0,0;0,1)>L\right)
< C\exp\left(-cK^2\right).
\end{align*}
\end{lemma}

In fact, for $K$ up to $L^{1/2}$ we will obtain a tail of $\exp(-2K^2)$ (via the first case of the comparison statement Theorem~\ref{t.comparison}). Note that it exactly corresponds to our ultimate goal, namely that $L^{-1/4}\pi_0(s)$ converges to $\frac{1}{2}\mathcal{N}(0, s(1-s))$, which on scale $(s(1-s))^{1/2}$ would have a tail at depth $x$ that satisfies the approximate asymptotics of $\exp(-2x^2)$.

\begin{proof}[Proof of Lemma~\ref{l.trans-fluc}]
Observe that, for any $K$, $\P(|\pi_0(s)| > K \mid \cL(0,0;0,1)>L)$ is upper bounded by
\begin{align}
\MoveEqLeft[10]
  \P\left(\sup_{|x| > K} \Bigl(\cL(0,0; x,s) + \cL(x,s; 0,1)\Bigr) > L \ \Big|\  \cL(0,0;0,1)>L\right)\nonumber\\
  &\leq 2\cdot \P\left(\sup_{x > K} \Bigl(\cL(0,0; x,s) + \cL(x,s; 0,1)\Bigr) > L \ \Big|\  \cL(0,0;0,1)>L\right)\nonumber\\
  &\leq 2\cdot \frac{\P\left(\sup_{x > K} \Bigl(\cL(0,0; x,s) + \cL(x,s; 0,1)\Bigr) > L\right)}{\P(\cL(0,0; 0,1) > L)}, \label{e.crude tf bound step}
\end{align}
the factor of 2 coming from a union bound and using the distributional symmetry of $\cL(0,0; x, s)$ and $\cL(x, s; 0,1)$ under $x\mapsto -x$ from Lemma~\ref{l.dl symmetries} (as well as the independence of the two processes) to remove the absolute value in the condition under the supremum.
Then using the shear invariance and independence properties of $\cL$ and that $x>K$ for the inequality in the second line,
\begin{align*}
\cL(0,0; x,s) + \cL(x,s; 0,1) &\stackrel{d}{=} \cL(0,0; x-K, s) + \cL(x-K, s; 0,1) + (s(1-s))^{-1}\left[(x-K)^2 - x^2\right]\\
&\leq \cL(0,0; x-K, s) + \cL(x-K, s; 0,1) - (s(1-s))^{-1}K^2
\end{align*}
as a process in $x$. Thus we see that the RHS of \eqref{e.crude tf bound step} is upper bounded by
\begin{align*}
\frac{\P\left(\sup_{x > 0} \Bigl(\cL(0,0; x,s) + \cL(x,s; 0,1)\Bigr) > L + (s(1-s))^{-1}K^2\right)}{\P(\cL(0,0; 0,1) > L)}.
\end{align*}
Now using that $\cL(0,0; 0,1) = \sup_{x\in\R}\Bigl(\cL(0,0; x,s) + \cL(x,s; 0,1)\Bigr)$ from \eqref{e.compopsition law zero temp}, it follows that the previous display is upper bounded by
\begin{align*}
\frac{\P\left(\cL(0,0; 0,1) > L + (s(1-s))^{-1}K^2\right)}{\P(\cL(0,0; 0,1) > L)}.
\end{align*}
Applying Theorem~\ref{t.comparison} now bounds the previous display by 
$$C\exp\left(-cK^2(s(1-s))^{-1}L^{1/2}\right).$$
Replacing $K$ by $K(s(1-s))^{1/2}L^{-1/4}$ completes the proof. 
\end{proof}

We next give a two-point estimate. 
Combining it with the Kolmogorov-Chentsov criterion for tightness (see e.g.~\cite[Theorem~23.7]{Kallenberg}), the $\beta=\infty$ case of \Cref{p.tightness} follows.
\begin{proposition}\label{p.zero temperature two-point}
For all $K>0$,  $L\ge 2$, and $0<s<t<1$,
\begin{align*}
\P\left(|\pi_0(s)-\pi_0(t)| > K(t-s)^{1/2}L^{-1/4} \mid \cL(0,0;0,1)>L\right) \leq C\exp(-cK^2).
\end{align*}
\end{proposition} 

\begin{proof}
We give the proof under the assumption that $t-s\in(0,\frac{1}{2})$. The case where $t-s\in[\frac{1}{2}, 1]$ follows from Lemma~\ref{l.trans-fluc} easily.

Similar to the previous proof, the LHS of the display in the lemma is upper bounded by
\begin{align}
\MoveEqLeft[6]
  \P\left(\sup_{|x-y| > K} \Bigl(\cL(0,0; x,s) + \cL(x,s; y,t) + \cL(y,t; 0,1)\Bigr) > L \ \Big|\  \cL(0,0;0,1)>L\right)\nonumber\\
  &\leq 2\cdot \P\left(\sup_{x-y > K} \Bigl(\cL(0,0; x,s) + \cL(x,s; y,t) + \cL(y,t; 0,1)\Bigr) > L \ \Big|\  \cL(0,0;0,1)>L\right)\nonumber\\
  &\leq 2\cdot\frac{\P\left(\sup_{x-y > K} \Bigl(\cL(0,0; x,s) + \cL(x,s; y,t) + \cL(y,t; 0,1)\Bigr) > L\right)}{\P\left(\cL(0,0;0,1)>L\right)}.\label{e.crude tf bound step t-s}
\end{align}
Now, using the stationarity (and independence) properties of $\cL$,
\begin{align*}
&\cL(0,0; x,s) + \cL(x,s; y,t) + \cL(y,t; 0,1) 
\\ &\stackrel{d}{=} \cL(-K,0; x-K,s) + \cL(x-K,s; y,t) + \cL(y,t; 0,1)
 + (t-s)^{-1}\left[(x-y-K)^2-(x-y)^2\right]
\end{align*}
as a process in $(x,y)$. Since $x-y>K$, we have that
\[
(t-s)^{-1}\left[(x-y-K)^2-(x-y)^2\right] 
< -(t-s)^{-1}K^2.
\]
Thus we see that the RHS of \eqref{e.crude tf bound step t-s} is upper bounded by
\begin{align*}
2\cdot\frac{\P\left(\sup_{x-y > K} \Bigl(\cL(-K,0; x-K,s) + \cL(x-K,s; y,t) + \cL(y,t; 0,1)\Bigr) > L + (t-s)^{-1}K^2\right)}{\P\left(\cL(0,0;0,1)>L\right)}
\end{align*}
Now using that $\cL(-K,0; 0,1) = \sup_{x, y\in\R}\Bigl(\cL(-K,0; x,s) + \cL(x,s; y,t) + \cL(y,t; 0,1)\Bigr)$, and that $\cL(-K,0; 0,1)\stackrel{d}{=} \cL(0,0;0,1) - K^2$, it follows that the previous display is upper bounded by
\begin{align*}
\frac{\P\left(\cL(-K,0; 0,1) > L + (t-s)^{-1}K^2\right)}{\P(\cL(0,0; 0,1) > L)}
&= \frac{\P\left(\cL(0,0; 0,1) > L + ((t-s)^{-1}+1)K^2\right)}{\P(\cL(0,0; 0,1) > L)}.
\end{align*}
By Theorem~\ref{t.comparison}, this is upper bounded by
\begin{align*}
C\exp\left(-c(t-s)^{-1}K^2L^{1/2}\right).
\end{align*}
Replacing $K$ by $K(t-s)^{1/2}L^{-1/4}$ completes the proof.
\end{proof}

\subsection{Polymer transversal estimates}
\label{s.tightness.positive temp}
We now adapt the zero temperature arguments above to the positive temperature setting. 
Although the bounds in this subsection are not sufficient to derive the $\beta=1$ case of \Cref{p.tightness}, they will be used in the proof to be given in \Cref{sec:pos_tight}.

We start with the positive temperature analog of \Cref{l.trans-fluc}.
\begin{lemma}\label{l.positive temp crude tf}
For all $K>0$, $L\ge 2$ and $s\in(0,1)$, 
\begin{align*}
\P\left(\polyP\left[|\Gamma_0(s)| \geq K(s(1-s))^{1/2}\right] > \exp(-\tfrac{1}{2}K^2) \midd \cL^{\beta}(0,0;0,1) > L\right) < C\exp(-cK^2L^{1/2}).
\end{align*}
\end{lemma} 

\remark{Unlike the zero-temperature case, here the concentration is on an $O(1)$ scale. Some more work is needed to obtain an $L^{-1/4}$ scale concentration, and we will turn to that shortly.}

\begin{proof}[Proof of Lemma~\ref{l.positive temp crude tf}]
By the convolution formula, for any $K>0$ and $\varepsilon>0$,
\begin{align}
\MoveEqLeft[4]
\P\left(\polyP\left[|\Gamma_0(s)| \geq K\right] > \varepsilon \mid \cL^\beta(0,0;0,1) > L\right)\nonumber\\
&= \P\left(\int_{|x|\geq K} \mc Z(0,0;x,s)\mc Z(x,s; 0,1)\diff x > \varepsilon \mc Z(0,0;0,1)\midd \cL^\beta(0,0;0,1) > L\right)\nonumber\\
&\leq \frac{\P\left(\int_{|x|\geq K} \mc Z(0,0;x,s)\mc Z(x,s; 0,1)\diff x > \varepsilon e^{L-1/12}\right)}{\P\left(\cL^\beta(0,0;0,1) > L\right)} \nonumber\\
&\leq  \frac{2\cdot \P\left(\int_{x\geq K} \mc Z(0,0;x,s)\mc Z(x,s; 0,1)\diff x > \tfrac{1}{2}\varepsilon e^{L-1/12}\right)}{\P\left(\cL^\beta(0,0;0,1) > L\right)}, \label{e.convolution formula for tf}
\end{align}
where the last inequality uses the reflection symmetry of $\mc Z(0,0; \cdot, s)$ and $\mc Z(\cdot, s; 0,1)$ and their independence. By the shear invariance and independence of the same two processes, we have that the following distributional equality holds as processes in $x$ for every fixed $s$:
\begin{align*}
\mc Z(0,0;x,s)\mc Z(x,s; 0,1)
&\stackrel{d}{=} \mc Z(0,0;x-K,s)\mc Z(x-K,s; 0,1) e^{(s(1-s))^{-1}[(x-K)^2-x^2]}\\
&\leq \mc Z(0,0;x-K,s)\mc Z(x-K,s; 0,1) e^{-(s(1-s))^{-1}K^2},
\end{align*}
where the inequality is due to that $x\ge K$. Substituting this into \eqref{e.convolution formula for tf} yields that 
\begin{align*}
\MoveEqLeft[6]
\P\left(\polyP\left[|\Gamma_0(s)| \geq K\right] > \varepsilon \mid \cL^\beta(0,0;0,1) > L\right)\\
&\leq \frac{2\cdot \P\left(\int_{x\geq 0} \mc Z(0,0;x,s)\mc Z(x,s; 0,1)\diff x > \tfrac{1}{2}\varepsilon e^{L-1/12 + (s(1-s))^{-1}K^2}\right)}{\P\left(\cL^\beta(0,0;0,1) > L\right)}\\
&\leq \frac{2\cdot \P\left(\int_{\R} \mc Z(0,0;x,s)\mc Z(x,s; 0,1)\diff x > \tfrac{1}{2}\varepsilon e^{L-1/12 + (s(1-s))^{-1}K^2}\right)}{\P\left(\cL^\beta(0,0;0,1) > L\right)}\\
&= \frac{2\cdot\P\left(\cL^{\beta}(0,0;0,1) > L + (s(1-s))^{-1}K^2 +\log(\varepsilon/2)\right)}{\P\left(\cL^{\beta}(0,0;0,1) > L\right)}
\end{align*}
We set $\varepsilon = \exp\left(-(s(1-s))^{-1}K^2/2\right)$ and invoke Theorem~\ref{t.comparison} to obtain that the previous display is upper bounded by
\begin{align*}
C\exp\left(-cL^{1/2}(s(1-s))^{-1}K^2\right).
\end{align*}
Replacing $K$ by $K(s(1-s))^{1/2}$ completes the proof.
\end{proof} 

We next derive a two-point estimate.

\begin{lemma}\label{l.positive temp two point crude bound}
For all $K>0$, $L\ge 2$, and $0<s<t<1$,
\begin{align*}
\P\left(\polyP\left[|\Gamma_0(s)-\Gamma_0(t)| > K(t-s)^{1/2}\right] > \exp(-\tfrac{1}{2}K^2) \midd \cL^{\beta}(0,0;0,1) > L\right) < C\exp(-cK^2L^{1/2}).
\end{align*}
\end{lemma} 

\begin{proof}
We give the proof under the assumption that $t-s\in(0,\frac{1}{2})$, since case where $t-s\in[\frac{1}{2}, 1]$ follows from \Cref{l.positive temp crude tf} easily. Observe that, for any $\varepsilon >0$,
\begin{align}
\MoveEqLeft
\P\left(\polyP\left(|\Gamma_0(s) - \Gamma_0(t)| > K\right) \geq \varepsilon \midd \cL^\beta(0,0;0,1)>L\right)\nonumber\\
&= \P\left(\int_{|x-y|\geq K} \mc Z(0,0;x,s)\mc Z(x,s; y,t)\mc Z(y,t; 0,1)\diff x \diff y > \varepsilon \mc Z(0,0;0,1)\midd \cL^\beta(0,0;0,1) > L\right)\nonumber\\
&\leq \frac{\P\left(\int_{|x-y|\geq K} \mc Z(0,0;x,s)\mc Z(x,s; y,t)\mc Z(y,t; 0,1)\diff x  \diff y > \varepsilon e^{L-1/12}\right)}{\P\left(\cL^\beta(0,0;0,1) > L\right)} \nonumber\\
&\leq  \frac{2\cdot \P\left(\int_{x-y\geq K} \mc Z(0,0;x,s)\mc Z(x,s; y,t)\mc Z(y,t; 0,1)\diff x \diff y > \tfrac{1}{2}\varepsilon e^{L-1/12}\right)}{\P\left(\cL^\beta(0,0;0,1) > L\right)},\label{e.two-point tf positive temp derivation step}
\end{align}
where the factor of 2 comes from removing the absolute value in the condition under the supremum by a union bound and using that $\cL^\beta(x,s; y, t) \stackrel{d}{=}\cL^\beta(-x, s; -y,t)$, $\cL^\beta(0,0;x,s) \stackrel{d}{=}\cL^\beta(0,0; -x, s)$, and $\cL^\beta(y, t; 0,1) \stackrel{d}{=}\cL^\beta(-y,t; 0,1)$  each as processes in the relevant spatial variables, as well as the independence of all three processes on the LHS of the equalities.

Now, using the stationarity (and independence) properties of $\cL^\beta$,
\begin{align*}
\cL^\beta(0,0; x,s) + \cL^\beta(x,s; y,t) + \cL^\beta(y,t; 0,1) 
&\stackrel{d}{=} \cL^\beta(-K,0; x-K,s) + \cL^\beta(x-K,s; y,t) + \cL^\beta(y,t; 0,1)\\
&\qquad + (t-s)^{-1}\left[(x-y-K)^2-(x-y)^2\right]
\end{align*}
as a process in $(x,y)$. Now since $x-y>K$, we see that
\begin{align*}
(t-s)^{-1}\left[(x-y-K)^2-(x-y)^2\right] 
< -(t-s)^{-1}K^2.
\end{align*}
Thus the RHS of \eqref{e.two-point tf positive temp derivation step} is upper bounded by
\begin{align}\label{e.positive temp two point bound intermediate step}
2\cdot\frac{\P\left(\int_{x-y > K} \mc Z(-K,0; x-K,s)\mc Z(x-K,s; y,t)\mc Z(y,t; 0,1) \diff x \diff y > \frac{1}{2}\varepsilon e^{L-1/12 + (t-s)^{-1}(K)^2}\right)}{\P\left(\cL^\beta(0,0;0,1)>L\right)}.
\end{align}
Now using that 
$$\cL^\beta(-K,0; 0,1)-1/12 = \log \int_{x, y\in\R}\mc Z(-K,0; x-K,s)\mc Z(x-K,s; y,t)\mc Z(y,t; 0,1)\diff x \diff y,$$
and that $\cL^\beta(-K,0; 0,1)\stackrel{d}{=} \cL^\beta(0,0;0,1) - K^2$, it follows that \eqref{e.positive temp two point bound intermediate step} is upper bounded by
\begin{align*}
2\cdot \frac{\P\left(\cL^\beta(0,0; 0,1) > L + ((t-s)^{-1}+1)K^2+\log(\varepsilon/2)\right)}{\P(\cL^\beta(0,0; 0,1) > L)}.
\end{align*}
Set $\varepsilon = \exp(-(t-s)^{-1}K^2/2)$. Applying Theorem~\ref{t.comparison} gives that the previous display is upper bounded by
\begin{align*}
C\exp\left(-c(t-s)^{-1}K^2L^{1/2}\right).
\end{align*}
Replacing $K$ by $K(t-s)^{1/2}$ completes the proof.
\end{proof}

\section{Proportionality and estimates on sums}  \label{sec:propsum}
In this section, we record a number of estimates on the sum of passage times or free energies.
Consider $(s_1, \cdots, s_{k})\in \rDe_{k}([0,\infty))$ and $\vec y\in\R^k$ for $k\in \N$.
For the convenience of notations we denote $s_0=y_0=0$, and adopt the shorthand $\cL^\beta = \cL^\beta(0,0;0,s_k)$ and $\cL^\beta_i = \cL^\beta(s_{i-1}, y_{i-1}; s_i, y_i)$ for each $1\le i \le k$.
Estimates in this section provide control on tail probabilities for the sum $\sum_{i}\cL^\beta_i$ (such as \Cref{l.free energy individual to overall values}, \Cref{l.inf control}) or on the deviation of each $\cL^\beta_i$ from $(s_i-s_{i-1})L$ conditional on $\sum_{i}\cL^\beta_i > L$ (Lemma~\ref{l.k-point proportionality}).

We assume that $s_k<C_0$ for a large $C_0$. All the constants (within this section) can depend on $k$ and $C_0$.

Our first statement bounds the upper tail of $\sum_{i}\cL^\beta_i$, in terms of the upper tail of $\cL^\beta$.
\begin{lemma}  \label{l.free energy individual to overall values}
For any $L\geq 2$, any $M$, and any $\vec z\in \R^k$, if $\min_{1\le i \le k}s_i-s_{i-1}> L^{-1}$, we have
\[
\P\left(\sup_{\vec y: \|\vec y-\vec z\|_\infty\le L^{-2}}\sum_{i=1}^k\cL^\beta_i >M, \; \cL^\beta(0,0;z_k,s_k) < M - C\log L\right)< C\exp(-cL^2).
\]
\end{lemma}

\begin{proof}
This is immediate in the case of $\beta=\infty$ by subadditivity, i.e., $\cL^\beta(0,0; y_k, s_k) \geq \sum_{i=1}^k \cL^\beta_i$ and the unconditional local fluctuation bound (Lemma~\ref{lem:fh-int-ub}) to obtain that $|\cL^\beta(0,0; y_k, s_k)-\cL^\beta(0,0; z_k, s_k)| \leq 1$ with probability at least $1-\exp(-cL^2)$. 

We next turn to $\beta=1$. 
We note from the unconditional fluctuation bound (\Cref{lem:fh-int-ub}) that, for any fixed intervals $I_1, \ldots, I_{k}$, each with length $L^{-2}$, and each $z_i\in I_i$,
$$\inf_{\vec y\in \prod_{i=1}^{k} I_i}\sum_{i=1}^k\cL^\beta_i > \sup_{\vec y: \|\vec y-\vec z\|_\infty\le cL^{-2}}\sum_{i=1}^k\cL^\beta_i - 1,$$
with probability at least $1-C\exp(-cL^2)$. 
Under this event, and assuming that $\sup_{\vec y: \|\vec y-\vec z\|_\infty\le L^{-2}}\sum_{i=1}^k\cL^\beta_i >M$, we have (with $y_k=z_k$)
\begin{align*}
\cL^\beta(0,0;z_k,s_k)
&\geq \log \int_{\prod_{i=1}^{k-1} I_i} \exp\left(\sum_{i=1}^k \cL^\beta_i\right)\, \prod_{y=1}^{k-1}\diff y_i\\
&> M-1 +(k-1)\log(L^{-2}).
\end{align*}
This completes the proof.
\end{proof}
By taking $L=M$ and using \Cref{lem:fh-ut}, we get the following.
\begin{corollary}   \label{cor.p.tail bound for sum}
For any large enough $L$, and any $\vec z\in \R^k$, if $\min_{1\le i \le k}s_i-s_{i-1}> L^{-0.99}$,
\[
\P\left(\sup_{\vec y: \|\vec y-\vec z\|_\infty\le L^{-2}}\sum_{i=1}^k\cL^\beta_i >L\right) < \exp\left(-\frac{4}{3}s_k^{-1/2}L^{3/2} + Cs_k^{-1/4}L^{3/4}\right).
\]
\end{corollary}
We note that the constraint $\min_{1\le i \le k}s_i-s_{i-1}> L^{-0.99}$ is to ensure that the bound is much larger than $\exp(-cL^2)$, the upper bound in \Cref{l.free energy individual to overall values}.

The following statement asserts that, conditional on the sum of independent free energies being large, the individual terms are with high probability proportionate to the total, up to a certain scale of fluctuation.

\begin{lemma}\label{l.k-point proportionality}
Fix each $y_i=0$. 
For any $L$ large enough, $K>C_1L^{3/8}$ for a constant $C_1>0$ (so that $KL^{1/4} > C_1L^{5/8}$), if $\min_{1\le i \le k}s_i-s_{i-1}> L^{-0.99}$, then for each $j=1,\ldots, k$,
\begin{align*}
\P\left(\cL^\beta_j < (s_j-s_{j-1})L - K(s_j-s_{j-1})^{1/2}L^{1/4} \midd \sum_{i=1}^{k}\cL^\beta_i > s_kL\right) < \exp(-cK^2).
\end{align*}
\end{lemma}

The above bound is optimal except for the fact that we require $K>CL^{3/8}$, while it should hold for $K>C$; the loss is due to the non-optimal error term in our tail bound \Cref{lem:fh-ut}.
The constraint $\min_{1\le i \le k}s_i-s_{i-1}> L^{-0.99}$ is due to applying \Cref{lem:fh-ut} and \Cref{cor.p.tail bound for sum}. We also mention the related result \cite[Theorem 1.1]{liu2022geodesic} that implies that, conditional on $\cL(0,0;0,1) = L$, the weight of the geodesic up to height $s$ is $sL$ plus a random term on scale $L^{1/4}$ which, when scaled to be unit order, converges to a Gaussian random variable.

\begin{proof}[Proof of Lemma~\ref{l.k-point proportionality}]
By the independence of these $\cL_j^\beta$, without loss of generality, we prove the estimate for $j=k$.

Let $Y = \sum_{i=1}^{k-1}\cL^\beta_i$, and $X = \cL^\beta_k-(s_k-s_{k-1})L$. Then the probability in the LHS equals
\begin{align}
\MoveEqLeft[2]
 \sum_{\ell=K(s_k-s_{k-1})^{1/2}L^{1/4}}^\infty \P\left(X \in - \ell +[-1,0] \midd \sum_{i=1}^{k}\cL^\beta_i > s_kL\right) \nonumber\\
&= \frac{\sum_{\ell=K(s_k-s_{k-1})^{1/2}L^{1/4}}^\infty \P\left(\sum_{i=1}^{k}\cL^\beta_i > s_kL \midd X \in - \ell +[-1,0]\right)\cdot \P\left( X \in -\ell +[-1,0]\right)}{\P\left(\sum_{i=1}^{k}\cL^\beta_i > s_kL\right)} \nonumber\\
&\leq \frac{\sum_{\ell=K(s_k-s_{k-1})^{1/2}L^{1/4}}^\infty \P\left(Y > s_{k-1}L + \ell  \midd X \in - \ell +[-1,0]\right)\cdot \P\left( X \in -\ell +[-1,0]\right)}{\P\left(\sum_{i=1}^{k}\cL^\beta_i > s_kL\right)}. \label{e.k-point proportionality derivation}
\end{align}
We can lower bound the denominator by $\prod_{i=1}^k \P(\cL^\beta_i>(s_i-s_{i-1})L)$, which, by \Cref{lem:fh-ut}, is $>\exp\left(- \frac{4}{3}s_kL^{3/2} - Cs_k^{1/2}L^{3/4}\right)$.
For the numerator, note that $Y$ is independent of $X$, so the summand indexed by $\ell$ in \eqref{e.k-point proportionality derivation} can be upper bounded, using \Cref{cor.p.tail bound for sum} and Theorem~\ref{lem:fh-ut} by
\begin{align*}
\MoveEqLeft[0.5]
\exp\left(-\tfrac{4}{3}s_{k-1}^{-1/2}\left(s_{k-1}L+\ell\right)^{3/2} -\tfrac{4}{3}(s_k-s_{k-1})^{-1/2}\left((s_k-s_{k-1})L - \ell\right)^{3/2} + Cs_k^{1/2}L^{3/4}\right)\\
&\le  \exp\left(-\tfrac{4}{3}s_kL^{3/2} - c\ell^2(s_k-s_{k-1})^{-1}L^{-1/2} + Cs_k^{1/2}L^{3/4}\right).
\end{align*}
Substituting this into \eqref{e.k-point proportionality derivation} and using that $K> C_1L^{3/8}$ for a large enough constant $C_1$ (so that $cK^2$ is much larger than $s_k^{1/2}L^{3/4}$) yields the claim.
\end{proof} 

The following lemma provides control on the lower tail of the sum $\sum_{i=1}^k \cL^\beta_i$, conditional on the upper tail of $\cL^\beta$.

\begin{lemma}\label{l.inf control}
Assume that $\min_{1\le i \le k}s_i-s_{i-1}> t_0$ for some $t_0>0$.
For any $L>0$,
\[
\P\left(\inf_{\vec y: \|\vec y\|_{\infty} \leq L^{-1/4}\log L} \sum_{i=1}^{k}\cL^\beta_i < L - L^{5/8}\log L \midd \cL^\beta>L\right)\\
< C\exp(-c(\log L)^2),
\]
where the constants $C,c$ may depend on $t_0$.
\end{lemma}

\begin{proof}
We assume that $L$ is large enough since otherwise the conclusion follows obviously.

We start with $\beta=1$. We first claim that, conditional on $\cL^\beta > L$, it holds with probability at least $1-C\exp(-c(\log L)^2)$ that $\sup_{\vec y : \|\vec y\|_{\infty} \leq L^{-1/4}\log L} \sum_{i=1}^{k}\cL^\beta_i > L-C\log L$. Indeed, suppose that this inequality does not hold. In the $\beta=1$ case, we know from the quenched transversal fluctuation estimate Lemma~\ref{l.positive temp crude tf} that with conditional probability at least $1-C\exp(-c(\log L)^2)$,
\begin{align*}
\exp(L) 
&< \left(1-e^{-cL^{-1/2}(\log L)^2}\right)^{-1} \int_{[-L^{-1/4}\log L,L^{-1/4}\log L]^{k-1}} \exp\left(\sum_{i=1}^{k} \cL^\beta_i\right) \prod_{i=1}^{k-1}\diff y_i,
\end{align*}
where $y_k=0$ in the integral.
If $\sup_{\vec y : \|\vec y\|_{\infty} \leq L^{-1/4}\log L} \sum_{i=1}^{k}\cL^\beta_i < L-C\log L$, the RHS is upper bounded by
\begin{align*}
C L^{1/2}(\log L)^{-2}(2L^{-1/4}\log L)^k \exp\left(L - C\log L\right) \ll \exp(L),
\end{align*}
which is a contradiction.

In the $\beta = \infty$ case, by Lemma~\ref{l.trans-fluc}, with conditional probability at least $1-C\exp(-c(\log L)^2)$ it holds that $\cL^{\beta} = \sup_{\vec y : \|\vec y\|_{\infty} \leq L^{-1/4}\log L} \sum_{i=1}^{k}\cL^\beta_i$, which implies our claim since we have conditioned on $\cL^{\beta}>L$.

Next, we know from \Cref{lem:fh-int-ub} that, with (unconditional) probability at least $1-C\exp(-cL^{3/2}\log L)$,
\[
\sup_{\|\vec y\|_{\infty} \leq L^{-1/4}\log L}\sum_{i=1}^k\left|\cL^\beta_i - \cL^\beta(0, s_{i-1}; 0, s_i)\right|
\leq \tfrac{1}{2}L^{3/4}(\log L)^{1/2}\cdot (L^{-1/4}\log L)^{1/2} = \tfrac{1}{2}L^{5/8}\log L.
\]
By Theorem~\ref{lem:fh-ut} we know $\P(\cL^\beta > L) > \exp(-CL^{3/2})$, the previous bound also holds conditionally on $\cL^\beta > L$ with probability at least $1-C\exp(-cL^{3/2}\log L)$. This completes the proof.
\end{proof}

\section{Concentration of polymers}\label{s.tightness.polymer measure concentration}

In this section, we mainly work with polymers, i.e., set $\beta=1$. We prove the following fact: the polymer measure at a given height $s$ is a Dirac mass spread out over an $L^{-1/2}$ interval around a random location.
More specifically, we define 
$$\pi(s) = \argmax_{x\in\R} \cL^\beta(0,0;x,s) + \cL^\beta(x,s; 0,1),$$
for any $s\in(0,1)$.
Note that this is a function purely of the environment. 
The main result of this section asserts that under the polymer measure $\polyP$, the polymer path at each height in $[L^{-5/8}, 1-L^{-5/8}]$ stays within $O(L^{-1/2}\log L)$ of $\pi$ with high probability.
\begin{proposition}\label{p.closeness of maximizer and polymer}
There exists $M_0>0$ such that for any $L\ge 2$, $M>M_0$ and $s\in[L^{-5/8}, 1-L^{-5/8}]$, 
\begin{align*}
\P\left(\polyP\left(|\Gamma_0(s)-\pi(s)| > ML^{-1/2}\log L\right) > L^{-2M} \midd \cL^{\beta}(0,0;0,1) > L\right) < C\exp(-c(\log(L))^2).
\end{align*}
\end{proposition}

In the rest of this section, we shall always take $s\in [L^{-5/8}, 1-L^{-5/8}]$.
For the convenience of notations, we adopt the notation $s_1=s$ and $s_2=1-s$, $\cL^\beta_1(x) = \cL^\beta(0,0;x,s)$ and $\cL^\beta_2(x) = \cL^\beta(x,s; 0,1)$, and $\cL^\beta=\cL^\beta(0,0;0,1)$.

\subsection{Global to two-segments conditioning}
In the upcoming proof of Proposition~\ref{p.closeness of maximizer and polymer}, it will be useful to go from conditioning on $\cL^\beta > L$ to conditioning on $\cL^{\beta}_i(x)\in (h_i, h_i+\diff h_i)$ for each $i=1,2$, with $x, h_1, h_2$ belonging to a set of nice values. The following is the general statement that allows us to do this.

Let $\mrm{Val}\subseteq\R^2$ be defined by
\begin{align*}
\mrm{Val} = \left\{(h_1, h_2) \in \R^2: L-(\log L)^2<h_1+h_2<L+(\log L)^2, h_i > s_iL-s_i^{1/2}L^{5/8}\log L \text{ for } i=1,2\right\}.
\end{align*}
\begin{lemma}[Global to two-segments conditioning]\label{l.global conditioning to two}
For any $L>0$, $0<K\le \log L$, and any event $A$,
\begin{multline*}
\P\left(A\mid \cL^\beta > L\right)
< \max_{\substack{x\in L^{-2}\Z, |x|\leq KL^{-1/4}}} \sup_{(h_1, h_2) \in \mrm{Val}} \P\left(A \midd \cL^\beta_i(x)\in (h_i, h_i+\diff h_i), i=1,2\right) \\ \times C\exp( C(\log L)^2L^{1/2} ) + C\exp(-cK^2).
\end{multline*}
\end{lemma}

\begin{proof}
We assume that $K>1$ since otherwise the conclusion follows trivially by setting $C$ large.

We observe from Lemma~\ref{l.positive temp crude tf} that, with probability at least $1-C\exp(-cK^2)$, conditionally on $\cL^\beta > L$,
\begin{align*}
\int_{[- KL^{-1/4},  KL^{-1/4}]} \exp\left((\cL^\beta_1+ \cL^\beta_2)(x)\right)\diff x
&\geq \left(1-\exp\left(-\tfrac{1}{2}K^2L^{-1/2}\right)\right)\exp(L)\\
&\geq \exp\left(L + \log\left(\tfrac{1}{2}K^2L^{-1/2}\right)\right),
\end{align*}
which implies that
\begin{align*}
\P\left(\max_{[-KL^{-1/4}, KL^{-1/4}]} \cL^\beta_1+ \cL^\beta_2 \geq L - \log L \midd \cL^\beta > L\right) >1-C\exp(-cK^2).
\end{align*}
This with unconditional local fluctuation estimates (\Cref{lem:fh-cont}) implies that
\[
\P\left( \max_{x\in L^{-2}\Z, |x|\le KL^{-1/4}}(\cL^\beta_1+\cL^\beta_2)(x)\le L- (\log L)^2\midd \cL^\beta>L\right) <C\exp(-cL^2(\log L)^4) + C\exp(-cK^2).
\]
Also, for each $|x|\le KL^{-1/4}$, 
by \Cref{t.comparison} and \Cref{l.free energy individual to overall values},
\[
\P\left((\cL^\beta_1+\cL^\beta_2)(x) \ge L+(\log L)^2 \mid \cL^\beta > L \right) < C\exp(-c(\log L)^2L^{1/2});
\]
and by proportionality (\Cref{l.k-point proportionality}), we have
\begin{multline*}
\P\left(\cL^\beta_i(x) \le s_iL - s_i^{1/2}L^{5/8}\log L, \;(\cL^\beta_1+\cL^\beta_2)(x)>L- (\log L)^2 \midd \cL^\beta > L \right) \\
<\P\left(\cL^\beta_i(x) \le s_iL - s_i^{1/2}L^{5/8}\log L \midd (\cL^\beta_1+\cL^\beta_2)(x)> L- (\log L)^2\right)\frac{\P\left( (\cL^\beta_1+\cL^\beta_2)(x)> L- (\log L)^2\right)}{\P(\cL^\beta>L)}\\<
C\exp(-c(\log L)^2L^{3/4}),    
\end{multline*}
for each $i=1, 2$, where the ratio is bounded using \Cref{cor.p.tail bound for sum} and \Cref{lem:fh-ut}.
Combining the above three estimates, we see that with probability $>1-C\exp(-cK^2)$, there exists one $x\in L^{-2}\Z, |x|\le KL^{-1/4}$, such that $\mrm{ValueCtrl}_x$ holds; where
\begin{align*}
\mrm{ValueCtrl}_x = \left\{ L- (\log L)^2<(\cL^\beta_1+\cL^\beta_2)(x) < L+(\log L)^2, \cL^\beta_i(x) > s_iL - s_i^{1/2}L^{5/8}\log L \text{ for } i=1,2\right\}.
\end{align*}
In other words, we have
\[
    \P(\mrm{ValueCtrl}^c \mid \cL^\beta > L)<C\exp(-cK^2),
\]
where $\mrm{ValueCtrl}=\cup_{x\in L^{-2}\Z, |x|\le KL^{-1/4}}\mrm{ValueCtrl}_x$.
Now we have
\[
\P( A \mid \cL^\beta>L) < \P(A\cap \mrm{ValueCtrl} \mid \cL^\beta > L)+ C\exp(-cK^2).
\]
Note that the first term in the RHS is bounded by
\[
\max_{\substack{x\in L^{-2}\Z, |x|\leq KL^{-1/4}}} \sup_{(h_1, h_2) \in \mrm{Val}} \P\left(A \midd \cL^\beta_i(x)\in (h_i, h_i+\diff h_i), i=1,2\right) \frac{\P(\mrm{ValueCtrl})}{\P(\cL^\beta > L)}.
\]
By \Cref{l.free energy individual to overall values} and \Cref{t.comparison}, we have
\begin{multline*}
\P(\mrm{ValueCtrl})\le \sum_{x\in L^{-2}\Z, |x|\leq KL^{-1/4}}\P(\mrm{ValueCtrl}_x)\\ \le \sum_{x\in L^{-2}\Z, |x|\leq KL^{-1/4}}\P( (\cL^\beta_1+\cL^\beta_2)(x)>L- (\log L)^2) < C\exp( C(\log L)^2L^{1/2} )\P(\cL^\beta>L).
\end{multline*}
Combining the last three displays leads to the conclusion.
\end{proof}

\subsection{Random location and concentration}
We next use \Cref{l.global conditioning to two} to control the location of $\pi(s)$, as well as prove \Cref{p.closeness of maximizer and polymer}.

The next result asserts that $\pi(s)$ is of order $L^{-1/4}$.
\begin{proposition}\label{p.maximizer tf}
For any $L$ large enough, $s\in [L^{-5/8}, 1-L^{-5/8}]$, and $0<K\le \log(L)$,
\begin{align*}
\P\left(|\pi(s)| > KL^{-1/4} \mid \cL^\beta > L\right) < C\exp(-cK^2).
\end{align*}
\end{proposition}
The general idea to prove \Cref{p.maximizer tf} is to (1) upper bound $\max_{(-KL^{-1/4}, KL^{-1/4})^c}\cL^\beta_1+\cL^\beta_2$ conditional on $\cL^\beta_1(x)$ and $\cL^\beta_2(x)$ for some $x=O(KL^{-1/4})$, using \Cref{cor.tent-pr}; and (2) connect the conditioning $\cL^\beta>L$ and the conditioning on the values of $\cL^\beta_1(x)$ and $\cL^\beta_2(x)$, using \Cref{l.global conditioning to two}.

\begin{proof}[Proof of \Cref{p.maximizer tf}]
We assume that $K>1$ since otherwise the conclusion follows trivially by setting $C$ large.

As in the proof of \Cref{l.global conditioning to two}, we have
\begin{align}\label{e.local max bound}
\P\left(\max\cL^\beta_1+ \cL^\beta_2 \geq L - \log L \midd \cL^\beta > L\right) >1-C\exp(-c(\log L)^2).
\end{align}
We also have
\begin{multline}  \label{e.out bound}
\P\left( \max_{(-(s_1\wedge s_2)L^{1/2}/2, (s_1\wedge s_2)L^{1/2}/2)^c} \cL^\beta_1 + \cL^\beta_2 \ge (1-\tfrac{1}{10}(s_1\wedge s_2))L\right) \\ < C\exp(-c(s_1\wedge s_2)L^{3/2})\P(\cL^\beta > L ),
\end{multline}
by upper bounding the LHS using \Cref{cor.p.tail bound for sum} plus shear invariance and a union bound, and lower bounding $\P(\cL^\beta > L)$ using \Cref{lem:fh-ut}.
Combining \eqref{e.local max bound} and \eqref{e.out bound} implies that
\begin{multline}  \label{eq:pi-trans}
\P\left(|\pi(s)|> KL^{-1/4} \mid \cL^\beta > L\right)<C\exp(-c(\log L)^2)\\
+ \P\left(\max_{[-\frac{1}{2}(s_1\wedge s_2)L^{1/2}, - KL^{-1/4}]\cup[KL^{-1/4}, \frac{1}{2}(s_1\wedge s_2)L^{1/2}]} \cL^\beta_1+ \cL^\beta_2 \ge L - \log L \midd \cL^\beta > L\right).    
\end{multline}
By \Cref{cor.tent-pr} with $a=KL^{-1/4}$, for any $|x| \le KL^{-1/4}/2$, and $(h_1, h_2) \in \mrm{Val}$, we have
\begin{multline*}
\P\Bigg(\max_{[-\frac{1}{2}(s_1\wedge s_2)L^{1/2}, - KL^{-1/4}]\cup[KL^{-1/4},\frac{1}{2}(s_1\wedge s_2)L^{1/2}]} \cL^\beta_1+ \cL^\beta_2 \ge L - \log L \\ \midd \cL^\beta_i(x)\in (h_i, h_i+\diff h_i), i=1,2\Bigg)
<C\exp(-cKL^{3/4})
\end{multline*}
Thus by \Cref{l.global conditioning to two}, the second term in the RHS of \eqref{eq:pi-trans} can be bounded by $C\exp(-cK^2)$, and the conclusion follows.
\end{proof}

We next finish proving the polymer concentration result.

\begin{proof}[Proof of Proposition~\ref{p.closeness of maximizer and polymer}]
Consider the event $\mrm{Tent}$ defined by
\[
\mrm{Tent} := \left\{\max_{|z-\pi(s)| \ge ML^{-1/2}\log L} \cL^\beta_1(z) + \cL^\beta_2(z) - (L - 2L^{1/2}|z-\pi(s)|) < 0\right\}
\]
To understand this definition, recall that we expect each of $\cL^\beta_1$ and $\cL^\beta_2$ to essentially adopt tent shapes under the conditioning $\cL^\beta > L$, where the tents each have slope approximately $\pm 2L^{1/2}$. Thus the sum $\cL^\beta_1 + \cL^\beta_2$ can be expected to be a line of slope $\pm 4L^{1/2}$ up to random fluctuations; in the definition of the event, we have reduced the slope magnitude by $1/2$ for the benefit of ignoring the random fluctuation.

Now, on $\mrm{Tent}$, 
\begin{align*}
\MoveEqLeft[10]
\int_{|x-\pi(s)| \ge ML^{-1/2}\log L}\exp\left((\cL^\beta_1+ \cL^{\beta}_2)(x)\right)\,\diff x\leq \int_{|x-\pi(s)| \ge ML^{-1/2}\log L}\exp\left(L - 2L^{1/2}|x-\pi(s)|\right)\,\diff x\\
&= 2 \int_{ML^{-1/2}\log L}^\infty \exp\left(L - 2L^{1/2}x\right)\,\diff x
= L^{-1/2}\exp\left(L-2M\log L\right),
\end{align*}
which implies that, on $\mrm{Tent}$ and $\cL^\beta > L$,
\begin{align*}
\polyP\left(|\Gamma_0(s)-\pi(s)|>ML^{-1/2}\log L \right) < L^{-2M}.
\end{align*}
Thus our task is now to upper bound $\P(\mrm{Tent}^c \mid \cL^\beta > L)$.
We note that by \Cref{p.maximizer tf}, 
\begin{equation}  \label{eq:max tf bd}
\P\left(|\pi(s)| > \log(L)L^{-1/4} \mid \cL^\beta > L\right) < C\exp(-c(\log L)^2).
\end{equation}
Then it remains to upper bound $\P(\mrm{Tent}^c \cap \{|\pi(s)| \le \log(L)L^{-1/4}\} \mid \cL^\beta > L)$.

We define the event $\mrm{Tent}_+$ by
\[
\mrm{Tent}_+ := \left\{\max_{|z| \ge (s_1\wedge s_2)L^{1/2}/3} \cL^\beta_1(z) + \cL^\beta_2(z) - (L - 3L^{1/2}|z|) < 0\right\}.
\]
Similar to \eqref{e.out bound} in the proof of \Cref{p.maximizer tf}, we have 
\[
\P(\mrm{Tent}_+^c) < C\exp(-c(s_1\wedge s_2)L^{3/2})\P(\cL^\beta > L ),
\]
by upper bounding the LHS using \Cref{cor.p.tail bound for sum} plus shear invariance and a union bound, and lower bounding $\P(\cL^\beta > L)$ using \Cref{lem:fh-ut}.
Then we have $\P(\mrm{Tent}_+^c\mid \cL^\beta>L)<C\exp(-c(s_1\wedge s_2)L^{3/2})$.
It now suffices to upper bound $\P(\mrm{Tent}^c \cap \mrm{Tent}_+\cap \{|\pi(s)| \le \log(L)L^{-1/4}\} \mid \cL^\beta > L)$, and we apply \Cref{l.global conditioning to two} for this.

Take $x\in L^{-2}\Z$, $|x|\le \log(L)L^{-1/4}$ and $(h_1, h_2)\in \mrm{Val}$, and consider
\begin{equation}  \label{eq:condptent}
\P\left(\mrm{Tent}^c\cap \mrm{Tent}_+\cap \{|\pi(s)| \le \log(L)L^{-1/4}\} \midd \cL^\beta_i(x)\in (h_i, h_i+\diff h_i), i=1,2\right).    
\end{equation}
Assuming that for each $i=1, 2$, $\cL^\beta_i(x)\in (h_i, h_i+\diff h_i)$, and
\[
\max_{ML^{-1/2}\log(L)/5\le |y|\le (s_1\wedge s_2)L^{1/2}/2} \cL^\beta_i(x+y) +\tfrac{3}{2}|y|(h_i/s_i)^{1/2}\le h_i ,
\]
and $\mrm{Tent}_+\cap\{|\pi(s)|\le \log(L)L^{-1/4}\}$, 
we must have $|\pi(s)-x|<ML^{-1/2}\log (L)/5$, and $\mrm{Tent}$ holds.
Therefore by \Cref{cor.tent-pr} (with $a=ML^{-1/2}\log L$), we have that \eqref{eq:condptent} is bounded by $C\exp(-cML^{1/2}\log L)+C\exp(-cL^{3/2})$.
Then by \Cref{l.global conditioning to two}, and taking $M_0$ large enough, we have 
$\P(\mrm{Tent}^c \mid \cL^\beta > L)<C\exp(-cL^{1/2}\log L)+C\exp(-c(\log L)^2)$. Thus the conclusion follows.
\end{proof}

\section{Tightness for polymers}  \label{sec:pos_tight}

We prove the $\beta=1$ case of \Cref{p.tightness} in this section.
We also denote $\cL^\beta=\cL^\beta(0,0;0,1)$ for simplicity of notations.

The main task is to prove the following two points estimate, which refines \Cref{l.positive temp two point crude bound}.
\begin{proposition}\label{p.two-point tf positive temp full}
For all $L\ge 2$, $0< s<t <1$ and $K>0$,
\begin{align*}
\P\left(\polyP\left(|\Gamma_0(s) - \Gamma_0(t)| > K(t-s)^{1/11}L^{-1/4}\right) >L^{-K} \midd \cL^\beta>L\right) < C\exp(-c(K\wedge \log L)^2).
\end{align*}
\end{proposition}
Compared to the zero temperature setting (i.e., \Cref{p.zero temperature two-point}), we weaken our demand to a H\"older $\frac{1}{11}-$ bound instead of $\frac{1}{2}-$, due to technical reasons which will be clear from its proof.

\Cref{p.two-point tf positive temp full} immediately implies the following.
\begin{corollary}\label{c.full interval two-point for positive temperature}
For all $L\ge 2$, $K>0$ and $0<s<t<1$,
\begin{align*}
\E\left[\polyP\left(|\Gamma_0(s)-\Gamma_0(t)| > K(t-s)^{1/11}L^{-1/4}\right)\midd \cL^\beta>L\right] < C\exp(-c(K\wedge \log L)^2).
\end{align*}
\end{corollary}
\begin{rem}  \label{rem:backbone}
As indicated in the introduction, from these tightness results and the polymer concentration of \Cref{p.closeness of maximizer and polymer}, one can define a random backbone $\tilde\pi$, by e.g., taking $\tilde\pi(s)=\pi(s)$ for each $s\in [0, 1]\cap L^{-10}\Z$, and linearly interpolate between them.
Then with probability $>1-C\exp(-c(\log L)^2)$, $\Gamma_0$ is within distance $L^{-1/2}(\log L)^2$ from $\tilde\pi$ at each $s\in [0, 1]\cap L^{-10}\Z$, by \Cref{p.closeness of maximizer and polymer}; and between any two points, $\Gamma_0$ can deviate at most $L^{-1}$, by \Cref{c.full interval two-point for positive temperature}.
Therefore (with the same probability) $\Gamma_0$ is within distance $L^{-1/2}(\log L)^2$ from the backbone $\tilde\pi$ throughout $[0,1]$.
\end{rem}

\begin{proof}[Proof of the $\beta=1$ case of \Cref{p.tightness}]
We recall that tightness on $\mc C([0,1], \R)$ follows by establishing a uniform modulus of continuity as well as one-point tightness. Since our processes are fixed at 0 at heights 0 and 1, the former also implies the latter, so we only need to establish a uniform modulus of continuity bound.

With \Cref{c.full interval two-point for positive temperature}, via a union bound over all $s=2^{-i}j$, $t=2^{-i}(j+1)$, $0< s<t<1$, with $i, j \in \Z$, we have the following. For any $K>0$, with probability $>1-C\exp(-c(K\wedge \log L)^2)$ conditional on $\cL^\beta>L$, 
\[
\sup_{t_0\le s<t\le 1-t_0}|\Gamma_0(s)-\Gamma_0(t)|(t-s)^{-1/12} \le KL^{-1/4}.
\]
Note that the exponent is $\frac{1}{12}$ in contrast to the $\frac{1}{11}$ present in Corollary~\ref{c.full interval two-point for positive temperature}; this is simply so that the union bound over all the scales can be performed. This completes the proof.
\end{proof}

To prove \Cref{p.two-point tf positive temp full}, the key idea is to upgrade \Cref{l.positive temp two point crude bound} using the concentration result \Cref{p.closeness of maximizer and polymer}.
It allows us to essentially say that (with high probability under $\P$) if the law of $\Gamma_0(s)$ under $\polyP$ assigns a probability greater than $\varepsilon$ (for a carefully chosen small $\varepsilon$) to the event of large transversal fluctuation, it must assign a close to $1$ probability to the event of having at least say half of the same transversal fluctuation. The latter event's probability is bounded by \Cref{l.positive temp two point crude bound}.
Note that in this argument it is actually not important that the interval around which most of the mass is spread is centered around $\pi(s)$, only that the interval is small.

We start by first upgrading the one-point estimate, at least $L^{-5/8}$ away from the ends.
\begin{proposition}\label{p.positive temp one-point tf}
For all $K>0$, $L\ge 2$, and $s\in[L^{-5/8}, 1-L^{-5/8}]$, 
\begin{align*}
\P\left(\polyP\left(|\Gamma_0(s)| \geq KL^{-1/4}\right) > L^{-K} \midd \cL^{\beta} > L\right) < C\exp(-c(K\wedge \log L)^2).
\end{align*}
\end{proposition}

\begin{proof}
We assume that $K$ is large enough since otherwise we just choose $C$ large to make the estimate hold.
We also assume that $L$ is large enough since otherwise the conclusion follows from \Cref{l.positive temp crude tf}.

Let $X(s,K) = \polyP(|\Gamma_0(s)| \geq KL^{-1/4})$. By Proposition~\ref{p.closeness of maximizer and polymer}, with probability $\ge 1-C\exp(-c(\log L)^2)$ conditional on $\cL^{\beta} > L$ we have
\begin{align*}
X(s,K)
&\leq \polyP\left(|\Gamma_0(s) - \pi(s)| \geq KL^{-1/2}\log L\right) + \one_{|\pi(s)| \geq \frac{1}{2}KL^{-1/4}}\\
&\leq L^{-2K} + \one_{|\pi(s)| \geq \frac{1}{2}KL^{-1/4}},
\end{align*}
and
\begin{align*}
X(s,\tfrac{1}{4}K)
&\geq \polyP\left(|\Gamma_0(s) - \pi(s)| \leq KL^{-1/2}\log L\right)\one_{|\pi(s)|\geq \frac{1}{2}KL^{-1/4}}\\
&\geq (1-L^{-2K})\one_{|\pi(s)|\geq \frac{1}{2}KL^{-1/4}}.
\end{align*}
Here we used that $KL^{-1/2}\log L<\frac{1}{4}KL^{-1/4}$.
These two bounds show that
\begin{align*}
X(s,K) > L^{-2K} \implies |\pi(s)| \geq \tfrac{1}{2}KL^{-1/4} \implies X(s,\tfrac{1}{4}K) \geq 1-L^{-2K},
\end{align*}
and thus
\[
\P\left(X(s,K) > L^{-2K}\midd \cL^{\beta}>L \right)< \P\left(X(s,\tfrac{1}{4}K) \geq 1-L^{-2K}\midd \cL^{\beta}>L\right)+C\exp(-c(\log L)^2).
\]
By Lemma~\ref{l.positive temp crude tf}, $\P\left(X(s,\tfrac{1}{4}K) > \exp(-K^2L^{-1/2}/32) \midd \cL^{\beta}>L\right) < C\exp(-cK^2)$. This completes the proof.
\end{proof}

Next, we turn to the two-point estimates. The proof strategies are analogous to what we just saw for the one-point: we combine a cruder estimate coming from shear invariance with the information that the polymer measure is localized on a smaller scale. We initially get the following two-point estimate, which is under the additional constraints that the two points are at least $L^{-5/8}$ away from each other, and $L^{-5/8}$ away from the ends. 
To get \Cref{p.two-point tf positive temp full} from it, it turns out that the cruder estimate \Cref{l.positive temp two point crude bound} is sufficient, since we just prove an H\"older $\frac{1}{10}-$ bound instead of $\frac{1}{2}-$.
\begin{proposition}\label{p.two-point tf positive temp}
For all $K>0$, $L\ge 2$, and $L^{-5/8}\le s<t\le 1-L^{-5/8}$, $t-s\ge L^{-5/8}$,
\begin{align*}
\P\left(\polyP\left(|\Gamma_0(s) - \Gamma_0(t)| > K(t-s)^{1/3}L^{-1/4}\right) > L^{-K} \midd \cL^\beta>L\right) < C\exp(-c(K\wedge \log L)^2).
\end{align*}
\end{proposition} 
We note that the exponent $1/3$ can be replaced by any number $<2/5$, due to the exponent of $5/8$ (we only need that their product is $<1/4$).

\begin{proof}[Proof of \Cref{p.two-point tf positive temp}]
This proof is very similar to that of \Cref{p.positive temp one-point tf}.
Again, we can assume that $L, K$ are large enough.
Let $X(s,t, K) = \polyP\left(|\Gamma_0(s)-\Gamma_0(t)| > K(t-s)^{1/3}L^{-1/4}\right)$. By Proposition~\ref{p.closeness of maximizer and polymer}, with probability $\ge 1-C\exp(-c(\log L)^2)$ conditional on $\cL^{\beta} > L$,
\begin{align*}
X(s,t, K)
&\leq \polyP\left(\max_{r\in\{s,t\}} |\Gamma_0(r) - \pi(r)| \geq KL^{-1/2}\log L\right) + \one_{|\pi(s)-\pi(t)| \geq \frac{1}{2}K(t-s)^{1/3}L^{-1/4}}\\
&\leq 2L^{-2K} + \one_{|\pi(s)-\pi(t)| \geq \frac{1}{2}K(t-s)^{1/3}L^{-1/4}}
\end{align*}
and
\begin{align*}
X(s,t, \tfrac{1}{4}K)
&\geq \polyP\left(\max_{r\in\{s,t\}} |\Gamma_0(r) - \pi(r)| \leq KL^{-1/2}\log L\right)\one_{|\pi(s)-\pi(t)| \geq \frac{1}{2}(t-s)^{1/3}KL^{-1/4}}\\
&\geq (1-2L^{-2K})\one_{|\pi(s)-\pi(t)| \geq \frac{1}{2}(t-s)^{1/3}KL^{-1/4}}.
\end{align*}
Here we used that $KL^{-1/2}\log L < \frac{1}{4}K(t-s)^{1/3}L^{-1/4}$ since $t-s>L^{-5/8}$.
These imply that
\begin{align*}
X(s,t,K) > 2L^{-2K} \implies |\pi(s)-\pi(t)| \geq \tfrac{1}{2}K(t-s)^{1/3}L^{-1/4} \implies X(s,t,\tfrac{1}{4}K) \geq 1-2L^{-2K}.
\end{align*}
Besides, by \Cref{l.positive temp two point crude bound} we have
$\P\left(X(s,t,\frac{1}{4}K) > \exp(-K^2L^{-1/2}/32) \midd \cL^{\beta}>L\right) < C\exp(-cK^2)$. These complete the proof.
\end{proof}

\begin{proof}[Proof of \Cref{p.two-point tf positive temp full}]
We assume that $K$ is large enough by taking $C$ large (if necessary), and assume that $L$ is large enough by applying \Cref{l.positive temp two point crude bound} otherwise.

For the case where $t-s< 2L^{-5/8}$,
by \Cref{l.positive temp two point crude bound}, we can bound the conditional probability in the statement of \Cref{p.two-point tf positive temp full} by
\[
C\exp(-cK^2(t-s)^{-9/11}) < C\exp(-cK^2),
\]
since $\exp(-\frac{1}{2} K^2(t-s)^{-9/11}L^{-1/2}) \leq \exp(-cK^2L^{1/88}) < L^{-K}$.

For the case where $t-s\ge 2L^{-5/8}$, if $s<L^{-5/8}$, for $s_*=s$ or $L^{-5/8}$, we apply \Cref{l.positive temp crude tf} to get that
\begin{align*}
\P\left(\polyP\left[|\Gamma_0(s_*)| \geq K (s_*(1-s_*))^{1/11}L^{-1/4}\right] > L^{-3K} \midd \cL^{\beta} > L\right) < C\exp(-cK^2s_*^{-9/11}),
\end{align*}
since $\exp(-\frac{1}{2} K^2(s_*(1-s_*))^{-9/11}L^{-1/2}) < L^{-3K}$.
Therefore we have
\begin{align*}
\P\left(\polyP\left(|\Gamma_0(s) - \Gamma_0(L^{-5/8})| > K(L^{-5/8}-s)^{1/11}L^{-1/4}\right) >L^{-2K} \midd \cL^\beta>L\right) < C\exp(-cK^2).
\end{align*}
Similarly, if $t>1-L^{-5/8}$, we have
\begin{align*}
\P\left(\polyP\left(|\Gamma_0(t) - \Gamma_0(1-L^{-5/8})| > K(t-1+L^{-5/8})^{1/11}L^{-1/4}\right) >L^{-2K} \midd \cL^\beta>L\right) < C\exp(-cK^2).
\end{align*}
Thus we have reduced the problem of $s, t$ to the same problem of $s\vee L^{-5/8}, t\wedge (1-L^{-5/8})$, and that follows from \Cref{p.two-point tf positive temp}.
\end{proof}

The remaining sections are devoted to proving finite dimensional convergence.

\section{Estimates on free energies under conditionings}\label{s.free energy under conditioning}

As indicated in \Cref{iop}, our strategy of deducing finite dimensional limit heavily relies on realizing the conditioning on $\cL^\beta(0,0;0,1)$ as conditioning on the existence of peaks at certain heights at certain locations. As such, we will often need estimates on the probability of the existence of such peaks given the global conditioning $\cL^\beta(0,0;0,1)$, or on the probability of the latter conditioned on the former.
One such estimate (\Cref{l.inf control}) has appeared before, and in this section we provide some more refined ones.

Let us introduce the setup, and some notations needed to state these estimates. 
We will work under a setting similar to that in \Cref{sec:propsum}.
Namely, we consider $(s_1, \cdots, s_{k-1})\in \rDe_{k-1}([0,1])$ and $\vec x\in\R^{k-1}$ for $k\in \N$, and denote $s_0=x_0=x_{k}=0$ and $s_k=1$.
All the constants within this section can depend on $k$ and $(s_1, \cdots, s_{k-1})$.
For the convenience of notations we adopt the shorthand $\cL^\beta = \cL^\beta(0,0;0,1)$ and $\cL^\beta_i = \cL^\beta(s_{i-1}, x_{i-1}; s_i, x_i)$ for each $1\le i \le k$.
We also write $\vec\cL^\beta$ for the vector $\{\cL^\beta_i\}_{i=1}^k$.

Below we use $x\approx y$ to denote that $x\in y+[0, e^{-L}]$; and for any vector $\vec h\in \R^k$,  $\smash{\vec\cL^\beta\approx \vec h}$ is the event where $\cL^\beta_i\approx h_i$ for each $i$.
A main reason for introducing this notation is that we will need to invoke coalescence or Brownian comparison statements (\Cref{prop:dp-tent-coal} and \Cref{prop:dp-tent-bcomp}) which do not allow conditioning on exact values.

Take $\vec a=(a_0,\ldots, a_k)\in [-L^{5/16}\log L, L^{5/16}\log L]^{k-1}$.
For each $i=1,\ldots, k-1$, we write  
\[\pi^*(s_i) = \argmax_{x} \cL^\beta(a_{i-1}, s_{i-1}; x, s_i) + \cL^\beta(x, s_i; a_{i+1}, s_{i+1}).\]
This differs from $\pi$ because $\pi(s_i)$ is defined as the maximizer of $\cL^\beta(0,0; x, s_i) + \cL^\beta(x, s_i; 0, 1)$, while $\pi^*(s_i)$ is the maximizer of profiles within certain time strips which are disjoint for different $i$; this disjointness gives some independence which will be useful for the arguments.

We next introduce useful notations generalizing the maximum to positive temperature: for $f:\R^{k-1}\to \R$ and a set $I\subseteq \R^{k-1}$,
\begin{align}\label{e.maxbeta definition}
\maxbeta_{\vec x\in I} f = \begin{cases}
\displaystyle\log \int_{I} \exp\bigl(f(x_1, \ldots, x_{k-1})\bigr)\, \diff x_1 \cdots \diff x_{k-1} & \beta = 1\\
\displaystyle\max_{\vec x\in I} f(x_1 \cdots \diff x_{k-1}) & \beta = \infty.
\end{cases}
\end{align}
We further define the restricted free energy $\cL^\beta[\vec y,R]$ for $R>0$ and $\vec y\in \R^{k-1}$ by
\[
\cL^\beta[\vec y,R] = \maxbeta_{\|\vec x- \vec y\|_\infty\le R} \sum_{i=1}^{k} \cL^\beta_i.
\]

Let $r_{\beta=1} = 1$ and $r_{\beta=\infty}=L^{-1/2}$. These are the fluctuation scales of the total free energy conditional on the peak heights and locations. 
Let $w_{\beta=1} = L^{-1/2}$ and $w_{\beta=\infty}=L^{-1}$. These are the window sizes around each $\pi^*(s_i)$ that would affect $\cL^\beta$, under the upper tail.

The next two statements record the above-mentioned complementary estimates and will be the goal of this section.
\begin{proposition}\label{p.k-point part to whole free energy}
Take any large enough $L, M$, and $\vec h\in \R^k$, $\vec a\in[-L^{5/16}\log L, L^{5/16}\log L]^{k+1}$, $\vec x\in[-L^{5/16}\log L, L^{5/16}\log L]^{k-1}$.
Denote $H=\sum_{i=1}^kh_i$ and assume that 
$H>L/2$ and each $\left|h_i-(s_i-s_{i-1})H\right|<L^{8/9}$. Then we have
\begin{multline}  \label{eq:kppwf1}
\P\bigg(\cL^\beta > H - (k-1)\beta^{-1}\log(2H^{1/2}) + Mr_\beta,\\ \max_{i=1, \ldots, k-1}|\pi^*(s_i) -x_i| \leq w_\beta \midd \vec\cL^\beta\approx \vec h\bigg)<C\exp\left(-cM^2L^{1/2}r_\beta\right) + C\exp(-cL^{3/2}),
\end{multline}
\begin{multline}  \label{eq:kppwf2}
\P\bigg(\cL^\beta[\vec x, L^{-1/2}(\log L)^2] > H - (k-1)\beta^{-1}\log(2H^{1/2}) + Mr_\beta,\\ \max_{i=1, \ldots, k-1}|\pi^*(s_i) -x_i| \leq w_\beta \midd \vec\cL^\beta\approx \vec h\bigg)
> c\exp\left(-CM^2L^{1/2}r_\beta\right)-C\exp(-cL^{3/2}).
\end{multline}
Moreover, we also have
\begin{equation}   \label{eq:kppwf3}
\P\left(\cL^\beta > H - (k-1)\beta^{-1}\log(2H^{1/2}) + M \midd \vec\cL^\beta\approx \vec h\right)<C\exp\left(-cML^{1/2}\right).
\end{equation}
\end{proposition}

The term $(k-1)\log(2H^{1/2})$ is meant to be present in the case $\beta=1$ and absent in the case $\beta=\infty$, and multiplying the term by $\beta^{-1}$ is a convenient notational tool to this effect (though in fact if one were to work out the arguments in the case of general  $\beta$, the term would be $(k-1)\beta^{-1}\log(2\beta H^{1/2})$).
The source of the log term for $\beta = 1$ comes from the fact that $\int_{-\infty}^\infty \exp(-4H^{1/2} |x|)\diff x = (2H^{1/2})^{-1}$, which itself is a result that the dominant contribution to the integral is from an interval of scale $L^{-1/2}$ around zero. Since on $\vec \cL^\beta\approx \vec h$ the terms in the exponential in the convolution formula for $\exp(\cL^{\beta})$ are essentially sums of two Brownian bridges with slope $-2H^{1/2}$ each, heuristically taking logarithms will yield that $\cL^\beta$ loses $(k-1)\log(2H^{1/2})$ compared to the peak height of $H$.

We next give an estimate on the probability of $\sum_{i=1}^k \cL^\beta_i$ being much smaller than $L$, conditional on $\cL^\beta>L$ and $\pi^*$. It can be viewed as a refinement of \Cref{l.inf control}.

\begin{proposition}\label{p.sum close under conditioning k-point}
For all large enough $L, M$ with $M<L^{0.01}$, and $\vec x\in[-L^{5/16}\log L, L^{5/16}\log L]^{k-1}$, $\vec x\in[-L^{5/16}\log L, L^{5/16}\log L]^{k-1}$, we have
\begin{align*}
\P\left(E_{k,M, w, L} \midd \cL^\beta > L, \max_{i=1, \ldots, k}|\pi^*(s_i) -x_i| \leq w_{\beta}\right) \leq 
\exp\left(-cM^2L^{1/2}r_\beta\right),
\end{align*}
where
\begin{multline}\label{e.E(k,m,W,L) definition}
E_{k,M,w, L} = \left\{\sum_{i=1}^k\cL^\beta_i\in L + (k-1)\beta^{-1}\log(2L^{1/2}) + [-L^{8/9}, - Mr_\beta]\right\}\\
\cap \bigcap_{i=1}^k\left\{\cL^\beta_i > (s_i-s_{i-1})L-k^{-1}L^{8/9}\right\}.
\end{multline}
\end{proposition}

Proposition~\ref{p.sum close under conditioning k-point} is proved by invoking Bayes' theorem and Proposition~\ref{p.k-point part to whole free energy}, and we give its proof now; we will return to proving Proposition~\ref{p.k-point part to whole free energy} later.
For any vector $\vec h\in\R^k$, denote
\begin{align*}
A_{\vec h} = \bigcap_{i=1}^k\left\{\cL^\beta_i \in h_i + [0,r_\beta]\right\}.
\end{align*}
\begin{proof}[Proof of Proposition~\ref{p.sum close under conditioning k-point}]
We assume without loss of generality that $M\in\N$. Define $\mrm{Val}_M$ by
\[
\mrm{Val}_M = \left\{\vec h\in (r_\beta\Z)^k : 
\parbox[c]{4in}{\centering 
$\sum_{i=1}^k h_i  \in L + (k-1)\beta^{-1}\log(2L^{1/2})  + [-L^{8/9}-r_\beta, - Mr_\beta]$ \\[6pt]
$h_i\in (s_i-s_{i-1})L + [-L^{8/9}, L^{8/9}],\quad i=1, \ldots, k$
}\right\}.
\]
By doing a disjoint decomposition and applying Bayes' theorem,
\begin{align*}
\MoveEqLeft[0]
\P\left(E_{k, M,w,L} \midd \cL^\beta > L, \max_{i=1, \ldots, k}|\pi^*(s_i) -x_i| \leq w_{\beta}\right) \nonumber\\
&\leq \frac{\sum_{\vec h\in \mrm{Val}_M}\P\left(\cL^\beta > L, \max_{i=1, \ldots, k}|\pi^*(s_i) -x_i| \leq w_{\beta} \midd A_{\vec h}\right)\cdot \P\left(A_{\vec h} \right)}{\P\left(\cL^\beta > L, \max_{i=1, \ldots, k}|\pi^*(s_i) -x_i| \leq w_{\beta}\right)}.
\end{align*}
We can also decompose the denominator in a similar fashion, but with the sum over all $\vec h$. Now for each $\vec h \in \mrm{Val}_M$, we wish to bound the ratio
\[
\frac{\P\left(\cL^\beta > L, \max_{i=1, \ldots, k}|\pi^*(s_i) -x_i| \leq w_{\beta} \midd A_{\vec h}\right)\cdot \P(A_{\vec h})}{\P\left(\cL^\beta > L, \max_{i=1, \ldots, k}|\pi^*(s_i) -x_i| \leq w_{\beta} \midd A_{\vec h'}\right)\cdot \P(A_{\vec h'})},
\]
where $\vec h'\in (r_\beta\Z)^k$ is defined as follows: for each $i=2, \ldots, k$, $h_i' = h_i$, while $\sum_{i=1}^k h_i'\in L +(k-1)\beta^{-1}\log(2L^{1/2}) +[0, 1]$,
$h_1'-h_1 \in \lfloor h_1'-h_1\rfloor +[0, r_\beta)$.

By \Cref{t.comparison} we have
\[
\frac{\P(A_{\vec h})}{\P(A_{\vec h'})} <C\exp(C(h_1'-h_1)L^{1/2}),
\]
and by \Cref{p.k-point part to whole free energy} we have
\[
\frac{\P\left(\cL^\beta > L, \max_{i=1, \ldots, k}|\pi^*(s_i) -x_i| \leq w_{\beta} \midd A_{\vec h}\right)}{\P\left(\cL^\beta > L, \max_{i=1, \ldots, k}|\pi^*(s_i) -x_i| \leq w_{\beta} \midd A_{\vec h'}\right)}
< C\exp(-c(h_1'-h_1)^2L^{1/2}r_\beta) + C\exp(-cL^{3/2}).
\]
By combining these two estimates, we can now bound the ratio as desired.
Then by summing over all $\vec h\in \mrm{Val}_M$ the conclusion follows.
\end{proof}

We next give the proof of \Cref{p.k-point part to whole free energy}. 

\subsection{$k=2$ setting}
The basic step is to prove the following $k=2$ version: from this, we can obtain the general $k$-version by invoking coalescence (Proposition~\ref{prop:dp-tent-coal}) to break down the case of general $k$ to a collection of $k=2$ cases.

The notations within this subsection are slightly different, and we setup now. We take large enough $L>0$, $h_1$ and $h_2=\Theta(L)$, $0<s_1, s_2<1$, $|x_*|\le 2L^{5/16}\log L$.
All the constants in this section can depend on $s_1\wedge s_2$.
We write $\vec\cL^\beta=(\cL^\beta(0,0;x_*,s_1), \cL^\beta(x_*,s_1;0,s_1+s_2))$.
For any $R>0$ we denote
\[\cL^\beta[x_*,R]=\maxbeta_{|x-x_*|\le R} \cL^\beta(0,0;x,s_1)+ \cL^\beta(x,s_1;0,s_1+s_2)\]
and we write $\cL^\beta[R]=\cL^\beta[0, R]$.

We also denote $\lambda = (s_1^{-1/2}h_1^{1/2} + s_2^{-1/2}h_2^{1/2})$ to be the first order of the slope of the sum $\cL^\beta(0,0;x,s_1)+ \cL^\beta(x,s_1;0,s_1+s_2)$ under the conditioning that $\vec \cL^\beta \approx \vec h$. We take some $x_-$, $x_+$ with $|x_-|, |x_+|\le 2L^{5/16}\log L$ and denote
\[
\pi^*(s_1)=\argmax \cL^\beta(x_-, 0; \cdot, s_1) + \cL^\beta(\cdot, s_1; x_+, s_1+s_2).
\]
\begin{lemma}\label{l.part free energy to whole free energy}
Denote $W=10^{-6}(h_1s_1)^{1/2}\wedge (h_2s_2)^{1/2}$ and assume $h_1,h_2 = \Theta(L)$.
In the case of $\beta=1$, for any $M>L^{1/16}\log L$ we have
\begin{multline}  \label{eq:ptfree1}
\P\Big(\cL^\beta[W] > h_1+h_2 - \beta^{-1}\log(\lambda) + ML^{-1/4}, |\pi^*(s_1)-x_*|\leq w_\beta  \midd \vec\cL^\beta\approx \vec h \Big) \\ < C\exp(-cM^2)+C\exp(-cL^{3/2}),
\end{multline}
and for any $M>0$,
\begin{multline}  \label{eq:ptfree2}
\P\Big(\cL^\beta[x_*,L^{-1/2}(\log L)^2] > h_1+h_2 - \beta^{-1}\log( \lambda) + ML^{-1/4}, |\pi^*(s_1)-x_*|\leq w_\beta \midd \vec\cL^\beta\approx \vec h\Big)\\ 
> c\exp(-CM^2)-C\exp(-cL^{3/2}).
\end{multline}
In the case of $\beta = \infty$, the first two bounds hold for any $M>0$, after (1) replacing $ML^{-1/4}$ with $ML^{-1/2}$ in both; (2) in the lower bound, replacing $\cL^\beta[x_*, L^{-1/2}(\log L)^2]$ with $\cL^\beta[x_*, L^{-1}]$.

Moreover, for both $\beta=1$ and $\beta=\infty$, and $M$ large enough, we have
\begin{equation}  \label{eq:ptfree3}
\P\Big(\cL^\beta[W] > h_1+h_2 - \beta^{-1}\log(\lambda) + M \midd \vec\cL^\beta\approx \vec h \Big) < C\exp(-cML^{1/2}).
\end{equation}
\end{lemma}

\begin{proof}
The general idea of this proof is to invoke \Cref{l.dp-comp-s}, then do computations of Brownian motions. 

We can assume that $M$ is large enough since otherwise the estimates trivially hold.
For the convenience of notations, we denote $S(x)=\cL^\beta(0,0;x,s_1)+\cL^\beta(x,s_1;0,s_1+s_2)$, and write $B(x)=S(x_*+x)-S(x_*)+2\lambda|x|$.
In light of \Cref{l.dp-comp-s}, we shall think of $B(x)$ as a Brownian motion in the interval $[-CL^{-1/2}\log L, CL^{-1/2}\log L]$.

\noindent\textit{Positive temperature.} 
We first consider the case of $\beta=1$ (where $w_{\beta=1}=L^{-1/2}$), and turn to the zero temperature ($\beta = \infty$) case later.

\noindent\textbf{Upper bound.} 
By the convolution formula,
\begin{equation}  \label{eq:partfreeconv}
\cL^\beta[W] = \log \int_{-W}^{W} \exp(S(x))\diff x.
\end{equation}
Since $h_1, h_2=\Theta(L)$ and $|x_*|, |x_-|, |x_+|\le 2L^{5/16}\log L$, by \Cref{cor.tent-pr}, conditional on $\vec\cL^\beta\approx \vec h$, with probability at least $1-C\exp(-cM^2)$,
\begin{equation}  \label{eq:S11}
S(x) < h_1+h_2-2\lambda|x-x_*|+C|x-x_*|L^{5/16}\log L+ ML^{-1/4},
\end{equation}
for any $|x-x_*|\le L^{-1/2}$, and 
\begin{equation}  \label{eq:S12}
S(x) < h_1+h_2-2\lambda|x-x_*|+C|x-x_*|L^{5/16}\log L+M|x-x_*|^{1/2}(|\log(|x-x_*|LM^{-2})|+1),
\end{equation}
for any $x$ with $|x-x_*|\ge L^{-1/2}$, $|x|\le W$.

To get \eqref{eq:ptfree3}, we note that by plugging \eqref{eq:S11} and \eqref{eq:S12} into \eqref{eq:partfreeconv}, we have that $\cL^\beta[W]<h_1+h_2+\log(\lambda^{-1})+CM^2L^{-1/2}$ when $M>L^{1/4}$. 
(Note that the RHS of \eqref{eq:S12} is maximized when $|x-x_*|$ is of order $M^2/L$.)
Relabeling $M$ completes the proof.

To get \eqref{eq:ptfree1},
we note that with probability $>1-C\exp(-cL^{3/2})$ conditional on $\vec\cL^\beta\approx \vec h$, the event $|\pi^*(s_1)-x_*|\leq L^{-1/2}$ implies that
$S(x)<h_1+h_2+ML^{-1/4}+C\exp(-cL)$ for all $|x|\le W$.
Plugging the minimum of this bound and \eqref{eq:S12} into \eqref{eq:partfreeconv}, we get that $\cL^\beta[W]<h_1+h_2+\log(\lambda^{-1})+CML^{-1/4}+CL^{-3/16}\log L$. Using that $M>L^{1/16}\log L$ the estimate \eqref{eq:ptfree1} follows.

\smallskip

\noindent\textbf{Lower bound.} 
As the comparison in \Cref{l.dp-comp-s} is for some interval with length of order $L^{1/2}$, we need to do a truncation for $\pi^*(s_1)$.
Namely, we claim
\begin{equation}  \label{eq:pisgCexc}
\P(|\pi^*(s_1)| > W \mid \vec\cL^\beta\approx \vec h)<C\exp(-cL^{3/2}).    
\end{equation}
Indeed, from \Cref{lem:fh-ut} and \Cref{lem:fh-cont}, and \Cref{prop:dp-tent-coal}, we can deduce that
\[
\frac{\P\left(\sup_{|x|>W} \cL^{\beta}(x_-,0;x,s_1)\ge \cL^{\beta}(x_-,0;x_*,s_1), \cL^{\beta}_1\approx h_1\right)}{\P(\cL^{\beta}_1\approx h_1)} <C\exp(-cL^{3/2}),
\]
\[
\frac{\P\left(\sup_{|x|>W} \cL^{\beta}(x,s_1;x_+,s_1+s_2)\ge \cL^{\beta}(x_*,s_1;x_+,s_1+s_2),  \cL^{\beta}_2\approx h_2\right)}{\P(\cL^{\beta}_2\approx h_2)} <C\exp(-cL^{3/2}).
\]
These together imply \eqref{eq:pisgCexc}.

Now by definition
\begin{align}  \label{eq:LbxlogL}
\MoveEqLeft[2]
\cL^\beta[x_*, L^{-1/2}(\log L)^2]= \log \int_{-L^{-1/2}(\log L)^2}^{L^{-1/2}(\log L)^2} \exp(S(x_*+x))\diff x\nonumber\\
&= S(x_*)+ \log \int_{-L^{-1/2}(\log L)^2}^{L^{-1/2}(\log L)^2} \exp\left(-2\lambda|x| + B(x)\right)\diff x. 
\end{align}
Define events $A_M$ and $E$ as
\begin{equation*}
\begin{split}
A_M 
&= \left\{\min_{|x|\in [L^{-1/2}/4, L^{-1/2}/2]} B(x) \geq ML^{-1/4}\right\} \cap \left\{\min_{|x|\le L^{-1/2}/4} B(x) \geq -L^{-1/4}\right\}\\
&\qquad \cap \left\{ B(x) \ge -|x|^{1/2}\log(|x|L^{1/2}) \text{ for all }|x|\in [L^{-1/2}/2,L^{-1/2}\log L]\right\},\\
E &= 
\left\{ \max_{|x|\ge L^{-1/2}, |x+x_*| \le W} B(x)-2\lambda|x| <   \max_{|x| \le L^{-1/2}} B(x)-2\lambda|x| -C\exp(-cL) \right\}.
\end{split}
\end{equation*}
If $B$ were replaced by a two-sided Brownian motion with bounded drift, the probability of $A_M\cap E$ would be $>c\exp(-CM^2)$, by standard Brownian motion computations.
Then by \Cref{l.dp-comp-s},
\begin{multline*}
  \P(A_M\cap E\mid  \vec\cL^\beta\approx \vec h )>c\exp(-CM^2)(1-C\exp(-cL))-C\exp(-cL^{3/2})\\>c\exp(-CM^2)-C\exp(-cL^{3/2}).
\end{multline*}
Assuming $A_M\cap E$, we can lower bound \eqref{eq:LbxlogL} by $h_1+h_2+\log(\lambda^{-1})+cML^{-1/4}-CL^{-1/4}$;
and $|\pi^*(s_1)|\le L^{-1/2}$ outside an event with probability $<C\exp(-cL^{3/2})$, by \Cref{prop:dp-tent-coal}.
Noting that $M$ is taken large enough, together with \eqref{eq:pisgCexc}, the lower bound \eqref{eq:ptfree2} follows.

These complete the proof of the lemma in the $\beta=1$ case. We now turn to the $\beta=\infty$ case.

\smallskip

\noindent\emph{Zero temperature:} We now need to bound the maximum of $S$ instead of its integral. For the upper bound, since $h_1, h_2=\Theta(L)$ and $|x_*|, |x_-|, |x_+| \le 2L^{5/16}\log L$, by \Cref{cor.tent-pr}, conditional on $\vec\cL^\beta\approx \vec h$, with probability at least $1-C\exp(-cM^2)$,
\[
S(x) < h_1+h_2-2\lambda|x-x_*|+C|x-x_*|L^{5/16}\log L+ ML^{-1/2},
\]
for any $|x-x_*|\le L^{-1}$,
and
\[
S(x) < h_1+h_2-2\lambda|x-x_*|+C|x-x_*|L^{5/16}\log L+M|x-x_*|^{1/2}(|\log(|x-x_*|LM^{-2})|+1),
\]
for any $x$ with $|x-x_*|\ge L^{-1}$, $|x|\le W$.
Then $S(x)$ in $[-W, W]$ is at most $h_1 + h_2 + CM^2L^{-1/2}$ (note that the RHS of the previous display is maximized when $|x-x_*|$ is of order $M^2/L$), and we get \eqref{eq:ptfree3}, by relabeling $M$.

Under the additional assumption that $|\pi^*(s_1)-x_*|\leq L^{-1}$ (note that $w_{\beta=\infty}=L^{-1}$), by \Cref{prop:dp-tent-coal}, we have that $S(x)$ in $[-W, W]$ is at most $h_1 + h_2 + ML^{-1/2}$ outside another event of probability $<C\exp(-cL^{3/2})$. Thus we get \eqref{eq:ptfree1}.

For the lower bound, we note that \eqref{eq:pisgCexc} still holds with the same proof. Then by \Cref{prop:dp-tent-coal}, we just need to consider the probability of
\[
\max_{|x|\le L^{-1}} B(x) \ge ML^{-1/2}, \quad 
\left|\argmax_{|x+x_*|\le W} B(x)-2\lambda|x|\right| \le L^{-1}.
\]
Again by \Cref{l.dp-comp-s}, conditional on $\vec\cL^\beta\approx \vec h$, the above happens with probability $>c\exp(-CM^2)-C\exp(-cL^{3/2})$.
This completes the proof. 
\end{proof}

Below we return to using the notations defined before this subsection, i.e., at the beginning of Section~\ref{s.free energy under conditioning}.

\subsection{Proof of \Cref{p.k-point part to whole free energy}}
Note that (with $y_0 = y_{k}=0$)
\begin{align*}
\cL^\beta = \maxbeta_{y_1, \ldots, y_{k-1}} \sum_{i=1}^k \cL\left(y_{i-1}, s_{i-1}; y_i, s_i\right).
\end{align*}
We want to simplify the RHS using coalescence (Proposition~\ref{prop:dp-tent-coal}). Since we only know that coalescence holds up to a distance of order $L^{1/2}$, we have to first argue that we can restrict the $\max^{(\beta)}$ to be over an interval centered at 0 of size much smaller than $L^{1/2}$,  with high probability conditionally on $\vec\cL^\beta\approx \vec h$.

For notational convenience let, for $R>0$ and $i=1,\ldots, k-1$, and $x\in\R$,
\begin{align*}
\cL^\beta_i[x,R] = \maxbeta_{|y-x|\le R}\cL^\beta(x_{i-1}, s_{i-1}; y, s_i) + \cL^\beta(y, s_i; x_{i+1}, s_{i+1}).
\end{align*}

\noindent\textbf{Restricting the interval and coalescence:} Let $\delta = 10^{-6}\min_{i=1, \ldots, k}(s_i-s_{i-1})$.
Let $\cE_0$ denote the event that
\[
\cL^\beta \leq \cL^\beta[\vec 0, \delta L^{1/2}] + C\exp(-cL),
\]
where $\vec 0 \in \R^{k-1}$ is the vector with every entry being $0$.
Then we have
\[
\P(\cE_0^c\mid \vec\cL^\beta\approx \vec h)
<
\P(\cE_0^c, \cL>H-C\log L \mid \vec\cL^\beta\approx \vec h) + \P(\cL\le H-C\log L \mid \vec\cL^\beta\approx \vec h) .
\]
The second term in the RHS is bounded by $C\exp(-cL^2)$ by \Cref{l.free energy individual to overall values} and \Cref{lem:fh-ut}; and the first term is bounded by
\[
\P(\cE_0^c\mid \cL>H-C\log L) \frac{\P(\cL>H-C\log L)}{\P(\vec\cL^\beta\approx \vec h)}.
\]
By \Cref{l.positive temp crude tf}, we have $\P(\cE_0^c\mid \cL>H)<C\exp(-cL^{3/2})$; and the ratio is bounded by $C\exp(CL^{25/18})$ using \Cref{lem:fh-ut} (and the assumption that $\vec x\in[-L^{5/16}\log L, L^{5/16}\log L]^{k-1}$, and that each $\left|h_i-(s_i-s_{i-1})H\right|<L^{8/9}$).
Then we conclude that $\P(\cE_0^c\mid \vec\cL^\beta\approx \vec h)<C\exp(-cL^{3/2})$.

By the convolution formula, this with \Cref{prop:dp-tent-coal} implies that $\P(\mrm{Coal}\mid \vec\cL^\beta\approx \vec h)>1-C\exp(-cL^{3/2})$, where the coalescence event $\mrm{Coal}$ is defined by
\[
\mrm{Coal} = \biggl\{\cL^\beta \in \sum_{i=1}^{k-1} \cL^\beta_i[0, \delta L^{1/2}]
- \sum_{i=2}^{k-1} \cL^\beta_i + [-C\exp(-cL), C\exp(-cL)]\biggr\}.
\]

\smallskip

\noindent\textbf{Upper bound:} Under $\vec\cL^\beta\approx \vec h$, the sum $\sum_{i=2}^{k-1} \cL^\beta_i\in\sum_{i=2}^{k-1}h_i  + [0,(k-2)e^{-L}]$. Therefore, assuming $\vec\cL^\beta\approx \vec h$ and $\mrm{Coal}$,
\begin{align*}
\MoveEqLeft[2]
\Bigl\{\cL^\beta > H - (k-1)\beta^{-1}\log(2H^{1/2}) + Mr_\beta\Bigr\}\\
&\subseteq \Biggl\{\sum_{i=1}^{k-1} \cL^\beta_i[0,\delta L^{1/2}]\geq \sum_{i=1}^{k-1} (h_i+h_{i+1}) -(k-1)\beta^{-1}\log(2H^{1/2}) + Mr_\beta - C\exp(-cL) \Biggr\}.
\end{align*}
Let $P_i$ be the probability of
\[
 \cL^\beta_i[0,\delta L^{1/2}] \geq h_i + h_{i+1} - \beta^{-1}\log(2H^{1/2}) + k^{-1}M,   
\]
conditional on $\vec \cL^\beta\approx\vec h$.
Let $P_i'$ be the probability of
\[
 \cL^\beta_i[0,\delta L^{1/2}] \geq h_i + h_{i+1} - \beta^{-1}\log(2H^{1/2}) + k^{-1}Mr_\beta ,   
\]
and that $|\pi^*(s_i)-x_i|\le w_\beta$, conditional on $\vec \cL^\beta\approx\vec h$.
By a union bound, it suffices to upper bound $\sum_{i=1}^{k-1} P_i$ and $\sum_{i=1}^{k-1} P_i'$.

For \eqref{eq:kppwf3},
we can then invoke \eqref{eq:ptfree3} in \Cref{l.part free energy to whole free energy} (using shear and translation invariance properties) to obtain an upper bound on each $P_i$. 
Similarly, for \eqref{eq:kppwf1} we invoke \eqref{eq:ptfree1} in \Cref{l.part free energy to whole free energy} to obtain an upper bound on $P_i'$.
We note that, since $|h_i-(s_i-s_{i-1})H|<L^{8/9}$, it follows that $\log((s_i-s_{i-1})^{-1/2}h_i^{1/2}+(s_{i+1}-s_i)^{-1/2}h_{i+1}^{1/2}) = \log(2H^{1/2}) + O(L^{-1/9})$, while $\beta^{-1}L^{-1/9}<r_\beta$. This completes the proof.

\noindent\textbf{Lower bound:}
By the convolution formula and \Cref{prop:dp-tent-coal}, conditional on $\vec \cL^\beta\approx\vec h$, with probability at least $1-C\exp(-cL^{3/2})$,
\begin{align}
\cL^\beta[\vec x, L^{-1/2}(\log L)^2]
&\geq \sum_{i=1}^{k-1}\cL^\beta_i[x_i, L^{-1/2}(\log L)^2] - \sum_{i=2}^{k-1}\cL^\beta_i  - C\exp(-cL)\nonumber\\
&\geq \sum_{i=1}^{k-1}\cL^\beta_i[x_i, L^{-1/2}(\log L)^2] - \sum_{i=2}^{k-1}h_i  - C\exp(-cL).\label{e.L lower bound via coalescence}
\end{align}
Take $2\le i \le k-1$.
Consider the processes $\cL^\beta(x_{i-1}, s_{i-1}; \cdot, s_i)$ and $\cL^\beta(\cdot, s_{i-1}; x_i, s_i)$ conditional on $\cL^\beta_i\approx h_i$, as a measure on $\mc C([0,1],\R)^2$.
By \Cref{prop:dp-tent-bcomp}, it can be replaced by the product measure of its marginals (up to an error of $C\exp(-cL)$), on an event with probability $>1-C\exp(-cL^{3/2})$, with Radon-Nikodym derivative bounded between $1-C\exp(-cL)$ and $1+C\exp(-cL)$.
By \Cref{l.part free energy to whole free energy} (plus shear and translation invariance properties), 
we can lower bound the probability of
\begin{equation}  \label{eq:Lbetai}
    \cL^\beta_i[x_i, L^{-1/2}(\log L)^2] \ge h_i+h_{i+1}-\beta^{-1}\log(2H^{1/2}) + Mr_\beta, |\pi^*(s_i)-x_i|\le w_\beta,
\end{equation}
conditional on $\vec\cL^\beta\approx \vec h$, by $c\exp(-CM^2L^{1/2}r_\beta)-C\exp(-cL^{3/2})$.
(Here we also use that 
$$\log((s_i-s_{i-1})^{-1/2}h_i^{1/2}+(s_{i+1}-s_i)^{-1/2}h_{i+1}^{1/2}) = \log(2H^{1/2}) + O(L^{-1/9})$$
and $\beta^{-1}L^{-1/9}<r_\beta$.)
By the above coupling to independence, the probability of \eqref{eq:Lbetai} for each $i$, conditional on  $\vec\cL^\beta\approx \vec h$, is also $>c\exp(-CM^2L^{1/2}r_\beta)-C\exp(-cL^{3/2})$.
Plugging these into \eqref{e.L lower bound via coalescence} finishes the proof of \eqref{eq:kppwf2}.
\qed

\section{Global and segment maximizers}  \label{s:gasm}

In \Cref{p.closeness of maximizer and polymer} we have shown that (conditional on the upper tail) the polymer $\Gamma_0$ at any $s\in (0,1)$ is close to the maximizer $\pi(s)$, whose definition we recall is 
$$\pi(s) = \argmax_{x\in\R} \cL^\beta(0,0;x,s) + \cL^\beta(x,s; 0,1).$$
In zero temperature, we also use $\pi$ to denote the geodesic $\pi_0$, which is defined through the same expression.

As already alluded, in our computations to obtain finite dimensional Gaussianity, it is more convenient to use another optimizer
$$\pi^*(s) = \argmax_{x\in\R} \cL^\beta(0,s_-;x,s) + \cL^\beta(x,s; 0,s_+).$$
Here $s_-\in [0, s)$ and $s_+\in (s, 1]$.
In the rest of this section, all the constants can depend on $s_-, s, s_+$.

The main goal of this section is the following closeness between $\pi(s)$ and $\pi^*(s)$.
As in previous sections, we write $\cL^\beta=\cL^\beta(0,0;0,1)$.
\begin{proposition}\label{l.strip maximizer is same as top to intermediate maximizer}
For any $L\ge 2$, we have
\[
\P\left(|\pi(s) - \pi^*(s)| > L^{-1/2}(\log L)^2 \midd \cL^\beta>L \right) < C\exp(-c(\log L)^2).
\]
\end{proposition}

The case where $s_-=0$ and $s_+=1$ is obvious, since then $\pi(s) = \pi^*(s)$.
Below we write the proof for the case where $s_->0$ and $s_+<1$.
The proof for the other two cases are similar and we omit them.

For the rest of this section, we shall take the shorthand $\cL^\beta_i=\cL^\beta(x_{i-1},s_{i-1}; x_i, s_i)$ for some $\vec x = (x_1, x_2, x_3)\in \R^3$,  and $x_0=x_4=0$, $s_0=0$, $s_1=s_-$, $s_2=s$, $s_3=s_+$, $s_4=1$.
We also write $\smash{\vec\cL^\beta}$ for the vector $\{\cL^\beta_i\}_{i=1}^4$; in other words, we now have four parts instead of two like in previous sections.
For any vector $\vec h\in \R^k$,  $\vec\cL^\beta\approx \vec h$ denotes the event where $\cL^\beta_i\in h_i+[0, e^{-L}]$ for each $i$.

Our general idea is to replace the conditioning $\cL^\beta>L$ by $\vec\cL^\beta\approx \vec h$, for some reasonable $\vec x\in \R^3$ and $\vec h$.
Then we bound $|\pi(s)-x_2|$ conditional on $\vec\cL^\beta\approx \vec h$, using coalescence and the tent picture.

We first restrict $\pi^*$ to an interval of length $<L^{1/2}$.
\begin{lemma}  \label{lem:piss516}
For each $i=1, 2, 3$ and $L$ large, 
\[
\P\left( |\pi^*(s_i) | > L^{5/16} \log L  \midd \cL^\beta>L \right) < C \exp(-c(\log L)^2).
\]
\end{lemma}
\begin{proof}
We write the proof for $i=2$; the other two cases follow similarly.

Let $\cE$ denote the event where $|\pi^*(s) | > L^{5/16} \log L$.
By \Cref{l.inf control} we have $\P(\cE_1 \mid \cL^\beta > L) > 1- C\exp(-c(\log L)^2)$, where $\cE_1$ is the event
\[
 \inf_{|x_1|, |x_2|, |x_3| \leq L^{-1/4}\log L} \sum_{i=1}^4\cL^\beta_i \ge  L - L^{5/8}\log L.
\]
We have that $\cE\cap\cE_1\subset \cE_2$, where $\cE_2$ is the event
\[
\sup_{|x_2|>L^{5/16} \log L, x_1=x_3=0}\sum_{i=1}^4\cL^\beta_i \ge  L - L^{5/8}\log L.
\]
By the shear invariance property and \Cref{cor.p.tail bound for sum}, and a union bound, we have that
$\P(\cE_2)<C\exp\left(-\frac{4}{3}L^{3/2} -cL^{9/8}(\log L)^2 \right)$.
By combining this with the tail estimate \Cref{lem:fh-ut} we get that $\P(\cE_2\mid \cL^\beta>L)<C\exp(-cL^{9/8}(\log L)^2)$.
As we can upper bound $\P(\cE\mid \cL^\beta>L)$ by $\P(\cE_1^c \mid \cL^\beta>L) + \P(\cE_2\mid \cL^\beta>L)$, the conclusion follows.
\end{proof}

In replacing the conditioning $\cL^\beta>L$ into $\vec\cL^\beta\approx \vec h$, we will need to bound the ratio
$\P( \vec\cL^\beta\approx \vec h)\P(\cL^\beta>L)^{-1}$, for which we need that $\sum_{i=1}^4 h_i >L-\log L$.
For this we will invoke \Cref{p.sum close under conditioning k-point}, as well as a rough lower bound of $\cL^\beta_i$ conditional on the upper tail $\cL^\beta>L$.

\begin{lemma}  \label{l:rough-clbi}
For any $\vec x\in [-L^{5/16}\log L, L^{5/16}\log L]^3$, and $i\in\{1, 2, 3, 4\}$, we have
\[
\P\left( \cL^\beta_i<(s_i-s_{i-1})L - L^{27/32}\log L) \midd \cL^\beta>L \right) < C\exp(-c(\log L)^2).
\]
\end{lemma}
\begin{proof}
We prove this for $i=2$, and the other cases would follow verbatim.

We use $\cL^\beta_{j,0}$ to denote  $\cL^\beta_j$ with each of $x_1, x_2, x_3$ replaced by $0$.
Let $\cE$ denote the event $\cL^\beta_2<(s_2-s_1)L - L^{27/32}\log L$, and define $\cE_0$ by
\[
\cE_0 = \left\{\cL^\beta_{j,0}>(s_j-s_{j-1})L-L^{53/64}\right\},
\]
for each $j=1,2,3,4$.
By \Cref{l.inf control}, we have
\begin{equation}  \label{eq:l:rough-clbi1}
\P\left(\sum_{j=1}^4 \cL^\beta_{j,0} <L-L^{5/8}\log L \midd \cL^\beta>L\right) < C\exp(-c(\log L)^2).    
\end{equation}
By \Cref{l.k-point proportionality}, we have
\[
  \P\left(\cE_0^c, \sum_{j=1}^4 \cL^\beta_{j,0} \ge L-L^{5/8}\log L \right) \\ <C\exp(-cL^{37/32})\cdot\P\left(\sum_{j=1}^4 \cL^\beta_{j,0} \ge L-L^{5/8}\log L \right).  
\]
By \Cref{lem:fh-ut} and \Cref{cor.p.tail bound for sum}, we can bound the ratio of the last factor over $\P(\cL^\beta>L)$ by $C\exp(CL^{9/8}\log L)$.
Therefore we have
\[
 \P\left(\cE_0^c, \sum_{j=1}^4 \cL^\beta_{j,0} \ge L-L^{5/8}\log L \midd \cL^\beta>L \right) <C\exp(-cL^{37/32}).
\]
Combining this with \eqref{eq:l:rough-clbi1} implies that
$\P(\cE_0^c \mid \cL^\beta>L ) < C\exp(-c(\log L)^2)$.

We next upper bound $\P(\cE\mid \cE_0)$.
In the case where $x_1x_2 <0$,  by the shift-invariance (\Cref{lem:fr-shift})
\[
\P(\cE\mid \cE_0)=\P\left(\cL^\beta(0,s_1;x_2-x_1,s_2) < (s_2-s_1)L - L^{27/32}\log L  \midd \cE_0 \right) .
\]
By \Cref{l.control near tent center} and using that $|x_1|, |x_2|\le L^{5/16}\log L$, this is bounded by $C\exp(-cL^{11/8}\log L)$.

In the case where $x_1x_2\ge 0$, 
by \Cref{l.control near tent center} and using that $|x_1|, |x_2|\le L^{5/16}\log L$, we have
\[
\P\left(\cL^\beta(0,s_1;x_2,s_2) < (s_2-s_1)L - L^{27/32}(\log L)/2  \midd \cE_0  \right) < C\exp(-cL^{11/8}\log L),
\]
\[
\P\left(\cL^\beta(x_1,s_1;0,s_2) < (s_2-s_1)L - L^{27/32}(\log L)/2  \midd \cE_0  \right) < C\exp(-cL^{11/8}\log L).
\]
Then by the quadrangle inequalities \eqref{eq:DL-quad} and \eqref{eq:FR-quad}, we still have $\P(\cE\mid \cE_0)<C\exp(-cL^{11/8}\log L)$.
Then since (by \Cref{lem:fh-ut}) $\P(\cE_0)\P(\cL^\beta>L)^{-1}< C\exp(CL^{85/64})$,
we have
\begin{equation*}
    \begin{split}
\P(\cE\mid \cL^\beta>L) & < C\exp(-c(\log L)^2) +\P(\cE\cap \cE_0 \mid \cL^\beta>L)       \\
& \le C\exp(-c(\log L)^2) +\P(\cE\cap \cE_0)\P(\cL^\beta>L)^{-1}
\\
& \le C\exp(-c(\log L)^2) +\P(\cE\mid \cE_0)\P(\cE_0)\P(\cL^\beta>L)^{-1}\\
& <C\exp(-c(\log L)^2).
    \end{split}
\end{equation*}
Thus the conclusion follows.
\end{proof}

We next prove that, given the tail $\vec\cL^\beta \approx \vec h$, $\pi(s)$ should be close to $x_2$.
Define
\begin{multline*}
   \mrm{Val}^4=\Bigg\{ \vec h\in \R^4: h_i \ge (s_i-s_{i-1})L-L^{27/32}\log L, \; \forall i=1,2,3,4;\\ L-\log L \le \sum_{i=1}^4 h_i <L+(\log L)^2 \Bigg\}. 
\end{multline*}
\begin{lemma}  \label{lem:condfar}
For any $\vec x\in [-L^{5/16}\log L, L^{5/16}\log L]^3$, and $\vec h \in \mrm{Val}^4$,
\[
\P\left( |\pi(s)-x_2|> L^{-1/2}(\log L)^2/2 \midd \vec\cL^\beta \approx \vec h \right) < C\exp(-cL^{1/2}(\log L)^2).
\]
\end{lemma}
\begin{proof}
By \Cref{lem:fh-ut} we have 
\begin{equation}  \label{eq:bdcl4}
  \P(\vec\cL^\beta \approx \vec h ) > c\exp\left(-\frac{4}{3} L^{3/2}-CL^{9/8}(\log L)^2\right).      
\end{equation}
Therefore, using \Cref{l.free energy individual to overall values} we have
\[
\P\left( \cL^\beta(0,0;x_2,s) < h_1+h_2 - C\log L \midd \vec\cL^\beta \approx \vec h \right) < C\exp(-cL^2),
\]
\[
 \P\left( \cL^\beta(x_2, s; 0,1) < h_3+h_4 - C\log L \midd \vec\cL^\beta \approx \vec h \right) < C\exp(-cL^2).
\]
Thus 
\begin{equation}  \label{eq:bdcl4cent}
\P\left( \cL^\beta(0,0;x_2,s) + \cL^\beta(x_2, s; 0,1) < L - C\log L \midd \vec\cL^\beta \approx \vec h \right) < C\exp(-cL^2).    
\end{equation}

We next restrict the intervals. We define $\cL^\beta_{\mrm{out}}$ as follows. For $\beta=\infty$, we let
\[
\cL^\beta_{\mrm{out}}=\sup_{\vec x\in \R^3\setminus [-L^{3/8}, L^{3/8}]^3} \sum_{i=1}^4 \cL^\beta_i.
\]
For $\beta=1$, we let
\begin{multline*}
\cL^\beta_{\mrm{out}}=
\left(\sup_{|x_2|>L^{3/8}} \cL^\beta(x_2,s;0,1)+\cL^\beta(0,0;x_2,s)\right) \vee \\
\left(\sup_{|x_2|\le L^{3/8}} (\cL^\beta_{\mrm{out},-}(x_2)+\cL^\beta(x_2,s;0,1))\vee (\cL^\beta(0,0;x_2,s)+\cL^\beta_{\mrm{out},+}(x_2))\right),    
\end{multline*}
where 
\[
\cL^\beta_{\mrm{out},-}(x_2) = \log \int_{|x_1|\ge L^{3/8}} \exp\left(\cL^\beta_1 + \cL^\beta_2 \right) \diff x_1,
\]
\[
\cL^\beta_{\mrm{out},+}(x_2) = \log \int_{|x_3|\ge L^{3/8}} \exp\left(\cL^\beta_3 + \cL^\beta_4 \right) \diff x_3.
\]
Using \Cref{cor.p.tail bound for sum} and shear invariance, and a union bound, we can upper bound $\P(\cL^\beta_{\mrm{out}}> L - (\log L)^{3/2})$ by $C\exp\left(-\frac{4}{3}L^{3/2}-cL^{5/4}(\log L)^2\right)$.
Then by \eqref{eq:bdcl4} we have 
\begin{equation}  \label{eq:cloutbound}
\P\left(\cL^\beta_{\mrm{out}}> L - (\log L)^{3/2} \mid \vec\cL^\beta \approx \vec h\right)<C\exp(-cL^{5/4}(\log L)^2).    
\end{equation}

We next consider the restricted part.
\[
\cL^\beta_{\mrm{res}}(x_2)=\max_{x_1, x_3\in [-L^{3/8}, L^{3/8}]} \sum_{i=1}^4 \cL^\beta_i,
\]
for $\beta=\infty$, and
\[
\cL^\beta_{\mrm{res}}(x_2)=
\log \int_{|x_1|\le L^{3/8}} \exp\left(\cL^\beta_1 + \cL^\beta_2 \right) \diff x_1 + \log \int_{|x_3|\le L^{3/8}} \exp\left(\cL^\beta_3 + \cL^\beta_4 \right) \diff x_3,
\]
for $\beta=1$.
Using \Cref{prop:dp-tent-coal} and \Cref{cor.tent-pr}, one can upper bound $\cL^\beta_{\mrm{res}}$ assuming a coalescence event, as in the proof of \Cref{p.k-point part to whole free energy}. Omitting the details, we can deduce that
\[
\P\left(\sup_{x\in I}  \cL^\beta_{\mrm{res}}(x) > L - (\log L)^{3/2} \midd \vec\cL^\beta\approx \vec h \right) < C\exp(-cL^{1/2}(\log L)^2),
\]
where $I=[-L^{3/8}, L^{3/8}]\setminus [x_2-L^{-1/2}(\log L)^2/2, x_2+L^{-1/2}(\log L)^2/2]$.
This with \eqref{eq:cloutbound} implies that
\begin{multline*}
\P\left(\sup_{|x-x_2|>L^{-1/2}(\log L)^2/2}  \cL^\beta(0,0;x_2,s) + \cL^\beta(x_2, s; 0,1) > L - (\log L)^{3/2} \midd \vec\cL^\beta\approx \vec h \right) \\ < C\exp(-cL^{1/2}(\log L)^2).    
\end{multline*}
This and \eqref{eq:bdcl4cent} imply the conclusion.
\end{proof}

We now finish proving \Cref{l.strip maximizer is same as top to intermediate maximizer}, using all the above ingredients as well as \Cref{p.sum close under conditioning k-point}, which lower bounds the sum $\sum_{i=1}^4 \cL^\beta_i$ by $L-\log L$, given that $\cL^\beta>L$.

\begin{proof}[Proof of \Cref{l.strip maximizer is same as top to intermediate maximizer}]
We can assume that $L$ is large enough, since otherwise the conclusion follows trivially.

Denote $w_{\beta=1}=L^{-1/2}$ and $w_{\beta=\infty}=L^{-1}$ as in the previous section.
By Lemmas~\ref{lem:piss516} and \ref{l:rough-clbi}, we can bound the LHS in the display by
\begin{multline}  \label{eq:strplhsm}
\sum_{\vec x\in ([-L^{5/16}\log L, L^{5/16}\log L]\cap w_\beta\Z)^3} \P\Big( \max_{i=1,2,3}|\pi^*(s_i)-x_i| \le w_\beta, |\pi(s)-x_2|>L^{-1/2}(\log L)^2/2 
 \\
 \cL^\beta_i \ge (s_i-s_{i-1})L - L^{27/32}\log L, \forall i=1,2,3,4
 \midd \cL^\beta>L \Big) + C\exp(-c(\log L)^2). 
\end{multline}
By \Cref{l.free energy individual to overall values} and \Cref{t.comparison}, we have 
\[
\P\left(\sum_{i=1}^4 \cL^\beta_i \ge L+(\log L)^2 \midd \cL^\beta>L  \right) < C\exp(-cL^{1/2}(\log L)^2).
\]
This together with \Cref{p.sum close under conditioning k-point} implies that, for any $\vec x\in ([-L^{5/16}\log L, L^{5/16}\log L]\cap w_\beta\Z)^3$,
\begin{multline*}
\P\Big(
\max_{i=1,2,3}|\pi^*(s_i)-x_i| \le w_\beta,  
\vec\cL^\beta \not\in \mrm{Val}^4,
\\
 \cL^\beta_i \ge (s_i-s_{i-1})L - L^{27/32}\log L, \forall i=1,2,3,4
\midd \cL^\beta>L
\Big)<C\exp(-c(\log L)^2) .
\end{multline*}
Then each summand in \eqref{eq:strplhsm} can be bounded by
\begin{multline*}
\sum_{\vec h}\P\Big(
\max_{i=1,2,3}|\pi^*(s_i)-x_i| \le w_\beta,  |\pi(s)-x_2|>L^{-1/2}(\log L)^2/2,
\vec\cL^\beta \approx \vec h,
\midd \cL^\beta>L
\Big)\\ +C\exp(-c(\log L)^2),
\end{multline*}
where the sum is over $<\exp(CL)$ many $\vec h\in \mrm{Val}^4$, such that any element in $\mrm{Val}^4$ is $\approx \vec h$ for one of them. 
We note that the summand for each $\vec h$ is bounded by
\[
\P\left(
|\pi(s)-x_2|>L^{-1/2}(\log L)^2/2
\midd \vec\cL^\beta \approx \vec h
\right)
\P(\vec\cL^\beta \approx \vec h)\P(\cL^\beta>L)^{-1}.\]
By \Cref{lem:condfar}, the first factor is bounded by $C\exp(-cL^{1/2}(\log L)^2)$.
Then the sum over $\vec h$ is bounded by
\[
C\exp(-cL^{1/2}(\log L)^2)\P( \vec\cL^\beta \in \mrm{Val}^4)\P(\cL^\beta>L)^{-1}.
\]
By \Cref{l.free energy individual to overall values} and \Cref{t.comparison}, we can bound this by $C\exp(CL^{1/2}\log L)$, since $\vec\cL^\beta \in \mrm{Val}^4$ implies that $\sum_{i=1}^4 \cL^\beta_i\ge L-\log L$.
Now since each summand in \eqref{eq:strplhsm} is bounded by $C\exp(-c(\log L)^2)$, plugging this estimate back leads to the conclusion.
\end{proof}

\section{Finite dimensional Brownian bridge limit}
\label{s.fdd convergence}

In this section, we prove finite dimensional convergence of $\pi_0$ and $\Gamma_0$ to Brownian bridge (under upper tails).
As in previous sections we take the setup of $(s_1, \cdots, s_{k-1})\in \rDe_{k-1}([0,1])$ for $k\in \N$, and denote $s_0=0$ and $s_k=1$; and all the constants within this section can depend on $k$ and $(s_1, \cdots, s_{k-1})$.
Also recall that we define
\[\pi^*(s_i) = \argmax_{x} \cL^\beta(0, s_{i-1}; x, s_i) + \cL^\beta(x, s_i; 0, s_{i+1}),\]
for each $i=1,\ldots, k-1$. We adopt the shorthand $\cL^\beta = \cL^\beta(0,0;0,1)$.

We shall prove that as $L\to\infty$, $L^{1/4}\{\pi^*(s_i)\}_{i=1}^{k-1}$ conditional on $\cL^\beta>L$ converges to a joint Gaussian, matching that for a Brownian bridge. 
We note that even in zero temperature, $\pi^*(s_i)$ does not a priori coincide with $\pi_0(s_i)=\pi(s_i)$, which is instead given as the maximizer of processes from height $0$ to $s_i$ and $s_i$ to $1$ (instead of $s_{i-1}$ to $s_i$ and $s_i$ to $s_{i+1}$ here), i.e., $\pi(s_i) = \argmax_z \cL^\beta(0, 0; z, s_i) + \cL^\beta(z, s_i; 0, 1)$.
However, we have shown that, conditional on the upper tail, with high probability $|\pi^*(s_i)-\pi(s_i)|$ is of order smaller than $L^{-1/4}$ (\Cref{l.strip maximizer is same as top to intermediate maximizer}); and in positive temperature the polymer measure concentrates at height $s_i$ in a window of order smaller than $L^{-1/4}$ around $\pi(s_i)$ (\Cref{p.closeness of maximizer and polymer}).
Therefore, the Gaussian limit of $\{\pi^*(s_i)\}_{i=1}^{k-1}$ as follows suffices for us to deduce our main results.

We denote $w_{\beta=1}=L^{-1/2}$ and $w_{\beta=\infty}=L^{-1}$ as before.
Fix any compact set $\cK\subseteq \R^{k-1}$, and all the constants below can depend on $\cK$.
\begin{theorem}\label{t.k-point density}
As $L\to\infty$, uniformly over $\vec x, \vec y\in \cK$ (with $x_0=x_k=y_0=y_k=0$ for the convenience of notations),
\begin{align*}
\MoveEqLeft[18]
\frac{\P\left(\max_{i=1, \ldots, k-1}|\pi^*(s_i)-x_iL^{-1/4}| \leq w_\beta\mid \cL^\beta > L\right)}{\P\left(\max_{i=1, \ldots, k-1}|\pi^*(s_i) - y_iL^{-1/4}| \leq w_\beta \mid \cL^\beta > L\right)}\\
&\to \exp\left(-2\left[\sum_{i=1}^k \frac{(x_i-x_{i-1})^2 - (y_i-y_{i-1})^2}{s_i-s_{i-1}}\right]\right).
\end{align*}
\end{theorem}
We note that the RHS is the ratio of the joint density of $(\frac{1}{2}B(s_1), \ldots, \frac{1}{2}B(s_{k-1}))$ evaluated at $\vec x$ and $\vec y$, where $B$ is a standard Brownian bridge on $[0,1]$.

Note that this is a comparison of probabilities, while we wish to show weak convergence. The following lemma allows the transition. Its proof is fairly straightforward real analysis and we relegate it to Appendix~\ref{s.abstract weak convergence}.

\begin{lemma}\label{l.fdd convergence}
Let $d\geq 1$ and suppose $\{\vec X_\varepsilon\}_{\varepsilon >0}$ is a family of $\R^d$-valued random vectors such that as $\varepsilon\to 0$, $\vec X_\varepsilon \to \vec X$ in distribution for some random vector \smash{$\vec X$}. Suppose also that there is a continuous strictly positive integrable function $f:\R^d\to(0,\infty)$ such that, for every compact set $\cK\subseteq \R^d$, uniformly over $\vec x,\vec y\in \cK$ as $\varepsilon\to0$,
\begin{equation}\label{e.local density hypothesis}
\frac{\P\left(\vec X_\varepsilon \in \vec x + [-\varepsilon,\varepsilon)^d\right)}{\P\left(\vec X_\varepsilon \in \vec y + [-\varepsilon,\varepsilon)^d\right)}\cdot \frac{f(\vec y)}{f(\vec x)} \to 1.
\end{equation}
Then $\vec X$ is absolutely continuous with respect to the Lebesgue measure of $\R^d$, and has density given by $f(\vec x)/\int_{\R^d} f(\vec z)\diff \vec z$.
\end{lemma}

With these results in hand, we may prove our main theorems.

\begin{proof}[Proofs of Theorems~\ref{thm:main-dl} and \ref{thm:main-dp}]
The tightness of $\{L^{1/4}\pi_0\}_{L\geq 2}$ and $\{L^{1/4}\Gamma_0\}_{L\geq 2}$ conditional on $\cL^\beta>L$ in the space $\mc C([0,1], \R)$ is given by \Cref{p.tightness}, and it remains to establish finite dimensional convergence.

For $B$ being a standard Brownian bridge on $[0,1]$,
\Cref{t.k-point density} and \Cref{l.fdd convergence} imply that $\{2L^{1/4}\pi^*(s_i)\}_{i=1}^{k-1}\to \{B(s_i)\}_{i=1}^{k-1}$ in distribution. \Cref{l.strip maximizer is same as top to intermediate maximizer} along with the Borel-Cantelli lemma guarantees that $L^{1/4}\pi(s_i) - L^{1/4}\pi^*(s_i) \to 0$ almost surely for every $i$. This implies that, for $\beta=\infty$, $\{2L^{1/4}\pi_0(s_i)\}_{i=1}^{k-1}\to \{B(s_i)\}_{i=1}^{k-1}$ in distribution,
completing the proof of \Cref{thm:main-dl}.

For $\beta=1$, taking expectations in \Cref{p.closeness of maximizer and polymer} further yields that $\P(\max_{i=1,\ldots, k-1}|\Gamma_0(s_i)-\pi^*(s_i)| \leq ML^{-1/2}\log L\mid \cL^\beta>L)>1-L^{-cM}-C\exp(-c(\log L)^2)$, for $M$ being a large enough constant. Thus $L^{1/4}\Gamma_0(s_i) - L^{1/4}\pi^*(s_i)\to 0$ almost surely for every $i$. We then obtain that $\{2L^{1/4}\Gamma_0(s_i)\}_{i=1}^{k-1}\to \{B(s_i)\}_{i=1}^{k-1}$ in distribution,
completing the proof of \Cref{thm:main-dp}.
\end{proof}

The rest of this section and the next are devoted to proving \Cref{t.k-point density}.

\subsection{Finite dimensional convergence}

We introduce some useful shorthand to make the notation simpler: $\vec\cL^{\beta, \vec x}$ is a vector whose $i$\textsuperscript{th} component is given by
\begin{align*}
\cL^{\beta, \vec x}_i = \cL^{\beta}(x_{i-1}L^{-1/4}, s_{i-1}; x_iL^{-1/4}, s_i) 
\end{align*}
Below we also use $o(1)$ to denote any quantity that $\to 0$ as $L\to\infty$.

\begin{proof}[Proof of \Cref{t.k-point density}]
We start by noting that,
by \Cref{l.inf control},
\[
\P\left(\sum_{i=1}^k\cL^{\beta, \vec x}_i < L - L^{5/8}\log L \midd \cL^\beta > L\right) < C\exp(-c(\log L)^2);
\]
and by \Cref{l.free energy individual to overall values}, plus \Cref{lem:fh-ut} and \Cref{t.comparison}, for some $C_0>0$
\[
\P\left(\sum_{i=1}^k\cL^{\beta, \vec x}_i > L + C_0\log L \midd \cL^\beta > L\right) < C\exp(-cL^{1/2}\log L).
\]
So we can upper bound $\P(\max_{i=1, \ldots, k-1}|\pi^*(s_i)-x_iL^{-1/4}| \leq w_\beta,  \cL^\beta > L)$ by
\begin{equation*}
\P\left(\mc E_{\mrm{sum}}, \max_{i=1, \ldots, k-1}|\pi^*(s_i)-x_iL^{-1/4}|<w_\beta, \cL^\beta > L\right)+ C\exp(-c(\log L)^2)\cdot \P(\cL^\beta > L),
\end{equation*}
where $\mc E_{\mrm{sum}}$ is the event
\[
L - L^{5/8}\log L\le \sum_{i=1}^k\cL^{\beta, \vec x}_i \le  L + C_0\log L.
\]
Take $M_*$ to be a large constant, and let $\mc E_{\mrm{prop}}(M_*)$ be defined by
\begin{align*}
\mc E_{\mrm{prop}}(M_*) = \bigcap_{i=1}^k\left\{\cL^{\beta, \vec x}_i \geq (s_i-s_{i-1})L- M_*L^{7/8}\right\}.
\end{align*}
Using \Cref{l.k-point proportionality} (proportionality statement), and \Cref{cor.p.tail bound for sum} and \Cref{lem:fh-ut} (tail bounds for the sum $\sum_{i=1}^k \cL^{\beta,\vec x}_i$ and $\cL^\beta$), we have
\begin{align*}
\MoveEqLeft[6]
\P\left(\mc E_{\mrm{prop}}^c, \sum_{i=1}^k\cL^{\beta, \vec x}_i \geq L-L^{5/8}\log L\right)\\
&\leq\exp\left(-cM_*^2L^{5/4}\right)\cdot\P\left(\sum_{i=1}^k\cL^{\beta, \vec x}_i \geq L-L^{5/8}\log L\right)\\
&\leq \exp\left(-cM_*^2L^{5/4} - \frac{4}{3}(L-L^{5/8}\log L)^{3/2} + CL^{3/4}\right)\\
&\leq \exp\left(-cM_*^2L^{5/4}+ CL^{9/8}\log L\right)\cdot\P\left(\cL^\beta > L\right).
\end{align*}
As $M_*$ is large enough, this is upper bounded by $\exp(-cM_*^2L^{5/4})\cdot \P(\cL^\beta > L)$.

Thus far, we have overall shown that
\begin{align*}
&\P\left(\max_{i=1, \ldots, k-1}|\pi^*(s_i) - x_iL^{-1/4}| < w_\beta, \cL^\beta > L\right)\\
&\leq \P\left(\mc E_{\mrm{prop}}(M_*), \mc E_{\mrm{sum}}, \max_{i=1, \ldots, k-1}|\pi^*(s_i) - x_iL^{-1/4}|<w_\beta,  \cL^\beta > L\right)+ C\exp(-c(\log L)^2)\cdot\P(\cL^\beta > L).
\end{align*}
Set (as before) $r_{\beta=1} = 1$ and $r_{\beta=\infty}=L^{-1/2}$. 
Using \Cref{p.sum close under conditioning k-point}, we conclude that 
\begin{align}
\MoveEqLeft[0]
\P\left(\max_{i=1, \ldots, k-1}|\pi^*(s_i) - x_iL^{-1/4}| < w_\beta, \cL^\beta > L\right)\nonumber\\
&\leq \bigl(1-e^{-cM_*^2L^{1/2}r_\beta}\bigr)^{-1}
\P\left(\parbox{4.1in}{\centering$ \mc E_{\mrm{prop}}(M_*),\mc E_{\mrm{sum}}, \max_{i=1, \ldots, k-1}|\pi^*(s_i) - x_iL^{-1/4}|<w_\beta, \cL^\beta > L$\\[4pt]
$\sum_{i=1}^k\cL^{\beta, \vec x}_i > L+ (k-1)\beta^{-1}\log(2L^{1/2}) -M_*r_\beta$}\right) \label{e.massaged location prob to analyze}\\[4pt]
&\qquad\quad + C\exp(-c(\log L)^2)\cdot\P\left(\cL^\beta > L\right).\nonumber
\end{align}
We next do a restriction. Recall the notation $\maxbeta_{\vec x\in I} f$ for $f:\R^{k-1}\to \R$ and a set $I\subseteq \R^{k-1}$ from \eqref{e.maxbeta definition}.
Define
\begin{equation}  \label{eq:clbetaone}
\cL^{\beta}[1] = \maxbeta_{\|\vec z\|_\infty \le 1} \sum_{i=1}^k\cL^\beta(s_{i-1}, z_{i-1}; s_i, z_i),
\end{equation}
where $z_0=z_k=0$.
Then, by the (quenched) one-point transversal fluctuation estimate (\Cref{p.positive temp one-point tf} for $\beta = 1$ and \Cref{l.trans-fluc} for $\beta =\infty$), \eqref{e.massaged location prob to analyze} is bounded by
\begin{equation}\label{e.next massaged version}
\begin{split}
\MoveEqLeft
\bigl(1-e^{-cM_*^2L^{1/2}r_\beta}\bigr)^{-1}
\P\left(\parbox{3.65in}{\centering$ \mc E_{\mrm{prop}}(M_*), \mc E_{\mrm{sum}}, \max_{i=1, \ldots, k-1}|\pi^*(s_i) - x_iL^{-1/4}|<w_\beta$,\\ 
$\cL^\beta[1] > (1-\beta^{-1}e^{-cL^{1/4}})L$,\\
$\sum_{i=1}^k\cL^{\beta, \vec x}_i > L+ (k-1)\beta^{-1}\log(2L^{1/2}) -M_*r_\beta$}\right)\\
&\qquad + C\exp(-c(\log L)^2)\cdot\P\left(\cL^\beta > L\right).
\end{split}
\end{equation}
To analyze the first term in the previous display, we define
\begin{equation}  \label{eq:eeeval}
E_{\mrm{val}} = \left\{(h_1, \ldots , h_k)\in (e^{-L}\Z)^k: \parbox{110mm}{\centering $L+ (k-1)\beta^{-1}\log(2L^{1/2})- (M_*+1)r_{\beta}\le \sum_{i=1}^k h_i \le L+C_0\log L$, \\$h_i\in (s_i-s_{i-1})L + [-2M_*L^{7/8}, kM_*L^{7/8}]$ for $i=1, \ldots, k$}\right\}.    
\end{equation}
Here we take the fine mesh $(e^{-L}\Z)^k$ instead of $\R^k$ because (as mentioned before) we will need to apply \Cref{prop:dp-tent-coal} and \Cref{prop:dp-tent-bcomp}. We also recall the notation $\vec\cL^{\beta, \vec x} \approx \vec h$ for the event that each $\cL^{\beta, \vec x}_i \in h_i+[0,e^{-L}]$.

Now we see that the probability in the first term of \eqref{e.next massaged version} is bounded by
\begin{equation}
\sum_{\vec h\in E_{\mrm{val}}} \P\Bigl(\max_{i=1, \ldots, k-1}|\pi^*(s_i)-x_iL^{-1/4}|<w_\beta, \cL^\beta[1] > (1-\beta^{-1}e^{-cL^{1/4}})L  \midd  \vec\cL^{\beta, \vec x} \approx \vec h\Bigr)
\cdot \P\Bigl(\vec\cL^{\beta, \vec x} \approx \vec h\Bigr). \label{e.maximizer location probability breakup}
\end{equation}
Our goal is to relate each of these terms to the corresponding one with $\vec x$ replaced by $\vec y$, which we state precisely in the next two lemmas to be proved later. The first factor is essentially unchanged:
\begin{lemma}\label{l.maximizer location prob relation}
For $\vec h=(h_1, \ldots, h_k)\in E_{\mrm{val}}$,
\begin{align*}
\MoveEqLeft[4]
\P\Bigl(\max_{i=1, \ldots, k-1}|\pi^*(s_i)-x_iL^{-1/4}|<w_\beta, \cL^\beta[1] > (1-\beta^{-1}e^{-cL^{1/4}})L \ \Big|\  \vec\cL^{\beta, \vec x} \approx \vec h\Bigr)\\
&\leq (1+o(1))\P\Bigl(\max_{i=1, \ldots, k-1}|\pi^*(s_i)-y_iL^{-1/4}|<w_\beta, \cL^\beta[1] > L \ \Big|\  \cL^{\beta, \vec y} \approx \vec h\Bigr).
\end{align*}

\end{lemma}
The proof of this lemma is somewhat involved and will be given in \Cref{s.joint comparison of probabilities}.
For the second factor in \eqref{e.maximizer location probability breakup}, we have the following statement, which is the source of the Brownian bridge density in our result. Its proof is a straightforward consequence of \Cref{t.comparison}, and we give it after completing the proof of \Cref{t.k-point density}. 
\begin{lemma}\label{l.h_i prob comparison}
For $(h_1, \ldots, h_k)\in E_{\mrm{val}}$,
\begin{align*}
\P\left(\vec\cL^{\beta, \vec x} \approx \vec h\right)
&= (1+o(1))\exp\left(-2\sum_{i=1}^k\frac{(x_i-x_{i-1})^2 - (y_i-y_{i-1})^2}{s_i-s_{i-1}}\right)\cdot \P\left(\cL^{\beta, \vec y} \approx \vec h\right).
\end{align*}

\end{lemma}

Inputting the information from the previous two lemmas into \eqref{e.maximizer location probability breakup} and \eqref{e.next massaged version} yields that 
\begin{align*}
\MoveEqLeft[2]
\P\left(\max_{i=1, \ldots, k-1}|\pi^*(s_i) - x_iL^{-1/4}| \leq w_\beta, \cL^\beta>L\right)\\
&< C\exp(-c(\log L)^2)\cdot\P\left(\cL^\beta > L\right)\\
&\quad+ (1+o(1))\left(1+\exp(-cM_*^2L^{1/2}r_\beta)\right)\exp\left(-2\sum_{i=1}^k\frac{(x_i-x_{i-1})^2-(y_i-y_{i-1})^2)}{s_i-s_{i-1}}\right)\\
&\qquad\times\sum_{\vec h\in E_{\mrm{val}}} \P\Bigl(\max_{i=1, \ldots, k-1}|\pi^*(s_i)-y_iL^{-1/4}|<w_\beta, \cL^\beta[1] > L \ \Big|\  \cL^{\beta, \vec y} \approx \vec h\Bigr) \cdot \P(\cL^{\beta, \vec y} \approx \vec h).
\end{align*}
Therefore we have, since $\{\cL^{\beta,\vec y} \approx \vec h\}$ are disjoint events for distinct $\vec h \in E_{\mrm{val}}$,
\begin{equation}  \label{eq:bdraxy}
\begin{split}
\MoveEqLeft[4]
\P\left(\max_{i=1, \ldots, k-1}|\pi^*(s_i)-x_iL^{-1/4}| \leq w_\beta \midd \cL^\beta > L\right)\\ 
&\leq (1+o(1))\left(1+e^{-cM_*^2L^{1/2}r_\beta}\right)\exp\left(-2\sum_{i=1}^k\frac{(x_i-x_{i-1})^2-(y_i-y_{i-1})^2)}{s_i-s_{i-1}}\right)\\
&\qquad\times \P\left(\max_{i=1, \ldots, k-1}|\pi^*(s_i)-y_iL^{-1/4}| \leq w_\beta \midd \cL^\beta > L\right)
+ C\exp(-c(\log L)^2).
\end{split}
\end{equation}
We need to lower bound $\P(\max_{i=1, \ldots, k-1}|\pi^*(s_i)-y_iL^{-1/4}| \leq w_\beta\mid \cL^\beta > L)$ in order to ensure that the error term $C\exp(-c(\log L)^2)$ is not dominating. For this we sum the previous display over $\vec x \in [-KL^{-1/4}, KL^{-1/4}]^k\cap (w_\beta Z)^k$, where $K$ is a large constant such that 
$$\sum_{\vec x\in [-KL^{-1/4}, KL^{-1/4}]^k\cap (w_\beta Z)^k}\P\left(\max_{i=1, \ldots, k-1}|\pi^*(s_i)-x_iL^{-1/4}| \leq w_\beta\midd \cL^\beta > L\right) \geq \frac{1}{2};$$
that this is possible follows from combining  transversal fluctuation bounds for $\pi(s_i)$ (\Cref{p.closeness of maximizer and polymer} and \Cref{p.maximizer tf} for $\beta=1$ or \Cref{l.trans-fluc} for $\beta=\infty$) and the bound on the closeness of $\pi(s_i)$ and $\pi^*(s_i)$ (Proposition~\ref{l.strip maximizer is same as top to intermediate maximizer}). Then from \eqref{eq:bdraxy} we have that
\begin{align*}
\left(\tfrac{1}{2}-C\exp(-c(\log L)^2)\right)
&\leq C\cdot\P\left(\max_{i=1, \ldots, k-1}|\pi^*(s_i)-y_iL^{-1/4}| \leq w_\beta\midd \cL^\beta > L\right)\\
&\quad\times \sum_{\vec x\in [-KL^{-1/4}, KL^{-1/4}]^k\cap (w_\beta Z)^k} \exp\left(-2\sum_{i=1}^k\frac{(x_i-x_{i-1})^2-(y_i-y_{i-1})^2)}{s_i-s_{i-1}}\right).
\end{align*}
This yields that $\P(\max_{i=1, \ldots, k-1}|\pi^*(s_i)-y_iL^{-1/4}| \leq w_\beta\mid \cL^\beta > L)$ is lower bounded by a polynomial in $L^{-1}$.
Thus we conclude that
\begin{align*}
\MoveEqLeft[4]
\frac{\P\left(\max_{i=1, \ldots, k-1}|\pi^*(s_i)-x_iL^{-1/4}| \leq w_\beta \mid \cL^\beta > L\right)}{\P\left(\max_{i=1, \ldots, k-1}|\pi^*(s_i)-y_iL^{-1/4}| \leq w_\beta \mid \cL^\beta > L\right)}\\ 
&\leq (1+o(1))\left(1+e^{-cM_*^2L^{1/2}r_\beta}\right)\exp\left(-2\sum_{i=1}^k\frac{(x_i-x_{i-1})^2-(y_i-y_{i-1})^2)}{s_i-s_{i-1}}\right)\\
&+C\exp(-c(\log L)^2).
\end{align*}
We can then get a lower bound of the same ratio by swapping $\vec x$ and $\vec y$.
Taking $L\to\infty$ followed by $M_*\to\infty$ completes the proof.
\end{proof}

\begin{proof}[Proof of Lemma~\ref{l.h_i prob comparison}]
It suffices to prove the case $\vec y = \vec 0$ and apply the resulting statement twice. By independence, $\P\left(\vec\cL^{\beta, \vec x} \approx \vec h\right) = \prod_{i=1}^k \P\left(\cL^{\beta, \vec x}_i \in h_i+[0, e^{-L}]\right)$. So we have to show that
\begin{align*}
\P\left(\cL^{\beta, \vec x}_i \in h_i + [0, e^{-L}]\right) = (1+o(1))\exp\left(-\frac{2x_i^2}{s_i-s_{i-1}}\right)\cdot \P\left(\cL^{\beta, \vec 0}_i \in h_i + [0, e^{-L}]\right).
\end{align*} 
Now, by shear invariance and using Theorem~\ref{t.comparison} in the third line,
\begin{align*}
\MoveEqLeft[2]
\P\left(\cL^{\beta, \vec x}_i \in h_i + [0, e^{-L}]\right)\\
&= \P\left(\cL^{\beta, \vec x}_i \geq h_i\right) - \P\left(\cL^{\beta, \vec x}_i \geq h_i+e^{-L}\right)\\
&= \P\left(\cL^{\beta, \vec 0}_i \geq h_i + \frac{(x_i-x_{i-1})^2L^{-1/2}}{s_i-s_{i-1}}\right) - \P\left(\cL^{\beta, \vec 0}_i \geq h_i + e^{-L} + \frac{(x_i-x_{i-1})^2L^{-1/2}}{s_i-s_{i-1}}\right)\\
&= (1+o(1))\exp\left(-2(s_i-s_{i-1})^{-3/2}h_i^{1/2}(x_i-x_{i-1})^2L^{-1/2}\right)\cdot\P\left(\cL^{\beta, \vec 0}_i \geq h_i \right)\\
&\quad - (1+o(1))\exp\left(-2(s_i-s_{i-1})^{-3/2}(h_i+e^{-L})^{1/2}(x_i-x_{i-1})^2L^{-1/2}\right)\cdot\P\left(\cL^{\beta, \vec 0}_i \geq h_i+e^{-L}\right).
\end{align*}
Since $(h_i+e^{-L})^{1/2} = h_i^{1/2}+O(e^{-L})$, and  $h_i^{1/2} = (s_i-s_{i-1})^{1/2}L^{1/2} + o(L^{1/2})$ due to that $h_i = (s_i-s_{i-1})L + o(L)$, the previous display equals
\begin{align*}
(1+o(1))\exp\left(-\frac{2(x_i-x_{i-1})^2}{s_i - s_{i-1}}\right)\left[\P\left(\cL^{\beta, \vec 0}_i \geq h_i\right) - \P\left(\cL^{\beta, \vec 0}_i \geq h_i+e^{-L}\right)\right],
\end{align*}
which is what we wanted to show.
\end{proof}

\section{Joint comparison of maximizer location and free energy across peaks}\label{s.joint comparison of probabilities}

In this section, we give the proof of Lemma~\ref{l.maximizer location prob relation}. 
We use the setup there: in particular, $C_0, M_*$ are large constants, $L$ is taken to be large (depending on $C_0, M_*$), $\vec h\in E_{\mrm{val}}$ for $E_{\mrm{val}}$ defined in \eqref{eq:eeeval},  
$\vec x, \vec y \in \cK$ for a compact set $\cK\subseteq \R^{k-1}$, and 
$(s_1, \cdots, s_{k-1})\in \rDe_{k-1}([0,1])$ for $k\in \N$, with $s_0=0$ and $s_k=1$.
All the constants within this section can depend on $C_0$, $M_*$, $\cK$, $k$, and $(s_1, \cdots, s_{k-1})$.

The proof strategy is to do a resampling on a small interval $I$ (with size of order $L^{-1/2}\log L$) in a way that a certain conditional probability of the event in question (the location of $\pi^*(s_i)$ for $i=1, \ldots, k$ and a lower bound on $\cL^\beta[1]$) is a function of the endpoint values at the boundary of $I$. The proof then comes down to showing that the density of these endpoint values under the conditioning $\smash{\cL^{\vec x} \approx \vec h}$ (which recall is shorthand for $\cL^{\beta,\vec x}_i \in h_i+[0,e^{-L}]$ for $i=1, \ldots, k$) is $1+o(1)$ of the density of the same under the conditioning $\cL^{\vec y} \approx \vec h$.

However, note that neither of the events $\max_{i=1, \ldots, k-1}|\pi^*(s_i) - x_iL^{-1/4}| < w_{\beta}$ and $\cL^\beta[1] > (1-\beta^{-1}\exp(-cL^{1/4}))L$ are functions of the profile on an interval of size of order $L^{-1/2}\log L$. Thus we will need to first argue that we can consider different events which do have this localized property. For the first event, we simply consider the maximizer on $I$ as a proxy, instead of on $\R$; clearly, if the first event holds, then it also holds that the restricted maximizer is $w_\beta$-close to $x_iL^{-1/4}$. We will then show (Lemma~\ref{l.restricted maximizer is maximizer}) that with high probability the restricted maximizer and true maximizer coincide. In the $\beta = \infty$ case, on this event we also have that $\cL^\beta[1]$ is a function of the profile on an interval of size of order $L^{-1/2}\log L$.

Modifying the second event, $\cL^\beta[1] > (1-\exp(-cL^{1/4}))L$, to be a local function of the profile in the $\beta = 1$ case is more difficult. A naive argument (invoking \Cref{p.closeness of maximizer and polymer} and \Cref{l.strip maximizer is same as top to intermediate maximizer}) shows that the free energy when restricted to an interval of size $ML^{-1/2}\log L$ (for a large number $M$) would capture a $(1-L^{-M})$ fraction of the total free energy. But at the end of the comparison we need to be able to return to the event of being larger than $\cL^\beta[1]>L$, and so we need to argue that, conditional on $\cL^{\vec x} \approx \vec h$ and $\cL^\beta[1] > (1-\exp(-cL^{1/4}))L$ (as well as the location of the restricted maximizer), with high probability $\cL^\beta$ will actually be larger than $L + L^{-E_*}$ for some $E_*$, i.e., the free energy overshoots by a polynomial amount. Then we know that restricting to a $ML^{-1/2}\log L$ window still results in the free energy being larger than $L$ (by picking $M$ large enough), and we can do our localized comparison argument with this event. 

In the following \Cref{s.conditional overshoot} we give the argument for saying that, conditional on $\cL^{\vec x} \approx \vec h$ and $\cL^\beta[1] > (1-\exp(-cL^{1/4}))L$, it holds with high probability that $\cL^\beta[1] > L+L^{-E_*}$ (Lemma~\ref{l.L overshoot}). In Section~\ref{s.comparison.restricted free energy} we argue that the free energy (in a form suitable for analysis in the upcoming proof of Lemma~\ref{l.maximizer location prob relation}) from the smaller interval of size $ML^{-1/2}\log L$ is also $L+L^{-E_*}$ (Lemma~\ref{l.conditional localization of event}). Then in Section~\ref{s.comparison.density comparison} we give the proof of Lemma~\ref{l.maximizer location prob relation}.

Denote $\cE_{\mrm{cond}}$ to be the event where $\max_{i=1, \ldots, k-1}|\pi^*(s_i)|<w_\beta, \cL^\beta[1] > (1-\beta^{-1}e^{-cL^{1/4}})L, \vec\cL^{\beta, \vec x} \approx \vec h$.
The following will be frequently used, to relate the conditioning on $\cE_{\mrm{cond}}$ and $\vec\cL^{\beta, \vec x} \approx \vec h$ for $\vec h\in E_{\mrm{val}}$ with $E_{\mrm{val}}$ as defined in \eqref{eq:eeeval}:
\begin{equation} \label{eq:condlowbd2}
\P\Bigg( \max_{i=1, \ldots, k-1}|\pi^*(s_i)|<w_\beta, \cL^\beta[1] > (1-\beta^{-1}e^{-cL^{1/4}})L \midd \vec\cL^{\beta, \vec x} \approx \vec h\Bigg)>c\exp(-CL^{1/2}).
\end{equation}
It follows from the lower bound in \Cref{p.k-point part to whole free energy} and $\vec h\in E_{\mrm{val}}$.

\subsection{Conditional overshoot at positive temperature}\label{s.conditional overshoot}

We work in the case of $\beta =1$ only in this subsection, in which case we recall that $w_\beta = L^{-1/2}$.

\begin{lemma}\label{l.L overshoot}
There exists a constant $\rho>0$ such that
\[
\P\Big(\cL^{\beta}[1] < L+\rho L^{-3} \midd \cE_{\mrm{cond}}\Big) < CL^{-1}.
\]
\end{lemma}

This lemma is proved by a resampling argument, and we explain the setup next.
For a function $f:\R\to\R$ and an interval $[a,b]$, we define the \emph{bridge of $f$ on $[a,b]$}, denoted $f^{[a,b]}$,~by
\begin{align*}
f^{[a,b]}(x) = f(x) - \frac{x-a}{b-a}f(b) - \frac{b-x}{b-a}f(a);
\end{align*}
in words, it is the function obtained by affinely shifting $f$ to equal $0$ at $a$ and $b$.

For any $x\in \R$, we define
\[
\cL^\beta_+[1](x)=\maxbeta_{|z_2|,\ldots, |z_{k-1}| \le 1} \sum_{i=2}^k\cL^\beta(z_{i-1}, s_{i-1};z_i, s_i),
\]
with $z_1=x$, $z_k=0$; note that this is a function of $\cL^\beta$ on the temporal strip $[s_1, 1]$.

Let $\h_{s_1,1}$ and $\h_{s_1,2}$ be the top two lines in the line ensemble associated with $\smash{\cL^\beta(0,0;\cdot, s_1)}$; in particular $\smash{\h_{s_1,1} = \cL^\beta(0,0;\cdot, s_1)}$. Let $\F$ be the $\sigma$-algebra generated by
\begin{itemize}
  \item $\h_{s_1,1}(z)$ for $z\in (-\infty, x_1L^{-1/4}] \cup[x_1L^{-1/4}+w_\beta, \infty)$,
  \item $(\h_{s_1,1})^{x_1L^{-1/4}+[0, \frac{1}{2}w_\beta]}$, $(\h_{s_1,1})^{x_1L^{-1/4}+ [\frac{1}{2}w_\beta, w_\beta]}$,
  \item $\h_{s_1,2}$ and $\cL^\beta_+[1]$.
\end{itemize}
In particular, conditional on $\F$, the remaining  randomness (to determine $\h_{s_1,1}$) is the value of $U:=\smash{\cL^\beta(0,0; x_1L^{-1/4}+\frac{1}{2}w_\beta, s_1)} = \smash{\h_{s_1,1}(x_1L^{-1/4}+\frac{1}{2}w_\beta)}$, by linear interpolation.
Therefore, for $z\in[x_1L^{-1/4}, x_1L^{-1/4}+\frac{1}{2}w_\beta]$ and $u\in\R$, we denote
\begin{align}\label{e.reconstruction formula}
\mf h^{\beta, u}_{s_1,1}(z) := \frac{z-x_1L^{-1/4}}{\frac{1}{2}w_\beta}\bigl(u-\h_{s_1,1}(x_1L^{-1/4})\bigr) + \h_{s_1,1}(x_1L^{-1/4}) + (\h_{s_1,1})^{x_1L^{-1/4}+[0,\frac{1}{2}w_\beta]}(z)
\end{align}
and a similar expression for $z\in\smash{[x_1L^{-1/4}+ \frac{1}{2}w_\beta, x_1L^{-1/4}+ w_\beta]}$. 
Further, given $\cF$ and $U$, we also determine the value of $\cL^\beta[1]$ via the formula for $\h_{s_1,1}$ as well as the convolution formula, since $\cL^\beta_+[1]$ is $\F$-measurable. We also observe that the dependence of $\smash{\cL^\beta[1]}$ on $U$ is increasing since the convolution formula has an increasing dependence on $\smash{\h_{s_1,1}}$, which itself depends on $U$ in an increasing manner.

Now we apply the Brownian Gibbs property. It implies that the distribution of $U$ is a normal random variable of $\F$-measurable mean $\mu$ and variance $\sigma^2$ given by
\begin{equation}\label{e.mu and sigma}
\begin{split}
\mu &= \tfrac{1}{2}\left(\h_{s_1,1}(x_1L^{-1/4}+ w_\beta) + \h_{s_1,1}(x_1L^{-1/4})\right)\\
\sigma^2 &= 2\cdot\frac{\frac{1}{2}w_\beta\cdot \frac{1}{2}w_\beta}{w_\beta} = \tfrac{1}{2}w_\beta,
\end{split}
\end{equation}
tilted by the Radon-Nikodym derivative $W^{\mrm{pt}}(U)/Z^{\mrm{pt}}$, where $W^{\mrm{pt}}$ and $Z^{\mrm{pt}}$ are given by
\[
W^{\mrm{pt}}(u) = W(\fh^{\beta, u}_{s_1,1}, \h_{s_1,2}), \quad
Z^{\mrm{pt}} = \E_\F\left[W^{\mrm{pt}}(U)\right].
\]
where $W(\fh^{\beta, u}_{s_1,1}, \h_{s_1,2})$ is from \eqref{e.rn derivative} for the interval $[x_1L^{-1/4}, x_1L^{-1/4}+w_\beta]$.

Now, if we require that $\cL^\beta[1] > (1-e^{-cL^{1/4}})L$, this is equivalent to $U$ being larger than some $\F$-determined value, due to the increasing dependence of $\cL^\beta[1]$ on $U$ already noted.
Further, if we also require that $|\pi^*(s_1)| \leq w_\beta$, it is not hard to see that this also is equivalent to a lower bound on $U$ (which may be $-\infty$).

These in turn imply that, conditional on $\cE_{\mrm{cond}}$, and $\vec\cL^{\beta, \vec x} \approx \vec h$, the distribution of $U$ is given by a Gaussian random variable of mean $\mu$ and variance $\sigma^2$ as given in \eqref{e.mu and sigma} tilted by $W^{\mrm{pt}}(U)/Z^{\mrm{pt}}$, and is further conditioned to be larger than an $\F$-measurable random variable, which we denote by $\mrm{Cor}$ (short for ``corner'', which is terminology introduced for an analogous object in \cite{hammond2016brownian}). Now, since $W^{\mrm{pt}}$ is an increasing function, this Gaussian random variable stochastically dominates the Gaussian random variable $X$ with the same mean and variance which is conditioned only to be larger than $\mrm{Cor}$.

If we know that $\mrm{Cor}$ is not too high (say less than order $L$) and $\mu$ is not too low (say $-\mu$ less than order $L$), then the Gaussian random variable will overshoot by an amount at least polynomial in $L^{-1}$. This upper bound on $\mrm{Cor}$ and lower bound on $\mu$ are recorded in the next two lemmas and will be proved shortly.

\begin{lemma}\label{l.corner upper bound}
There exists $K_0$ such that for $K > K_0$,
\begin{align*}
\P\left(\mrm{Cor} > h_1+K \midd \cE_{\mrm{cond}}\right) \leq \exp(-cKL^{1/2}).
\end{align*}

\end{lemma}

\begin{lemma}\label{l.mu lower bound}
There exists $K_0$ such that,
\begin{align*}
\P\left(\mu < -K_0L \midd \cE_{\mrm{cond}}\right) \leq \exp(-cL^{3/2}).
\end{align*}
\end{lemma}

However, knowing that the Gaussian overshoots by a polynomial-in-$L^{-1}$ amount does not immediately imply that the free energy overshoots $L$ by a comparable amount. For this we need to additionally know that the contribution to the free energy from the interval that gets perturbed by resampling $U$ is non-trivial (in fact, it contributes a positive proportion of the free energy, as we will show). This is recorded in the next lemma.

\begin{lemma}\label{l.perturbation has effect}
There exists $\rho>0$ such that, conditional on $\cE_{\mrm{cond}}$, with probability at least $1-\exp(-cL^{1/2})$, 
\begin{align*}
\int_{x_1L^{-1/4}+\frac{1}{4}w_\beta}^{x_1L^{-1/4}+\frac{3}{4}w_\beta}\exp\left(\h_{s_1,1}(x) + \cL^\beta_+[1](x)\right)\diff x\geq \rho \cL^\beta[1].
\end{align*}
\end{lemma}

With the previous three lemmas in hand, we now give the proof of Lemma~\ref{l.L overshoot}. Afterward, we will give the proofs of those lemmas.
\begin{proof}[Proof of Lemma~\ref{l.L overshoot}]
For every $u\in\R$, we define $\cL^{\beta, u}[1]$ the same way as $\cL^\beta[1]$ through \eqref{eq:clbetaone}, except for replacing $\cL^\beta(0,0;s_1,\cdot)=\fh^\beta_{s_1,1}$ by $\fh^{\beta,u}_{s_1,1}$.
We observe that, for any $\delta>0$, 
\begin{align*}
\exp\left(\cL^{\beta, U-\delta}[1]\right)
&= \int_{[-1,1]\setminus[x_1L^{-1/4}+\frac{1}{4}w_\beta, x_1L^{-1/4}+\frac{3}{4}w_\beta]} \exp\left(\mf h^{\beta, U-\delta}_1(x)+ \cL^\beta_+[1](x)\right)\diff x \\
&\qquad  +\int_{x_1L^{-1/4}+\frac{1}{4}w_\beta}^{x_1L^{-1/4}+\frac{3}{4}w_\beta} \exp\left(\mf h^{\beta, U-\delta}_1(x)+ \cL^\beta_+[1](x)\right)\diff x.
\end{align*}
Notice that $\mf h_{s_1,1}^{\beta, U-\delta}(x) \leq  \fh^\beta_{s_1,1}(x) - \frac{1}{2}\delta$ on $x_1L^{-1/4} + [\frac{1}{4}w_\beta, \frac{3}{4}w_\beta]$ (by the formula \eqref{e.reconstruction formula}). So the second term in the previous display is upper bounded by
\begin{align*}
\int_{x_1L^{-1/4}+\frac{1}{4}w_\beta}^{x_1L^{-1/4}+\frac{3}{4}w_\beta} \exp\left(\fh^\beta_{s_1,1}(x)+ \cL^\beta_+[1](x) -\tfrac{1}{2}\delta\right)\diff x.
\end{align*}
Note also that $\mf h_{s_1,1}^{\beta, U-\delta}(x) \leq \mf h_{s_1,1}^{\beta, U}(x) = \mf h_{s_1,1}^{\beta}(x)$, which gives an upper bound on the first term two displays above.
Then by \Cref{l.perturbation has effect}, there exists $\rho>0$ such that, with conditional probability at least $1-\exp(-cL^{1/2})$, $\exp(\cL^{\beta, U-\delta}[1])$ is upper bounded by
\begin{align*}
\left[1-\rho + e^{-\frac{1}{2}\delta}\rho\right]\int_{[-1,1]} \exp\left(\fh^\beta_{s_1,1}(x)+ \cL^\beta_+[1](x)\right)\diff x = \left[1-\rho + e^{-\frac{1}{2}\delta}\rho\right]\exp(\cL^\beta[1]).
\end{align*}
Since $e^{-x}\leq 1-x/2$ for $x\in[0,1]$, we see that the square bracket factor is at most $1-\rho\delta/4$.
In summary, with conditional probability at least $1-\exp(-cL^{1/2})$, for any $\delta>0$ sufficiently small,
\begin{align}\label{e.pertubation inequality}
\exp\left(\cL^{\beta, U-\delta}[1]\right) \leq (1-\rho\delta/4)\exp\left(\cL^\beta[1]\right).
\end{align}

Recall from the discussion preceding Lemma~\ref{l.corner upper bound} that, conditionally on $\F$, $U$ stochastically dominates a random variable $X$ which is Gaussian with mean $\mu$ and variance $\sigma^2$ as given in \eqref{e.mu and sigma} and conditioned to stay above $\mrm{Cor}$ (i.e., $X$ is not tilted by $W^{\mrm{pt}}(U)/Z^{\mrm{pt}}$). Now, since (by \Cref{l.corner upper bound} and \Cref{l.mu lower bound}) with conditional probability at least $1-\exp(-cL^{1/2})$ it holds that $\mrm{Cor} \leq h_1 + K_0$ (so that in particular $\mrm{Cor} = O(L)$) and $\mu\geq -K_0L$, it follows that, with conditional probability at least $1-L^{-1}$ that
\begin{align*}
X \geq \mrm{Cor} + L^{-3},
\end{align*}
using standard estimates for the normal tail bound (\Cref{l.normal bounds}).

In particular, since $U$ stochastically dominates $X$, taking $\delta = L^{-3}$ implies that $U-\delta \geq \mrm{Cor}$. Since $\mrm{Cor}$ is such that for any $u\geq \mrm{Cor}$ it holds that $\cL^{\beta, u}[1]\geq (1-e^{-cL^{1/4}})L$, we obtain from combining these observations with \eqref{e.pertubation inequality} that
\begin{align*}
\exp(\cL^\beta[1])  \geq (1+\tfrac{1}{5}\rho L^{-3})\exp(L).
\end{align*}
Taking logarithms and relabeling $\rho$ completes the proof. 
\end{proof}

We next prove the three lemmas that have been assumed.
\begin{proof}[Proof of Lemma~\ref{l.corner upper bound}]
We observe that by the definition of $\mrm{Cor}$, it holds almost surely that $\mrm{Cor} \leq \h_{s_1,1}(x_1L^{-1/4}+\frac{1}{2}w_\beta)$. So we obtain
\begin{align*}
\P\left(\mrm{Cor} > h_1+K \midd  \vec\cL^{\beta, \vec x} \approx \vec h\right)\leq \P\left(\h_{s_1,1}(x_1L^{-1/4}+ \tfrac{1}{2}w_\beta) > h_1+K \midd  \vec\cL^{\beta, \vec x} \approx \vec h\right).
\end{align*}
By \Cref{lem:fh-tent} this is upper bounded by $\exp(-cKL^{1/2})$. 
Combining this with \eqref{eq:condlowbd2} and taking $K$ large enough, we obtain the claimed bound.
\end{proof}

\begin{proof}[Proof of Lemma~\ref{l.mu lower bound}]
 Since $\mu = \frac{1}{2}(\h_{s_1,1}(x_1L^{-1/4}+ w_\beta) + \h_{s_1,1}(x_1L^{-1/4}))$, it holds by \Cref{lem:fh-tent} and the independence of $\cL^{\beta, \vec x}_i$ across $i$ that $\P\left(\mu < -K_0L \mid \vec\cL^{\beta, \vec x} \approx \vec h\right)<\exp(-cL^{3/2})$. Combining this with \eqref{eq:condlowbd2} completes the proof.
\end{proof}

\begin{proof}[Proof of Lemma~\ref{l.perturbation has effect}]
Take $C_*$ to be a large constant.
By using \Cref{p.k-point part to whole free energy} for a mesh of $L^2$ many $x\in [-1, 1]$, taking a union bound, and applying the unconditional local fluctuations estimates \Cref{lem:fh-cont}, we have
\begin{equation}  \label{eq:clbetalowb}
 \P\Bigg( \max_{|x|\le 1}\cL^\beta(x,s_1;0,1) > \sum_{i=2}^k h_i -(k-2)\log (2L^{1/2}) + \frac{1}{2}C_* \midd \vec\cL^{\beta, \vec x} \approx \vec h\Bigg)<\exp(-cC_*L^{1/2}).       
\end{equation}
Further by \Cref{cor.tent-pr},
\begin{multline*}
\P\left( \h_{s_1,1}(x)\le h_1 - C_* - L^{1/2}|x-x_1L^{-1/4}|/2, \;\forall x: |x-x_1L^{-1/4}|\geq 10C_*w_\beta, |x|\leq 1  \midd \vec\cL^{\beta, \vec x} \approx \vec h \right)\\
>1-C\exp(-cC_*L^{1/2}).
\end{multline*}
We can replace the conditioning $\vec\cL^{\beta, \vec x} \approx \vec h$ in the above two estimates by $\cE_{\mrm{cond}}$, using \eqref{eq:condlowbd2}.
Then from these, we have that with probability $>1-C\exp(-cL^{1/2})$ conditioning on $\cE_{\mrm{cond}}$, 
\begin{align*}
\MoveEqLeft[25]
\int_{|x-x_1L^{-1/4}|\geq 10C_*w_\beta, |x|\leq 1}\exp\left(\h_{s_1,1}(x) + \cL^\beta(x,s_1;0,1)\right)\diff x \\
\leq \exp\left(\sum_{i=1}^k h_i -(k-2)\log(2L^{1/2}) - \log(L^{1/2}) - \tfrac{1}{2}C_*\right)
&\leq \exp(L - \tfrac{1}{4}C_*)\leq \exp(- \tfrac{1}{4}C_*) \cL^\beta[1],
\end{align*}
using that $\vec h\in E_{\mrm{val}}$ and $C_*$ is large for the second inequality, and that we have conditioned on $\cL^\beta[1]>L$ for the last inequality. This implies that
\begin{align}\label{e.non-trivial contribution symmetric about zero}
\cL^\beta[1]\leq (1-\exp(- \tfrac{1}{4}C_*))^{-1} \int_{|x-x_1L^{-1/4}|\leq 10C_*w_\beta}\exp\left(\h_{s_1,1}(x) + \cL^\beta_+[1](x)\right)\diff x.
\end{align}

It also follows from \Cref{cor.tent-pr} that, with probability at least $1-C\exp(-cC_*L^{1/2})$ conditioning on $\vec\cL^{\beta, \vec x} \approx \vec h$,
\[
h_1-20C_*\leq \inf_{|x-x_1L^{-1/4}|\leq 10C_*w_\beta} \h_{s_1,1}(x) \leq \sup_{|x-x_1L^{-1/4}|\leq 10C_*w_\beta} \h_{s_1,1}(x) \leq h_1 + C_*.
\]
And by \Cref{cor.tent-pr} and \Cref{prop:dp-tent-coal}, with probability at least $1-C\exp(-cC_*L^{1/2})$ conditioning on $\vec\cL^{\beta, \vec x} \approx \vec h$,
there is
\[
\cL^\beta(x,s_{i-1};y,s_i) > h_i - 30C_*,
\]
for any $i=2,\ldots, k$, and $|x-x_{i-1}L^{-1/4}|, |y-x_iL^{-1/4}| \le 10 C_*w_\beta$; thus
\[
\inf_{|x-x_1L^{-1/4}|\leq 10C_*w_\beta} \cL^\beta_+[1](x) \geq \sum_{i=2}^k h_i -(k-2)\log (2L^{1/2}) - 30kC_*.
\]
Combining these two estimates with \eqref{eq:clbetalowb}, we get that, with probability at least $1-C\exp(-cC_*L^{1/2})$ conditioning on $\vec\cL^{\beta, \vec x} \approx \vec h$,
\begin{align*}
\frac{\int_{x_1L^{-1/4}+\frac{1}{4}w_\beta}^{x_1L^{-1/4}+\frac{3}{4}w_\beta}\exp\left(\h_{s_1,1}(x) + \cL^\beta(x,s_1;0,1)\right)\diff x}{\int_{x_1L^{-1/4}-10C_*w_\beta}^{x_1L^{-1/4}+10C_*w_\beta}\exp\left(\h_{s_1,1}(x) + \cL^\beta(x,s_1;0,1)\right)\diff x} \geq (40C_*)^{-1} \exp(-40kC_*).
\end{align*}
We can further replace the conditioning by $\cE_{\mrm{cond}}$, using \eqref{eq:condlowbd2}.
Then together with \eqref{e.non-trivial contribution symmetric about zero} the proof completes.
\end{proof}

\subsection{Free energy of restricted interval}\label{s.comparison.restricted free energy}

In this subsection, we prove the following statement. 

We introduce some notation to describe the free energy profiles in disjoint temporal strips defined by $[s_{i-1}, s_i]$. 
We define
\begin{equation}\label{e.h top bot definition}
\mf h^{\mrm{top}, i, \vec x}_1 = \cL^\beta(x_{i-1}L^{-1/4}, s_{i-1}; x_iL^{-1/4} + \cdot, s_i), \quad \mf h^{\mrm{bot}, i, \vec x}_1 = \cL^\beta(x_{i-1}L^{-1/4} + \cdot, s_{i-1}; x_iL^{-1/4}, s_i),
\end{equation}
for each $i=1,\ldots, k$.
For any vector $\vec z \in \R^{k-1}$, we always write $z_0=z_k=0$ for the convenience of notations.
We denote
\begin{equation}\label{e.hsumdef}
\fh^{\mrm{sum}}(\vec z) = \sum_{i=1}^{k-1} 
\mf h^{\mrm{top}, i, \vec x}_1(z_i) + \sum_{i=2}^k \mf h^{\mrm{bot}, i, \vec x}_1(z_{i-1}),
\end{equation}
We also write $I_{M,L}=[-ML^{-1/2}\log L, ML^{-1/2}\log L]$,
and $I_{\vec x}=L^{-1/4}\vec x+I_{M,L}^{k-1}\subset \R^{k-1}$.
Recall the $\maxbeta$ notation from \eqref{e.maxbeta definition}.

\begin{lemma}\label{l.conditional localization of event}
When $\beta=1$, there exists $\rho > 0$ such that, conditional on $\cE_{\mrm{cond}}$, it holds with probability at least $1-CL^{-1}$ that 
\begin{align*}
\maxbeta_{\vec z\in I^{k-1}_{M,L}} \fh^{\mrm{sum}}(\vec z)\geq L + \rho L^{-3} + \sum_{i=2}^{k-1} \cL^{\beta, \vec x}_i.
\end{align*}
When $\beta=\infty$, conditional on $\cE_{\mrm{cond}}$, it holds with probability at least $1-\exp(-cL^{1/2})$ that
\begin{align*}
\max_{\vec z\in I^{k-1}_{M,L}} \fh^{\mrm{sum}}(\vec z) \geq L + \sum_{i=2}^{k-1} \cL^{\beta, \vec x}_i. 
\end{align*}

\end{lemma}

To prove this in the $\beta=1$ case, the main thing we need is that the contribution to the free energy $\cL^\beta[1]$ from outside $I_{\vec x}$ is small, which we isolate in the following statement.
\begin{lemma}\label{l.outside center free energy contribution}
When $\beta=1$, conditional on $\cE_{\mrm{cond}}$, with probability at least $1-\exp(-cML^{1/2})$,
\begin{align*}
\maxbeta_{\vec z \in [-1,1]^{k-1}\setminus I_{\vec x}} \sum_{i=1}^k \cL^\beta(z_{i-1}, s_{i-1}; z_i, s_i) \leq L-\frac{1}{2}M\log L.
\end{align*}
\end{lemma}

To handle the $\beta=\infty$ case, we will need the following statement, which is that the maximizer restricted to $I_{\vec x}$ is with high probability the same as the global maximize.
We note that it also holds when $\beta=1$.

We define the restricted (in terms of the interval over which the maximization is performed) version of $\pi^*(s_i)$:  for $i=1, \ldots, k-1$ we let
\begin{align}\label{e.pi^*,res definition}
\pi^{*,\mrm{res}, \vec x}(s_i) = \argmax_{|z|\leq ML^{-1/2}\log L} \left(\cL^\beta(0, s_{i-1}; x_iL^{-1/4}+z, s_i) + \cL^\beta(x_iL^{-1/4}+z, s_i; 0, s_{i+1})\right),
\end{align}
for $M$ being a large number.
\begin{lemma}\label{l.restricted maximizer is maximizer}
With probability $1-\exp(-cML^{1/2}\log L)$ conditional on $\vec\cL^{\beta, \vec x} \approx \vec h$, for each $i=1, \ldots, k$,
\begin{align*}
\pi^*(s_i) = x_iL^{-1/4}+\pi^{*,\mrm{res},\vec x}(s_i).
\end{align*}

\end{lemma}

We give the proof of Lemma~\ref{l.conditional localization of event} before turning to proving Lemmas~\ref{l.outside center free energy contribution} and \ref{l.restricted maximizer is maximizer}.

\begin{proof}[Proof of Lemma~\ref{l.conditional localization of event}]
In this proof, by `conditional probability', we refer to the conditioning $\cE_{\mrm{cond}}$.

\noindent\emph{Case of $\beta=\infty$}: 
We let $\cE_{\mrm{eq}}$ denote the event where $\pi^*(s_i) = x_iL^{-1/4} + \pi^{*,\mrm{res}, \vec x}(s_i)$ for each $i=1, \ldots, k-1$.
We first note that by \Cref{l.restricted maximizer is maximizer} and \eqref{eq:condlowbd2}, for all large enough $M$, $\cE_{\mrm{eq}}$ holds with conditional probability  $>1-\exp(-cML^{1/2})$.
Now on $\cE_{\mrm{eq}}$ it follows by definition that $\cL^{\beta}[1] = \max_{\vec z \in I_{\vec x}} \sum_{i=1}^k \cL^\beta(z_{i-1}, s_{i-1}; z_i, s_i)$. Further by coalescence, which holds with conditional probability $>1-C\exp(-cL^{3/2})$ by \Cref{prop:dp-tent-coal} and \eqref{eq:condlowbd2}, it follows that 
\[
\max_{\vec z\in I^k_{M,L}} \fh^{\mrm{sum}}(\vec z)
=
\max_{\vec z \in I_{\vec x}} \sum_{i=1}^k \cL^\beta(z_{i-1}, s_{i-1}; z_i, s_i)
+\sum_{i=2}^{k-1}\cL^{\beta, \vec x}_i
\geq L + \sum_{i=2}^{k-1} \cL^{\beta, \vec x}_i,
\]
completing the proof.

\smallskip

\noindent\emph{Case of $\beta=1$}: By Lemma~\ref{l.L overshoot} there exists $\rho > 0$ such that, with conditional probability at least $1-CL^{-1}$, we have $\cL^\beta[1] \geq L + \rho L^{-3}$.
By Lemma~\ref{l.outside center free energy contribution}, with conditional probability at least $1-\exp(-cML^{1/2})$,
\begin{align*}
\exp(\cL^{\beta}[1]) - \exp\left(\maxbeta_{\vec z \in I_{\vec x}} \sum_{i=1}^k \cL^\beta(z_{i-1}, s_{i-1}; z_i, s_i)\right) \leq L^{-\frac{1}{2}M}\exp(L).
\end{align*}
The previous two inequalities, along with $\exp(x)\geq 1+x$, imply that
\begin{align*}
\exp\left(\maxbeta_{\vec z \in I_{\vec x}} \sum_{i=1}^k \cL^\beta(z_{i-1}, s_{i-1}; z_i, s_i)\right)
&\geq \exp(L+\rho L^{-3}) - L^{-\frac{1}{2}M}\exp(L)\\
&\geq \exp(L)\left[1+\rho L^{-3} - L^{-\frac{1}{2}M}\right]\\
&\geq \exp(L)\left[1+\tfrac{1}{2}\rho L^{-3}\right].
\end{align*}
By coalescence (\Cref{prop:dp-tent-coal}), with conditional probability at least $1-C\exp(-cL^{3/2})$,
\begin{multline*}
\maxbeta_{\vec z\in I^k_{M,L}} \fh^{\mrm{sum}}(\vec z)
\ge
\maxbeta_{\vec z \in I_{\vec x}} \sum_{i=1}^k \cL^\beta(z_{i-1}, s_{i-1}; z_i, s_i)
+\sum_{i=2}^{k-1}\cL^{\beta, \vec x}_i - kCe^{-cL}\\
> L + \tfrac{1}{4}\rho L^{-3}- kCe^{-cL} + \sum_{i=2}^{k-1}\cL^{\beta, \vec x}_i,
\end{multline*}
Since $kCe^{-cL}\ll \rho L^{-1}$, the proof is complete by relabeling $\rho$.
\end{proof}

The main step in proving Lemma~\ref{l.outside center free energy contribution} is to handle the case of $k=2$, which is given in the following statement; then we can obtain the statement for general $k$ by making use of coalescence (Proposition~\ref{prop:dp-tent-coal}).

\begin{lemma}\label{l.free energy in strip contribution from outside}
Conditional on $\cE_{\mrm{cond}}$, it holds with probability at least $1-\exp(-cML^{1/2})$ that, for each $i=1,\ldots, k-1$,
\begin{align*}
\int_{(-x_iL^{-1/4}+[-1,1])\setminus I_{M,L}} \exp\bigl(\mf h^{\mrm{top}, i, \vec x}_1(z)+\mf h^{\mrm{bot}, i+1, \vec x}_1(z)\bigr)\diff z\leq L^{-M}\exp\left(h_i+h_{i+1}\right).
\end{align*}
and
\begin{align*}
\MoveEqLeft[6]
\int_{-x_iL^{-1/4}+[-1,1]} \exp\bigl(\mf h^{\mrm{top}, i, \vec x}_1(z)+\mf h^{\mrm{bot}, i+1, \vec x}_1(z)\bigr)\diff z\leq \exp\left(h_i+h_{i+1} - \log(2L^{1/2})+M\right).
\end{align*}
\end{lemma}

\begin{proof}
The second inequality follows from \eqref{eq:ptfree1} in \Cref{l.part free energy to whole free energy} and \eqref{eq:condlowbd2}.

For the first inequality, by \eqref{eq:condlowbd2}, it suffices to show that the displayed inequality holds with probability at least $1-\exp(-cML^{1/2})$ given only that $\vec\cL^{\beta, \vec x} \approx \vec h$.
Indeed, conditional on $\vec\cL^{\beta, \vec x} \approx \vec h$, from \Cref{cor.tent-pr}, it holds with probability at least $1-C\exp(-cML^{1/2}\log L)$ that, for all $z\not\in I_{M,L}$, $z+x_iL^{-1/4} \in [-1, 1]$,
\begin{align*}
\mf h^{\mrm{top}, i, \vec x}_1(z) &\leq h_i -\tfrac{1}{2}M\log L - \tfrac{1}{4}L^{1/2}|z|,\\
\mf h^{\mrm{bot}, i+1, \vec x}_1(z) &\leq h_{i+1} - \tfrac{1}{2}M\log L - \tfrac{1}{4}L^{1/2}|z|.
\end{align*}
By plugging these into the integral, the conclusion follows.
\end{proof}

\begin{proof}[Proof of Lemma~\ref{l.outside center free energy contribution}]
By coalescence \Cref{prop:dp-tent-coal} and \eqref{eq:condlowbd2}, 
with probability at least $1-C\exp(-cL^{3/2})$ conditional on $\cE_{\mrm{cond}}$, the LHS is at most
\[
\maxbeta_{\vec z+L^{-1/4}\vec x \in [-1,1]^{k-1}\setminus I_{\vec x}} \sum_{i=1}^{k-1}\mf h^{\mrm{top}, i, \vec x}_1(z_i)+\mf h^{\mrm{bot}, i+1, \vec x}_1(z_i)+Ce^{-cL} - \sum_{i=2}^{k-1}h_i.
\]
Recall that the first term is the logarithm of an integral over $[-1,1]^{k-1}\setminus I_{\vec x}$. We can upper bound the integral by replacing $[-1,1]^{k-1}\setminus I_{\vec x}$ with the $k-1$ dimension product where all but the $i$\textsuperscript{th} one is $[-1,1]$ and the $i$\textsuperscript{th} one is $[-1,1]\setminus (x_iL^{-1/4}+I_{M,L})$, and summing over $i$ (i.e., the sum is inside the logarithm).
Using \Cref{l.free energy in strip contribution from outside}, we get an upper bound of 
\begin{align*}
\sum_{i=1}^k h_i + Ce^{-cL} - M\log L - (k-2)\log(2L^{1/2}) +M + C\log k.
\end{align*}
Using that $\vec h\in E_{\mrm{val}}$, by taking $M$ large enough the conclusion follows.
\end{proof}

\begin{proof}[Proof of Lemma~\ref{l.restricted maximizer is maximizer}]
We need to show that with high probability, conditional on $\vec\cL^{\beta, \vec x} \approx \vec h$, it holds that $\max_{i=1, \ldots, k-1}|\pi^*(s_i)-x_iL^{-1/4}|\leq ML^{-1/2}\log L$ so that $\pi^*(s_i) = x_iL^{-1/4}+\pi^{*,\mrm{res},\vec x}(s_i)$. 

Let $J_i$ denote the interval $\{z:|z+x_iL^{-1/4}|\le 10^{-6}(s_i-s_{i-1})^{1/2}\wedge (s_{i+1}-s_i)^{1/2}\}$.
By \Cref{prop:dp-tent-coal}, conditional on $\vec\cL^{\beta, \vec x} \approx \vec h$, with probability $>1-C\exp(-cL^{3/2})$ the event $\pi^*(s_i) \neq x_iL^{-1/4}+\pi^{*,\mrm{res},\vec x}(s_i)$
implies that 
\begin{equation}  \label{eq:hblowbd}
\sup_{z\in J_i\setminus I_{M,L}} \mf h_1^{\mrm{top}, i, \vec x}(z) + \mf h_1^{\mrm{bot}, i+1, \vec x}(z) > h_i + h_{i+1} - C\exp(-cL),    
\end{equation}
or 
\begin{multline}  \label{eq:hblowbd2}
\sup_{z\not\in J_i} \cL^\beta(0,s_{i-1};x_iL^{-1/4}+z,s_i) + \cL^\beta(x_iL^{-1/4}+z,s_i; 0,s_{i+1}) \\ \ge \cL^\beta(0,s_{i-1};x_iL^{-1/4},s_i) + \cL^\beta(x_iL^{-1/4},s_i; 0,s_{i+1}).
\end{multline}
We have that \eqref{eq:hblowbd} further implies that either
\[
\sup_{z\in J_i\setminus I_{M,L}} \mf h_1^{\mrm{top}, i, \vec x}(z) > h_{i+1} - C\exp(-cL),   
\]
or
\[
\sup_{z\in J_i\setminus I_{M,L}} \mf h_1^{\mrm{bot}, i+1, \vec x}(z) > h_i - C\exp(-cL),   
\]
happens; and by \Cref{cor.tent-pr}, each happens, conditional on $\vec\cL^{\beta, \vec x} \approx \vec h$, with probability $<\exp(-cML^{1/2}\log L)$.

As for \eqref{eq:hblowbd2}, by  \Cref{cor.tent-pr}, with probability $>1-C\exp(-cL^{3/4})$ conditional on $\vec\cL^{\beta, \vec x} \approx \vec h$, we have
that the RHS is $>h_i+h_{i+1}-L^{1/4}\log L$.
For the LHS, for any fixed $z\not\in J_i$ we have
\begin{multline*}
\P\left(\cL^\beta(0,s_{i-1};x_iL^{-1/4}+z,s_i) + \cL^\beta(x_iL^{-1/4}+z,s_i; 0,s_{i+1}) > h_i+h_{i+1}-L^{1/4}\log L\right) 
\\< C\exp\left(-\tfrac{4}{3}(s_{i+1}-s_{i-1})L^{3/2} - cz^2L^{1/2}\right)
\end{multline*}
by \Cref{lem:fh-ut}, and that $\vec h \in E_{\mrm{val}}$.
Then by a union bound over all $z\not\in J_i$, $z\in L^{-2}L$, and using the continuity estimate \Cref{lem:fh-cont}, we get that the LHS is $\le h_i+h_{i+1}-L^{1/4}\log L$ with probability $>1-C\exp\left(-\tfrac{4}{3}(s_{i+1}-s_{i-1})L^{3/2} - cL^{3/2}\right)$.
Then since 
$$\P\left(\cL^{\beta, \vec x}_i \in h_i + [0,e^{-L}]\right)>c\exp\left(-\tfrac{4}{3}(s_{i}-s_{i-1})L^{3/2} - CL^{3/4}\right)$$
and
$$\P\left(\cL^{\beta, \vec x}_{i+1} \in h_{i+1} + [0,e^{-L}]\right)>c\exp\left(-\tfrac{4}{3}(s_{i+1}-s_{i})L^{3/2} - CL^{3/4}\right),$$
by \Cref{lem:fh-ut}, we conclude that \eqref{eq:hblowbd2} happens with probability $<C\exp(-cL^{3/4})$ conditional on $\vec\cL^{\beta, \vec x} \approx \vec h$.
This completes the proof.
\end{proof}

\subsection{The density comparison}\label{s.comparison.density comparison}
We now finish the proof of \Cref{l.maximizer location prob relation}.

Recall the setup at the beginning of this section. Also recall the definition of $\pi^{*,\mrm{res}, \vec x}(s_i)$ from \eqref{e.pi^*,res definition} as a version of $\pi^*(s_i)$ defined for $i=1, \ldots, k-1$ but with a restricted interval over which the maximization is performed, and the definition of $\mf h^{\mrm{top}, i, \vec x}_1$ and $\mf h^{\mrm{bot}, i, \vec x}_1$
from \eqref{e.h top bot definition}, and $\fh^{\mrm{sum}}$ from \eqref{e.hsumdef}.

As before, we let $I_{M,L} = [-ML^{-1/2}\log L, ML^{-1/2}\log L]$, and $w_{\beta=1}=L^{-1/2}$, $w_{\beta=\infty}=L^{-1}$.

\begin{proof}[Proof of \Cref{l.maximizer location prob relation}] 
We want to estimate the conditional probabilities on both the LHS and the RHS and show that they are within $1+o(1)$ of each other. 
Towards this, for either $\iota=0$ or $1$ (we allow both as the RHS in Lemma~\ref{l.maximizer location prob relation} does not have the $\smash{\beta^{-1}e^{-cL^{1/4}}}$ term), we claim that
\begin{multline}  \label{eq:dencompcen}
   \P\Bigl(\max_{i=1, \ldots, k-1}|\pi^*(s_i)-x_iL^{-1/4}|<w_\beta, \cL^\beta[1] > (1-\iota\beta^{-1}e^{-cL^{1/4}})L \midd \vec\cL^{\beta, \vec x} \approx \vec h\Bigr)= (1+O(L^{-1}))\\
   \times
   \P\left(\max_{i=1, \ldots, k-1}|\pi^*(s_i)-x_iL^{-1/4}|<w_\beta, \maxbeta_{\vec z\in I_{M,L}^{k-1}} \fh^{\mrm{sum}}(\vec z) \geq L+ \beta^{-1}\rho L^{-3}+\sum_{i=2}^{k-1} \cL^{\beta,\vec x}_i \midd \vec\cL^{\beta, \vec x} \approx \vec h\right).
\end{multline}
Indeed, by Lemma~\ref{l.conditional localization of event}, the LHS is upper bounded by the RHS; while by \Cref{prop:dp-tent-coal},
\[
\cL^\beta[1] + \sum_{i=2}^{k-1} \cL^{\beta,\vec x}_i \ge \maxbeta_{\vec z\in I_{M,L}^{k-1}} \fh^{\mrm{sum}}(\vec z),
\]
with probability $>1-C\exp(-cL^{3/2})$, conditional on $\cL^{\beta, \vec x} \approx \vec h$.
Then since the LHS is lower bounded by $c\exp(-CL^{1/2})$ ($\beta=1$) or $c$ ($\beta=\infty$) by \Cref{p.k-point part to whole free energy} and $\vec h \in E_{\mrm{val}}$, the LHS is lower bounded by the RHS.

We now let $\cE_\pm$ denote the events where for each $i=1,\ldots, k-1$,
\[
\sup_{|z|\le w_\beta}\mf h^{\mrm{top}, i, \vec x}_1(z) + \mf h^{\mrm{bot}, i+1, \vec x}_1(z) > \sup_{|z|\ge w_\beta, z\in I_{M,L}}\mf h^{\mrm{top}, i, \vec x}_1(z) + \mf h^{\mrm{bot}, i+1, \vec x}_1(z) \mp C\exp(-cL),
\]
and
\[
\maxbeta_{\vec z\in I_{M,L}^{k-1}} \fh^{\mrm{sum}}(\vec z) \geq L+ \beta^{-1}\rho L^{-3}+\sum_{i=2}^{k-1} h_i \mp C\exp(-cL).
\]
By \Cref{l.restricted maximizer is maximizer} and \Cref{p.k-point part to whole free energy}, 
the probability in the RHS of \eqref{eq:dencompcen} is in
\begin{equation}  \label{eq:deninter}
  \left[\P(\cE_- \mid \vec\cL^{\beta, \vec x} \approx \vec h)-\exp(-cML^{1/2}),\;
\P(\cE_+ \mid \vec\cL^{\beta, \vec x} \approx \vec h)+\exp(-cML^{1/2})
\right].  
\end{equation}
Now we wish to make use of \Cref{prop:dp-tent-bcomp} to replace the processes $\mf h^{\mrm{top}, i, \vec x}_1$ and $\mf h^{\mrm{bot}, i, \vec x}_1$ by a collection of processes given by independent Brownian bridges, up to an exponentially small $L^\infty$ error and on an event with probability $1-\exp(-cL)$. 

Denote
\[
G=L+ \beta^{-1}\rho L^{-3}-\sum_{i=1}^k h_i
\]
Since $\vec h\in E_{\mrm{val}}$ from \eqref{eq:eeeval}, we have $G<-(k-1)\beta^{-1}\log(2L^{1/2})+Cr_\beta$ (recall that $r_{\beta=1}=1$ and $r_{\beta=\infty}=L^{-1/2}$).

\noindent\textbf{Reduce to Brownian bridges}. 
We denote $\theta=10^{-7}\min_{i=1,\ldots, k}(s_i-s_{i-1})L^{1/2}$.
We let $\Theta$ denote the collection of all 
\[
\vec b=(\vec b^{\mrm{bot},\mrm{L}}, \vec b^{\mrm{bot},\mrm{R}}, \vec b^{\mrm{top},\mrm{L}}, \vec b^{\mrm{top},\mrm{R}}) = \bigl((b^{\mrm{bot},\mrm{L}}_i)_{i=2}^k, (b^{\mrm{bot},\mrm{R}}_i)_{i=2}^k,  (b^{\mrm{top},\mrm{L}}_i)_{i=1}^{k-1}, (b^{\mrm{top},\mrm{R}}_i)_{i=1}^{k-1}\bigr) \in \R^{4(k-1)}.
\]
For any $\vec b\in \Theta$, we denote by $B^{\dagger, i,\vec b}$ a rate $2$ Brownian bridge from $(-\theta, -2\theta L^{1/2}+b^{\mrm{\dagger},\mrm{L}}_i)$, to $(0,0)$, then to $(\theta, -2\theta L^{1/2}+b^{\mrm{\dagger},\mrm{R}}_i)$, 
for $\dagger=\mrm{top}$, $i=1, \ldots, k-1$, and $\dagger=\mrm{bot}$, $i=2, \ldots, k$.
All these $B^{\dagger, i}$ are independent of each other. One should think of $B^{\dagger, i}$ as being approximately $\mf h_1^{\dagger, i, \vec x} - h_i$, which we will use in a more precise sense via Proposition~\ref{prop:dp-tent-bcomp}.

We then denote by $\cE_{\vec b}$ the event where for each
$i=1,\ldots, k-1$,
\[
\sup_{|z|\le w_\beta}B^{\mrm{top}, i, \vec b}(z) + B^{\mrm{bot}, i+1, \vec b}(z) > \sup_{|z|\ge w_\beta, z\in I_{M,L}}B^{\mrm{top}, i, \vec b}(z) + B^{\mrm{bot}, i+1, \vec b}(z),
\]
and
\[
    \maxbeta_{\vec z\in I_{M,L}^{k-1}} \sum_{i=1}^{k-1}B^{\mrm{top}, i, \vec b}(z_i) + \sum_{i=2}^{k}B^{\mrm{bot}, i, \vec b}(z_{i-1})\geq G .
\]
We also denote by $\cE_{\vec b,\pm}$ the event where for each
$i=1,\ldots, k-1$,
\[
\sup_{|z|\le w_\beta}B^{\mrm{top}, i, \vec b}(z) + B^{\mrm{bot}, i+1, \vec b}(z) > \sup_{|z|\ge w_\beta, z\in I_{M,L}}B^{\mrm{top}, i, \vec b}(z) + B^{\mrm{bot}, i+1, \vec b}(z) \mp C\exp(-cL),
\]
and
\[
    \maxbeta_{\vec z\in I_{M,L}^{k-1}} \sum_{i=1}^{k-1}B^{\mrm{top}, i, \vec b}(z_i) + \sum_{i=2}^{k}B^{\mrm{bot}, i, \vec b}(z_{i-1})\geq G \mp C\exp(-cL).
\]
It is straightforward to check that 
\begin{equation}  \label{eq:bbpmcp}
|\P(\cE_{\vec b})-\P(\cE_{\vec b, \pm})|<C\exp(-cL).    
\end{equation}

For each $a>0$ we let $\Theta_a$ denote the collection of all $\vec b\in \Theta$, where each coordinate is in $[-a, a]$.
We write $D=\log L$.

Since $\vec x\in \cK$ which is a compact set, applying \Cref{prop:dp-tent-bcomp}, and using \Cref{l.control near tent center} to bound each $\fh^{\mrm{top},i}_1$ and $\fh^{\mrm{bot},i}_1$ at $\pm \theta$, we have that
\begin{equation}  \label{eq:cemlbd}
\P(\cE_-\mid \vec\cL^{\beta, \vec x} \approx \vec h) 
\ge (1-C\exp(-cD^2)) \min_{\vec b\in \Theta_{DL^{1/4}}} \P(\cE_{\vec b,-})  - C\exp(-cL),
\end{equation}
and
\begin{multline}  \label{eq:cemlbd2}
 \P(\cE_+\mid \vec\cL^{\beta, \vec x} \approx \vec h) 
\le 
\max_{\vec b\in \Theta_{DL^{1/4}}}  \P(\cE_{\vec b,+})+
\sum_{i=2}^{\lfloor L^{1/4} \rfloor }C\exp(-ci^2D^2) \max_{\vec b\in \Theta_{iDL^{1/4}}}  \P(\cE_{\vec b,+}) \\  + C\exp(-cL) + C\exp(-cD^2L^{1/2}).   
\end{multline}
By \eqref{eq:bbpmcp}, we can replace each $\P(\cE_{\vec b, \pm})$ by $\P(\cE_{\vec b})$.
We also note that these bounds are independent of $\vec x$.

We next state a comparison lemma for different $\vec b$.
\begin{lemma}  \label{lem:compb}
For any large enough $K>0$, and $\vec b, \vec g \in \Theta$ such that $\|\vec b- \vec g\|_\infty < KL^{1/4}$, we have
\[
\P(\cE_{\vec b}) = (1+O(K^2L^{-1/4}\log L)) \P(\cE_{\vec g}) + O(\exp(-cK^2M^{-1}L^{1/2}\log L)).
\]
\end{lemma}
Now for each $i=1,\ldots, \lfloor L^{1/4}\rfloor$, and any $\vec b\in \Theta_{iDL^{1/4}}$, $\vec g\in \Theta_{DL^{1/4}}$, we can find a sequence of vectors in $\Theta$, $\vec b=\vec b^{(0)}, \vec b^{(1)}, \ldots, \vec b^{(i)} = \vec g$, with $\|\vec b^{(j-1)}-\vec b^{(j)}\|_\infty \le DL^{1/4}$ for each $j=1,\ldots, i$.
Then \eqref{lem:compb} implies that for each $j$
\[
\P(\cE_{\vec b^{(j-1)}}) < (1+CD^2L^{-1/4}\log L) \P(\cE_{\vec b^{(j)}}) + C\exp(-cD^2M^{-1}L^{1/2}\log L),
\]
and, since $1+x \leq \exp(x)$,
\[
 \P(\cE_{\vec b}) 
< \exp(CiD^{2}L^{-1/4}\log L)\left(   \P(\cE_{\vec g}) + C\exp(-cD^2M^{-1}L^{1/2}\log L) \right).
\]
In other words, we have (for each $i=1,\ldots, \lfloor L^{1/4}\rfloor$)
\[
\max_{\vec b\in \Theta_{iDL^{1/4}}}  \P(\cE_{\vec b}) 
< \exp(CiD^{2}L^{-1/4}\log L)\left( \min_{\vec b\in \Theta_{DL^{1/4}}}  \P(\cE_{\vec b}) + C\exp(-cD^2M^{-1}L^{1/2}\log L) \right).
\]
Therefore, by plugging this into \eqref{eq:cemlbd2} (with each $\P(\cE_{\vec b, +})$ replaced by $\P(\cE_{\vec b})$ using \eqref{eq:bbpmcp}) we have
\[
\P(\cE_+\mid \vec\cL^{\beta, \vec x} \approx \vec h) < (1+C\exp(-cD^2)) \max_{\vec b\in \Theta_{DL^{1/4}}}  \P(\cE_{\vec b}) + C\exp(-cD^2M^{-1}L^{1/2}\log L).
\]
Also, \eqref{eq:cemlbd} (with each $\P(\cE_{\vec b, -})$ replaced by $\P(\cE_{\vec b})$ using \eqref{eq:bbpmcp}) and \Cref{lem:compb} imply that
\[
\P(\cE_-\mid \vec\cL^{\beta, \vec x} \approx \vec h) > (1-C\exp(-cD^2)) \max_{\vec b\in \Theta_{DL^{1/4}}}  \P(\cE_{\vec b}) -C\exp(-cD^2M^{-1}L^{1/2}\log L).
\]
Since that (as we have seen above) \eqref{eq:dencompcen} is lower bounded by $c\exp(-CL^{1/2})$ ($\beta=1$) or $c$ ($\beta=\infty$) by \Cref{p.k-point part to whole free energy}, and the RHS of \eqref{eq:dencompcen} is in the interval \eqref{eq:deninter}, we conclude that \eqref{eq:dencompcen} equals $(1+o(1))\max_{\vec b\in \Theta_{DL^{1/4}}}  \P(\cE_{\vec b})$, which is independent of $\vec x$.
\end{proof}

It remains to prove the comparison lemma on Brownian bridges.

\begin{proof}[Proof of \Cref{lem:compb}] 
We note that the events $\cE_{\vec b}$ and $\cE_{\vec g}$ are measurable with respect to the processes $B^{\dagger, i, \vec b}$ in the interval $I_{M,L}$.
Therefore, conditional on that each $B^{\dagger, i, \vec b}$ and  $B^{\dagger, i, \vec g}$ are the same at $\pm ML^{-1/2}\log L$, the conditional probabilities of $\cE_{\vec b}$ and $\cE_{\vec g}$ would be the same.

We let $\mrm{BdyCtrl}$ be the event where for each $\dagger=\mrm{top}$, $i=1, \ldots, k-1$, and $\dagger=\mrm{bot}$, $i=2, \ldots, k$, we have
\[
|B^{\dagger, i, \vec b} (ML^{-1/2}\log L)-b^{\dagger,\mrm{R}}_i\theta^{-1}ML^{-1/2}\log L+2M\log L| < K\log L,
\]
\[
|B^{\dagger, i, \vec b} (-ML^{-1/2}\log L)-b^{\dagger,\mrm{L}}_i\theta^{-1}ML^{-1/2}\log L+2M\log L| < K\log L,
\]
\[
|B^{\dagger, i, \vec g} (ML^{-1/2}\log L)-b^{\dagger,\mrm{R}}_i\theta^{-1}ML^{-1/2}\log L+2M\log L| < K\log L,
\]
\[
|B^{\dagger, i, \vec g} (-ML^{-1/2}\log L)-b^{\dagger,\mrm{L}}_i\theta^{-1}ML^{-1/2}\log L+2M\log L| < K\log L.
\]
It follows from Gaussian tail bounds (\Cref{l.normal bounds}) (and noting that $\|\vec b-\vec g\|\le KL^{1/4}$) that
\begin{align*}
\P\left(\mrm{BdyCtrl}^c\right) \leq \exp\left(-cK^2M^{-1}L^{1/2}\log L\right).
\end{align*}
It then suffices to control the ratio of the probability densities of $B^{\dagger, i, \vec b} (\pm ML^{-1/2}\log L)$ and $B^{\dagger, i, \vec g} (\pm ML^{-1/2}\log L)$, in the interval $-ML^{1/2}\log L+[-K\log L, K\log L]$.
For this, note that $B^{\dagger, i, \vec b} (\pm ML^{-1/2}\log L)$ and $B^{\dagger, i, \vec g} (\pm ML^{-1/2}\log L)$ are Gaussian random variables with the same variance of order $ML^{-1/2}\log L$, and mean differ by $\le K\theta^{-1}ML^{-1/4}\log L$.
Therefore the density ratio is 
\[
\exp(O((K\log L)(K\theta^{-1}ML^{-1/4}\log L)(ML^{-1/2}\log L)^{-1})) = 1+O(K^2L^{-1/4}\log L).
\]
Therefore the conclusion follows.
\end{proof}

\appendix

\section{Weak convergence lemma}\label{s.abstract weak convergence}

\begin{proof}[Proof of Lemma~\ref{l.fdd convergence}]
Fix $M>0$.
We then have that
\begin{equation}  \label{e.local density hypothesis1}
    \max_{\|\vec x\|_{\infty}, \|\vec y\|_{\infty} \leq M, \|\vec x -\vec y\|_{\infty} \leq \epsilon} \left|\frac{f(\vec x)}{f(\vec y)} - 1\right| \to 1.
\end{equation}
This follows from the continuity and strict positivity of $f$ combined with the compactness of $[-M,M]^d$.
Now we take any $\delta>0$, and $\vec x, \vec y \in [-M+\delta, M-\delta]^d$. 
For any $\varepsilon$ such that $\delta\varepsilon^{-1}\in \N$, by splitting $\vec x+ [-\delta, \delta)^d$ into $(\delta\varepsilon^{-1})^d$ many translations of $[-\varepsilon, \varepsilon)^d$, applying \eqref{e.local density hypothesis} across them, and using \eqref{e.local density hypothesis1}, we have
\[
 \frac{\P\left(\vec X_\varepsilon\in \vec x + [-\delta, \delta)^d\right)}{\P\left(\vec X_\varepsilon \in \vec y + [-\delta, \delta)^d\right)} \to 
 \frac{\int_{\vec x + [-\delta, \delta)^d}f(\vec z)\diff \vec z}{\int_{\vec y + [-\delta, \delta)^d}f(\vec z)\diff \vec z},
\]
as $\epsilon \to 0$.
Then by the continuity of $f$, and that $\vec X_\epsilon\to \vec X$ in distribution, we conclude that 
\[
 \frac{\P\left(\vec X\in \vec x + [-\delta, \delta)^d\right)}{\P\left(\vec X \in \vec y + [-\delta, \delta)^d\right)} = 
 \frac{\int_{\vec x + [-\delta, \delta)^d}f(\vec z)\diff \vec z}{\int_{\vec y + [-\delta, \delta)^d}f(\vec z)\diff \vec z}.
\]
We note that by sending $M\to\infty$, this holds for arbitrary $\vec x, \vec y \in \R^d$.
Then by the integrability of $f$, we get
\[
\P\left(\vec X\in \vec x + [-\delta, \delta)^d\right) = \frac{\int_{\vec x + [-\delta, \delta)^d}f(\vec z)\diff \vec z}{\int_{\R^d}f(\vec z)\diff \vec z}.
\]
Then the conclusion follows by sending $\delta\to 0$.
\end{proof}

\section{Tent Brownian comparison and estimates}   \label{sec:tnp}

In this appendix, we provide the proofs of \Cref{lem:fh-tent-up}, \Cref{l.dp-comp-s}, and \Cref{l.control near tent center}.

\begin{proof}[Proof of \Cref{lem:fh-tent-up}]
We wish to prove \Cref{lem:fh-tent-up} from \Cref{lem:fh-tent}, by applying the tail comparison of \Cref{t.comparison}. However, the proof of \Cref{t.comparison} we give uses \Cref{lem:fh-tent-up}, and we next explain how circular arguments are avoided.

First, recall the notation $L_M^{1/2} = (L-M)^{1/2}$ from the proof of \Cref{t.comparison}. That proof used \Cref{lem:fh-tent-up} in the special case of
\begin{align*}
\P\left(\max_{x\in\{-L_M^{1/2}, L_M^{1/2}\}} \Bigl|\hfh^{\beta}_{t,1}(x) - L + 2L^{1/2}|x|\Bigr| > ML^{1/4} \midd \hfh^{\beta}_{t,1}(0) > L+\delta\right) < \exp(-cM^2),
\end{align*}
for $M = C_1\log L$ and $0<\delta<C_0L^{1/4}$, i.e., the process is considered only at the two points $\pm L_M^{1/2}$ and not on the entire interval $[-L^{1/2}, L^{1/2}]$; this was done in the proof of \Cref{t.comparison} right before \eqref{e.E complement prob bound}. This application of \Cref{lem:fh-tent-up} with this value of $M$ is the only one in the proof of Theorem~\ref{t.comparison}. We will show how to prove the above display, and then we will be allowed to use Theorem~\ref{t.comparison} to prove \Cref{lem:fh-tent-up}.

Call the event in the previous display $A_{M,L}$. Take a large $C_2>0$, then the previous display is bounded by
\begin{align}\label{e.one-point tent shape bound}
\P\left(A_{M,L}, \hfh^\beta_{t,1}(0) < L +C_2L^{1/4} \mid \hfh^\beta_{t,1}(0) > L+\delta\right) + \P\left(\hfh^\beta_{t,1}(0) > L +C_2L^{1/4} \mid \hfh^\beta_{t,1}(0) > L+\delta\right).
\end{align}
The second term is upper bounded, using a trivial upper bound and then Theorem~\ref{lem:fh-ut}, by
\begin{align}\label{e.tail comparisons for tent shape}
\frac{\P\left(\hfh^\beta_{t,1}(0) > L +C_2L^{1/4}\right)}{\P\left(\hfh^\beta_{t,1}(0) > L+\delta\right)} \leq \frac{\exp\left(-\frac{4}{3}(L+C_2L^{1/4})^{3/2} + CL^{3/4}\right)}{\exp\left(-\frac{4}{3}(L+\delta)^{3/2} - CL^{3/4}\right)} \leq \exp\left(-cL^{3/4}\right),
\end{align}
as $C_2$ is large.
The first term of \eqref{e.one-point tent shape bound} is upper bounded by
\begin{align*}
\P\left(A_{M,L}\midd \hfh^\beta_{t,1}(0) \in (L+\delta, L+C_2L^{1/4})\right) \leq \sup_{L'\in(L+\delta, L+C_2L^{1/4})} \P\left(A_{M,L}\midd \hfh^\beta_{t,1}(0) \in (L', L'+\mrm dL')\right).
\end{align*}
Now applying the \Cref{lem:fh-tent} where we condition on the value of $\hfh^\beta_{t,1}(0)$ yields that the final term is upper bounded by $\exp(-cM^{2})$ when $0< M < cL^{3/4}$. Thus the bound on \eqref{e.one-point tent shape bound} is $\exp(-cM^2) + \exp(-cL^{3/4}) \leq \exp(-cM^2)$ when $M< L^{3/8}$, as is certainly the case for $M=C_1\log L$.

This establishes the needed bound on \eqref{e.one-point tent shape bound} and so we may now make use of Theorem~\ref{t.comparison}, as noted above. We now perform the same analysis as above but for $0< M <cL^{3/4}$, and with
\begin{align*}
\tilde A_{M,L} = \left\{\sup_{|x|\leq L^{1/2}} \Bigl|\hfh^{\beta}_{t,1}(x) - L + 2L^{1/2}|x|\Bigr| > ML^{1/4}\right\}
\end{align*}
replacing $A_{M,L}$, with $M^2L^{-1/2}$ replacing $C_2L^{1/4}$, and with $0$ replacing $\delta$; the only difference is that the bound in \eqref{e.tail comparisons for tent shape} now follows from Theorem~\ref{t.comparison} instead of Theorem~\ref{lem:fh-ut} (and is now $\exp(-cM^2)$ instead of $\exp(-cL^{3/4})$). The remaining analysis is the same and yields an overall bound on the LHS of \eqref{e.one-point tent shape bound} of $\exp(-cM^2)$ for $0< M <cL^{3/4}$.
\end{proof}

\begin{proof}[Proof of \Cref{l.dp-comp-s}]
Let $\cE_0$ denote the event where $\hat{\fh}^\beta_{t,1}(0)\in (L,L+\diff L)$.
Denote by $\cE_1$ the event
\[
\hat{\fh}^\beta_{t,2}(x) < -x^2 + 0.1L, \quad \forall x\in [-\tfrac{1}{2}L^{1/2}, \tfrac{1}{2}L^{1/2}], 
\]
and by $\cE_2$ the event where
\[
\hat{\fh}^\beta_{t,1}(-L^{1/2}/2), \hat{\fh}^\beta_{t,1}(L^{1/2}/2) >  - 0.1L.
\]
Since $L>t^{-1/3-\varepsilon}$ implies that $t^{-1/3}\log L < L^{1-\varepsilon'}$ for some $\varepsilon'>0$, applying \Cref{l.bk} and absorbing the $t^{-1/3}\log L$ term yields that $\PP(\cE_1\mid \cE_0) >1-C\exp(-cL^{3/2})$.
By \Cref{lem:fh-tent} we have $\PP(\cE_2\mid \cE_0) > 1-C\exp(-cL^{3/2})$.

We now assume the event $\cE_0\cap\cE_1\cap\cE_2$, and consider the Radon-Nikodym derivative between the two sets of processes. 
Let $\cF=\Fext([-t^{2/3}L^{1/2}/2, t^{2/3}L^{1/2}/2])$ be the sigma-algebra generated by $\hat{\fh}^\beta_{t,1}$ on $(-\infty, -L^{1/2}/2]\cup\{0\}\cup [L^{1/2}/2, \infty)$ and $\hat\fh^\beta_{t,2}$.
Then according to the Gibbs property (\Cref{lem:Gibbs}),
it suffices to bound
\begin{equation}  \label{eq:expexpmo}
   Z^{-1} W(B, \fh^\beta_{t,2}),
\end{equation}
where
\begin{itemize}
    \item $\hat{B}:[-L^{1/2}/2, L^{1/2}/2]\to \RR$ is a rate $2$ Brownian bridge, conditional on $\hat{B}(-L^{1/2}/2)=\hfh^\beta_{t,1}(-L^{1/2}/2)$, $\hat{B}(0)=L$, $\hat{B}(L^{1/2}/2)=\hfh^\beta_{t,1}(L^{1/2}/2)$; and $B:[-t^{2/3}L^{1/2}/2, t^{2/3}L^{1/2}/2]\to \RR$ satisfies that $\hat{B}(x)=t^{-1/3}B(t^{2/3}x)$;
    \item $W(B, \fh^\beta_{t,2})<1$ is the weight defined through \eqref{e.rn derivative}, in the interval $[-t^{2/3}L^{1/2}/2, t^{2/3}L^{1/2}/2]$;
    \item $Z=\EE[W(B, \fh^\beta_{t,2}) \mid \cF]$ is a renormalization constant.
\end{itemize}
Assuming $\cE_0\cap\cE_1\cap\cE_2$, we have
$\PP(\hat{B}(x)\ge 0.2L, \forall x\in [-L^{1/2}/2, L^{1/2}/2]\mid \cF)>1-C\exp(-cL^2)$; and whenever $\inf_{x\in [-L^{1/2}/2, L^{1/2}/2]}\hat{B}(x)\ge 0.2L$ and also under $\cE_1$, we have 
\[
\fh^\beta_{t,2}(t^{2/3}x) - B(t^{2/3}x) > t^{1/3}\cdot 0.1 L, \quad \forall x\in [-L^{1/2}/2, L^{1/2}/2].
\]
Therefore, from \eqref{e.rn derivative} we have
\[W(B, \fh^\beta_{t,2}) > 1- 2t^{2/3}L^{1/2}\exp(-t^{1/3}\cdot 0.1L)>1-C\exp(-ct^{1/3}L).\] 
Thus we have $Z > 1-C\exp(-ct^{1/3}L)$, and \eqref{eq:expexpmo} is $1+O(\exp(-ct^{1/3}L))$ under $\cE_0\cap\cE_1\cap\cE_2$.
This combined with the estimates on $\PP(\cE_1\mid \cE_0)$ and $\PP(\cE_2\mid \cE_0)$ leads to the conclusion.
\end{proof}

We next prove \Cref{l.control near tent center}.

\begin{proof}[Proof of \Cref{l.control near tent center}]
For the first bound, by \Cref{lem:fh-tent} we have that
\[
\P\left( |\hfh^\beta_{t,1}(\pm L^{1/2}/2)| > ML^{1/4}/5 \midd \hfh^\beta_{t,1}(0)\in (L,L+\diff L)\right) < \exp(-cM^2).
\]
Then the bound follows from \Cref{l.dp-comp-s} and standard Brownian bridge estimates.

For the second bound, we use a strategy similar to the proof of \Cref{lem:fh-tent-up}. Namely, the LHS is bounded by 
\begin{align*}
\MoveEqLeft[6]
\sup_{L'\in (L,L+M\sigma_I/2)} \P\left(\sup_{x\in I}\left|\hfh^\beta_{t,1}(x) - (L - 2L^{1/2}|x|)\right| > M\sigma_I\midd \hfh^\beta_{t,1}(0)\in (L', L'+\diff L')\right)\\
&\qquad + \P\left(\hfh^\beta_{t,1}(0)> L+M\sigma_I/2\midd \hfh^\beta_{t,1}(0)> L\right).
\end{align*}
Then by the first bound and \Cref{t.comparison}, the second bound follows.
\end{proof}

\section{Tail and tent estimates for small time}  \label{sec:appc}

In this appendix we explain how to adapt the proofs in \cite{GH22} to get \Cref{lem:fh-ut}, and \Cref{lem:fh-tent}, using inputs from \cite{das2023law}, and give the proof of \Cref{l.bk}.

\subsection{Adapting arguments from \cite{GH22}}

Theorems~\ref{lem:fh-ut} and \ref{lem:fh-tent} are proven in \cite{GH22} for $t >t_0$ for any $t_0$ fixed, so here we prove it for $0 < t < t_0$ for a particular $t_0$ coming from the new inputs from \cite{das2023law}. There are two types of modifications that are required: the first, to verify that the assumptions from \cite{GH22} hold for $0 < t < t_0$, and the second in arguments estimating some partition functions. We start with the first task.

Written with our scaling, \cite[Theorem~1.4]{das2023law} asserts that, for any $\varepsilon>0$, there exist $t_0$, $c$, and $s_0$, all depending on $\varepsilon$, such that, for $0< t < t_0$ and $M > M_0 t^{-1/12}$,
\begin{align*}
\P\left(\hfh^\beta_{t,1}(0) > M + t^{-1/3}\log t^{-1}\right) \leq \exp\left(-c\frac{M^2t^{1/6}}{\sqrt{1+ Mt^{1/3-\varepsilon/2}}}\right).
\end{align*}
If we additionally assume that $M > M_0t^{-1/3-\varepsilon}$ and $t<1$, so that $M^{1/2} \geq t^{-1/6-\varepsilon/2} \geq t^{-\varepsilon/2}$, the previous display implies that
\begin{align*}
\P\left(\hfh^\beta_{t,1}(0) > 2M\right) \leq \P\left(\hfh^\beta_{t,1}(0) > M + t^{-1/3}\log t^{-1}\right) \leq \exp\left(-cM^{3/2}t^{\varepsilon/2}\right) \leq \exp(-cM).
\end{align*}
For control on the lower tail, as recorded in \Cref{lem:fh-lt}, \cite[Theorem 1.7]{das2023law} and a similar calculation shows that, if $M > t^{-1/6}$,
\begin{align}\label{e.low time lower tail}
\P\left(\hfh^\beta_{t,1}(0) < -M\right) \leq \exp(-cM^2t^{1/6}) \leq \exp(-cM);
\end{align}
and one can also instead assume $M>t^{-1/12 - \varepsilon}$ and obtain a weaker bound of the form $\exp(-cM^{\alpha})$ for some $\alpha = \alpha(\varepsilon)$.

To apply the arguments of \cite{GH22}, we also need an a priori lower bound on $\P(\hfh^\beta_{t,1}(0) > M)$. For this it is easy to check that the proof of \cite[Lemma 4.4]{GH22} applies verbatim if we assume $M > \smash{t^{-1/6}}$ (or $M > \smash{t^{-1/12 - \varepsilon}}$, as above) and use \eqref{e.low time lower tail} in place of lower tail tightness (over $t\geq t_0$) of $\hfh^\beta_{t,1}(0)$. This will then yield, for $t>0$ and $M > t^{-1/6}$,
\begin{align*}
\P\left(\hfh^\beta_{t,1}(0) > M\right) \geq \exp(-5M^{3/2}).
\end{align*}
With these estimates in place, we satisfy Assumption (iv) from \cite[Section 1.2]{GH22}. Assumptions (i), (ii)(a), and (iii) were already verified for all $t>0$ in \cite{GH22}, and Assumption (ii)(b$'$) (from \cite[Section~1.7]{GH22}) is verified in \cite{GHZ25} for . Thus, \cite[Theorems~3.1, 4.1, and 4.3]{GH22} apply as they hold only assuming these assumptions (see \cite[Proposition 1.3]{GH22}), modulo controlling the partition function arising from the Gibbs property of $\hfh^\beta_t$, and we explain this point next.

In \cite{GH22}, the estimates lower bounding the partition function are captured in Lemma 2.16, Corollary 2.17, Proposition 2.19, and Corollary 2.20. The latter two are already stated for $t>0$. Lemma 2.16 says that the partition function $Z_t$ of a single curve with respect to a lower boundary curve $p$ on an interval $[z_1,z_2]$ is lower bounded by $\exp(-2t^{2/3}e^{-t^{1/6}})\int_{z_1}^{z_2}\exp(-t^{1/3}g(u))\diff u\cdot \P(B(u) > p(u) + g(u) + t^{-1/6}, \forall u\in[z_1,z_2])$ for any non-negative function $g$, and a quick inspection of the proof shows that a minor modification yields a lower bound of $\exp(-2t^{2/3})\int_{z_1}^{z_2}\exp(-t^{1/3}g(u))\diff u)\cdot \P(B(u) > p(u) + g(u), \forall u\in[z_1,z_2])$, which is better for small $t$. Carrying this change to Corollary 2.17 and then making use of these estimates in the proofs of \cite[Theorems~3.1 and 4.1]{GH22} will then yield Theorems~\ref{lem:fh-ut} and \ref{lem:fh-tent}.

\subsection{Proving \Cref{l.bk}}

Here we give the proof of \Cref{l.bk}.

\begin{proof}[Proof of \Cref{l.bk}]
First, in the case of $\beta=1$, \cite[Theorem 1.3]{GHZ25} directly yields that there exist $C, L>0$ such that for all $L > L_0(1\vee t^{-1/6})$
\begin{align*}
\MoveEqLeft[16]
\P\bigg( \sup_{x\in [a,a+1]}\hfh^\beta_{t,2}(x)+x^2> M+Ct^{-1/3}\log(L)\midd \hfh^\beta_{t,1}(0)\in (L, L +\diff L) \bigg)\\%
&\leq \P\left(\sup_{x\in[a,a+1]} \hfh^\beta_{t,1}(x) + x^2 > M\right) + |a|t^{2/3}\exp(-cL^{2}).
\end{align*}
For the case of $\beta=\infty$, \cite[Assumption (ii)(b)]{GH22} applied to $\hfh^\beta_{t}$ (which is applicable by \cite[Theorem 2.7]{GH22}) yields that the lefthand side of the previous display is at most $\P(\sup_{x\in[a,a+1]} \hfh^\beta_{t,1}(x) + x^2 > M + Ct^{-1/3}\log L)$, which is certainly also upper bounded by the righthand side above.

Now, by stationarity the probability in the previous display is the same as when $a=0$. Then \cite[Proposition 2.10]{GH22} yields that it is upper bounded by $4M\cdot \P(\hfh^\beta_{t,1}(0) > M)$ which, by Theorem~\ref{lem:fh-ut}, is upper bounded by $\exp(-\frac{4}{3}M^{3/2} + CM^{3/4})$. This completes the proof.
\end{proof}

\bibliographystyle{alpha}
\bibliography{bibliography}

\end{document}